\documentclass{amsart}

\usepackage[utf8]{inputenc}
\usepackage[margin=1in]{geometry}
\usepackage{amsmath, amssymb, amsthm}
\usepackage{graphicx}
\usepackage{dsfont}
\usepackage{xcolor}
\usepackage{dirtytalk}
\usepackage{comment}
\usepackage{hyperref}
\usepackage{mathtools}
\usepackage{enumitem}

\usepackage{tikz}

\makeatletter
\newenvironment{customproof}[1]{%
  \par\pushQED{\qed}%
  \normalfont\topsep6pt \trivlist
  \item[\hskip\labelsep\itshape
    Proof of #1\@addpunct{.}]\ignorespaces
}{%
  \popQED\endtrivlist\@endpefalse
}
\makeatother

\usepackage{tikz}

\theoremstyle{plain}
\newtheorem{theorem}{Theorem}[section]
\newtheorem{lemma}[theorem]{Lemma}

\newtheorem{corollary}[theorem]{Corollary}

\newtheorem{preremark}[theorem]{Remark}
\newenvironment{remark}{\begin{preremark}\rm}{\medskip \end{preremark}}

\theoremstyle{definition}
\newtheorem{definition}[theorem]{Definition}

\numberwithin{equation}{section}

\newcommand{\C}{\mathcal{C}}

\newcommand{\R}{\mathbb R}
\newcommand{\LL}{\mathcal L}
\newcommand{\A}{\mathcal A}
\newcommand{\dd}{\, \mathrm{d}}
\newcommand{\eps}{\varepsilon}
\newcommand{\one}{\mathds{1}}
\DeclareMathOperator{\supp}{supp}
\newcommand{\p}{\partial}

\DeclareMathOperator*{\Tail}{{Tail}}
\newcommand{\osc}{\operatorname{osc}}

\newcommand{\kmax}{N}

\title{ $C^{1+\alpha}$ regularity for fractional $p$-harmonic functions}
\author{Davide Giovagnoli, David Jesus and Luis Silvestre}

\newcommand{\Addresses}{{
  \bigskip
  \footnotesize

    Davide Giovagnoli, \textsc{Dipartimento di Matematica,
Universit\`a di Bologna, 
 Bologna, BO 40126, Italy}\par\nopagebreak
  \textit{E-mail address:} {\tt d.giovagnoli@unibo.it}\\

  \medskip

    David Jesus, \textsc{Applied Mathematics and Computational Sciences (AMCS), King
Abdullah University of Science and Technology (KAUST),  Thuwal 23955-6900, Kingdom of Saudi Arabia}\par\nopagebreak
  \textit{E-mail address:} {\tt david.dejesus@kaust.edu.sa}\\

\medskip

  Luis Silvestre, \textsc{Mathematics Department,
 University of Chicago,  
 Chicago, IL 60637, USA}\par\nopagebreak
  \textit{E-mail address:} {\tt luis@math.uchicago.edu}\\

 }

}

\date{\today}

\begin{document}

\begin{abstract}
We establish interior $C^{1,\alpha}$ regularity estimates for some $\alpha > 0$, for solutions of the fractional $p$-Laplace equation $(-\Delta_p)^s u = 0$ when $p$ is in the range $p \in [2,2/(1-s))$.
\end{abstract}

\maketitle

\section{Introduction}

The $p$-Laplace equation is a standard equation in the Calculus of Variations. It results from considering the first variation of the $W^{1,p}$ seminorm. A classical regularity result states that $p$-harmonic functions are $C^{1,\alpha}$ regular for some $\alpha > 0$. This was proved by Uraltseva \cite{ural1968degenerate} and Uhlenbeck \cite{uhlenbeck1977regularity} originally for systems and $p \geq 2$. Several different proofs for the scalar case and generalizations have been given by various authors \cite{evans1982,dibenedetto1983c,lewis1983,tolksdorf,giaquinta1986,wang1994compactness}.

A natural nonlocal variant arises from considering the standard Gagliardo $W^{s,p}$ semi-norm for $s \in (0,1)$ and $p \in [1,\infty)$, instead of the $W^{1,p}$ seminorm.
\begin{align}\label{eq:def_gagliardo}
    [u]_{W^{s,p}(\R^d)} := \left( \iint_{\R^d \times \R^d} \frac{|u(x)-u(y)|^p}{|x-y|^{d+sp}} \dd x \dd y \right)^{1/p}.
\end{align}

Minimizers of this seminorm for functions with prescribed values outside a domain $\Omega \subset \R^d$ satisfy the fractional $p$-Laplace equation
\begin{equation} \label{eq:fracplap}
    (-\Delta_p)^s u = 0  \qquad \text{ for $x \in \Omega$}.
\end{equation}
Here, $(-\Delta_p)^s u$ stands for the standard fractional $p$-Laplacian, which is defined as
\[
(-\Delta_p)^s u (x) = \int_{\R^d} \frac{|u(x)-u(y)|^{p-2}}{|x-y|^{d+sp}} (u(x)-u(y)) \dd y.
\]

The fractional $p$-Laplacian has been the subject of numerous recent studies. We review some relevant results in Section \ref{s:history}. By analogy with the local case, it is natural to expect that solutions to the equation \eqref{eq:fracplap} may also be $C^{1,\alpha}$ for some $\alpha>0$, at least for some range of $p$ and $s$. This is a well-known conjecture in the area that has remained open for a few years. In this paper, we prove it in the range $p \in [2,2/(1-s))$.

\begin{theorem} \label{t:main}
    Let $u \in L_{sp}^{p-1}(\R^d)$, be a solution to the equation \eqref{eq:fracplap} in the ball $B_2$. Assume that $s \in (0,1)$ and $2 \leq p <2/(1-s)$. Then $u$ is $C^{1,\alpha}$ regular in $B_1$ for some $\alpha > 0$ and the following estimate holds
    \[
    \|u\|_{C^{1,\alpha}(B_1)} \leq C \left( \|u\|_{L^\infty(B_2)} +  {\Tail}_{p-1,sp}(u;2) \right).
    \]
    Here $\alpha$ and $C$ are positive constants depending only on $d$, $s$ and  $p$. The quantity $ {\Tail}_{p-1,sp}(u,2)$ accounts for the values of $u$ outside the ball $B_2$. It is defined below.
\end{theorem}

Theorem \ref{t:main} applies to weak or viscosity solutions of \eqref{eq:fracplap}. In fact, it has been proven that these two notions of solutions coincide for the equation \eqref{eq:fracplap}. In order to keep the exposition as clear as possible, we initially explain the proof of the a priori estimate in Theorem \ref{t:main} when $u$ is a classical solution with enough regularity for all the computations to make sense. Then, in Section \ref{s:weak-solutions}, we discuss the required technical modifications necessary for the proof to apply to weak solutions. The proof of the a priori estimate is quite involved on its own. For that reason, we believe that presenting the proof in terms of weak solutions from the outset would significantly hinder the paper's readability.

By $L_{sp}^{p-1}(\R^d)$ we denote the \textit{tail space} defined by
\[
L_{sp}^{p-1}(\R^d):= \left\{ u \in L_{loc}^{p-1}(\R^d) : \int_{\R^d} \frac{|u(y)|^{p-1}}{(1+|y|)^{d+sp} } \dd y < \infty \right\}.
\]
The condition $u \in L_{sp}^{p-1}(\R^d)$ is a standard requirement to make sense of $(-\Delta_p)^s u$.
We also define the \textit{tail function} as
\begin{equation} \label{def:Tail_fun}
\text{Tail}_{p-1,sp}(u;x,R):= \left( \int_{B_R^c(x)}\frac{|u(y)|^{p-1}}{|y-x|^{d+sp}} \dd y \right)^{\frac{1}{p-1}}, \quad R>0,
\end{equation}
and whenever $x=0$ we write it as $\text{Tail}_{p-1,sp}(u;R)$. 

\subsection{Previous results}

\label{s:history}

To our knowledge, the precise formulation of the fractional $p$-Laplacian that we use in this paper was first considered in 2010 by Ishii and Nakamura in \cite{ishii2010class}. In that work, they develop the viscosity solution theory, including the comparison principle and existence of solutions by Perron's method, related to this equation. Slightly earlier, Andreu, Maz\'on, Rossi and Toledo \cite{andreu2008nonlocal} considered a parabolic equation with a nonlocal version of the $p$-Laplacian similar to \eqref{eq:fracplap} but with an integrable kernel $J(x-y)$ instead of $|x-y|^{-d-sp}$.

The work by Bjorland, Caffarelli and Figalli \cite{bjorland2012non} from 2012 was motivated by the fractional $p$-Laplace equation \eqref{eq:fracplap}, as explained in its first page. The results presented on that paper, however, are about a related but different class of integro-differential equations.

After those early works, there have been many publications concerning the equation \eqref{eq:fracplap} and related models. There are results on well posedness, regularity, and bounds for eigenvalues. In the last decade, the interest in this equation has exploded. The question of $C^{1,\alpha}$ regularity of solutions is a well-known open problem in this area, which is explicitly mentioned at least in \cite{mosconi2016recent,Palatucci2018dirichlet,kuusi2018regularity,garofalo2019thoughts,kuusi2022gradient,bogelein2024announcement,bogelein2024regularity,iannizzotto2024survey,biswas2025lipschitz,Diening2025calderon,garain2025higher}. In this paper, we resolve this open problem in the range $p \in [2,2/(1-s))$.

It is impossible to review all the related results concerning the equation \eqref{eq:fracplap} that have been published in recent years, even if we restrict our attention to regularity results. We mention a few specific results below that are most closely related to the work in this paper.

In 2014, Di Castro, Kuusi and Palatucci obtained a Harnack inequality for solutions of \eqref{eq:fracplap} in \cite{di2014nonlocal}. The same authors proved the H\"older continuity of solutions, for a small H\"older exponent, in \cite{di2016local}. They use techniques inspired by the ideas of De Giorgi, Nash and Moser. In \cite{lindgren2016holder}, Lindgren also obtained H\"older estimates but using non-divergence techniques as in \cite{silvestre2006holder}. These interior H\"older estimates were extended up to the boundary by Iannizzotto, Mosconi and Squassina in \cite{Iannizzotto2016global}.

In 2017, Brasco and Lindgren obtained regularity estimates in Sobolev spaces in \cite{brasco2017higher}. A year later in \cite{brasco2018higher}, Brasco, Lindgren and Schikorra obtained H\"older regularity estimates with precise exponents. Specifically, for $p > 2$, they prove that the solution is $C^\alpha$ for every $\alpha$ strictly less than $\min(1,sp/(p-1))$. In particular, they get that the solution is almost Lipschitz in the range $p \in [2,1/(1-s)]$. H\"older estimates for precise exponents in the range $p<2$ were studied in 2024 by Garain and Lindgren in \cite{garain2024higher}.

There are two interesting survey articles that appeared in 2018 that describe problems related to the equation \eqref{eq:fracplap}. One is by Palatucci \cite{Palatucci2018dirichlet} and the other is by Kuusi, Mingione and Sire \cite{kuusi2018regularity}. Both of these articles mention the potential $C^{1,\alpha}$ regularity of solutions of \eqref{eq:fracplap} for $p$ and $s$ in some range as an open problem in the area.

In 2018, in a short paper by Schikorra, Shieh and Spector \cite{schikorra2018regularity}, they study a modified equation and show by a simple argument that the solutions are $C^{s+\alpha}$ for some $\alpha>0$. While this result does not apply to the equation \eqref{eq:fracplap}, it presented a precise regularity result for a similarly-looking equation.

Estimates in Sobolev spaces with precise exponents were obtained more recently by B\"ogelein, Duzaar, Liao and Bisci in \cite{bogelein2024gradient,bogelein2024regularity,bogelein2025gradient} and by Diening, Kim, Lee and Nowak \cite{Diening2025higher}. From these estimates, they can also derive estimates in H\"older spaces. In particular, \cite{bogelein2024regularity}, they prove that solution is almost Lipschitz for the range $p \in [2,2/(1-s))$.

The Lipschitz regularity of solutions for $p < 1/(1-s)$ was recently obtained by Biswas and Topp. The range of applicability was later extended to $p < 2/(1-s)$ by Biswas and Sen \cite{biswas2025improved}. These papers introduced the idea of applying the Ishii-Lions technique to the equation \eqref{eq:fracplap}.

We summarize the main result from \cite{biswas2025lipschitz,biswas2025improved} in our setting.
\begin{theorem} \label{t:biswas&co}
    If $u \in W^{s,p}_{loc}(B_2) \cap L^{p-1}_{sp}(\R^d)$ is a weak solution to \eqref{eq:fracplap} in $B_2$, then $u$ is locally Lipschitz continuous in $B_1$. Moreover, an estimate holds of the form
    \[ \|u\|_{Lip(B_1)} \leq C \left( \|u\|_{L^\infty(B_2)} + {\Tail}_{p-1,sp}(u;2) \right).\]
\end{theorem}

In this paper, we use Theorem \ref{t:biswas&co} as the starting point of our analysis. Moreover, we use some of the ideas from \cite{biswas2025lipschitz,biswas2025improved} in part of our proof. We also use some ideas that originated in the study of regularity properties for general integro-differential equations like \cite{silvestre2006holder,Kassman-Schwab,imbert2019weak}. In the next section, we outline the ideas of our proof.

\subsection{Strategy of the proof} \label{s:strategy}

The starting point of our analysis comes after the Lipschitz continuity results from \cite{biswas2025lipschitz,biswas2025improved}. From these, we know that our solution $u$ is Lipschitz continuous in any ball $B_r$ with $r<2$ with an estimate of its Lipschitz norm. Consequently, we know that the gradient $\nabla u$ is bounded in $B_1$, hence up to a rescaling, we suppose $\|\nabla u\|_{L^\infty(B_1)}\leq 1$.

The \textbf{first stage} in the proof is to establish an improvement on our estimate for the gradient $\nabla u$ in dyadic balls. That is, we show iteratively that $\|\nabla u\|_{L^\infty(B_{2^{-k}})} \leq C_0 (1-\delta)^k$ for $k = 0,1,2,\dots,\kmax$. Here $\delta >0$ is a small parameter depending on $d$, $s$, $p$. The value of $\kmax$ may be finite or infinite. It will be infinite in those cases that $\nabla u(0) = 0$. When $\kmax < \infty$, we will see that our iteration only stops when $\nabla u$ is very close to a fixed direction $e \in S^{d-1}$ in most of the ball $B_{2^{-\kmax}}$. When we reach that condition, the iteration in the first stage of the proof stops and we move on to the second stage. The proofs of the first stage are the most delicate in this paper. We achieve them analyzing the linearized equation applied to the directional derivatives. The linearized kernel may be highly degenerate. We need to find those directions where the kernel is nondegenerate and they correspond to obtaining a growth condition for $u$. The technique to leverage the interplay between the average values of the gradient of $u$ in certain directions and the non-degeneracy of the kernel is a technical novelty introduced in this paper.

We give a rough description of this technique. Assuming we have a solution of \eqref{eq:fracplap} such that $|\nabla u|\leq 1$ in $B_1$, we start from a point $x\in B_{1/2}$  such that $\nabla u(x)$ is close to some unit vector $e$. If this point does not exist, then the improvement of flatness readily follows. We want to show that either $\nabla u$ is close to $e$ in most of $B_1$, or the directional derivative $e \cdot \nabla u$ is less than some quantifiable quantity $(1-\delta)$ in all of $B_{1/2}$. We consider a line segment $[x,x+\nu]$ with direction $\nu\in S^{d-1}$. We can see $\nabla u$ as a map from $B_1$ to itself. Furthermore, $\{\nabla u(x+t\nu)\,:\, t\in[0,1]\}$ corresponds to a curve starting from a point close to $e$. Unless $\nabla u$ stays close to $e$ in most of $B_1$, there must be some directions $\nu$ for which the curve will eventually leave far from $e$. The main idea of the proof is to consider the first time $t_\nu$ when we \emph{feel} the set $\{e \cdot \nabla u < 1-r\}$. The minimality of $t_\nu$ implies that for $t\leq t_\nu$ the difference $u(x+t\nu)-u(x)$ is comparable with $t \nu \cdot e$. We deduce the non-degeneracy of the kernel $K_u(x,y)$ along this segment. Then, we show how to carry out an iteration in the style of De Giorgi to get the results that lead to the first stage of the proof. The details are given in Section \ref{s:reduc_grad}.

At the end of the first stage, in most of $B_{2^{-\kmax}}$, the gradient $\nabla u$ is very close to a vector $e$ with norm $C_0 (1-\delta)^{\kmax}$. The \textbf{second stage} consists in showing that in $B_{2^{-\kmax-1}}$ the gradient $\nabla u$ is close to $e$ at every point, not just in most of the points. This is done by a different argument, using the Ishii-Lions method applied to the function $u(x)-e \cdot x$, borrowing some ideas from \cite{biswas2025lipschitz,biswas2025improved}.

In the \textbf{third stage} of the proof, we establish the H\"older continuity of $\nabla u$ in $B_{2^{-\kmax-2}}$ by analyzing again the linearized equation applied to the directional derivatives. Since at this point we know that the gradient $\nabla u$ is everywhere close to a fixed nonzero vector $e$, we work in the realm of mild uniform ellipticity as in \cite{Kassman-Schwab}.

The value of $\kmax$ where the iteration of the first stage transitions into the second stage depends on the point $x$. In any case, the three stages fit together harmoniously providing a $C^{\alpha}$ estimate for $\nabla u$ at the origin that does not depend on the value of $\kmax$. By a standard translation argument, we can extend this estimate to all of $B_1$, or to any ball $B_r$ with $r<2$ adjusting the parameters of the inequality appropriately.

We point out that the techniques we use in the first stage of our proof are purely nonlocal. Therefore, the constants in the estimates we obtain with this method do not remain bounded in the limit $s \to 1$.

The analysis in this paper is restricted to the superquadratic case $p \geq 2$. The condition $p < 2$ would affect several parts of the proof, especially on the first stage. The nature of the equation is rather different. For $p>2$ the concern is that the linearized kernel (which we call $K_u$ and define in \eqref{eq:kernel}) may vanish at some pairs of points. When $p<2$ the problem is that $K_u$ may equal $+\infty$ at some places. It is also natural to expect $C^{1,\alpha}$ regularity for the equation \eqref{eq:fracplap} for some range of $p < 2$ and $s$, but that remains open for now. 

The restriction $p < 2/(1-s)$ coincides with the range where almost-Lipschitz regularity is obtained in \cite{bogelein2024regularity} and Lipschitz regularity is obtained in \cite{biswas2025improved}. Based on these results, the solutions of \eqref{eq:fracplap} are not expected to be Lipschitz for $p > 2/(1-s)$. However, no examples are known to us with a zero right-hand side.

\section{Preliminaries}

We write $\LL_u$ for the linearization of $(-\Delta_p)^s$ at $u$. It is given by the following formula
\begin{equation} \label{eq:linearized}
\begin{aligned}
    \LL_u v(x) = (p-1) \int_{\R^d} \frac{|u(x)-u(y)|^{p-2}}{|x-y|^{d+sp}} (v(x)-v(y))\dd y
    = \int_{\R^d} (v(x)-v(y))\, K_u(x,y) \dd y.
\end{aligned}
\end{equation}
Here, we define the kernel $K_u$ by
\begin{equation} \label{eq:kernel}
    K_u(x,y) = (p-1) \frac{|u(x)-u(y)|^{p-2}}{|x-y|^{d+sp}}.
\end{equation}

It is convenient to keep in mind that when $u$ is a differentiable function, the kernel $K_u$ is of homogeneity $-d-2+p(1-s)$. This makes the linearized operator $\LL_u$ a nonlocal operator of order $2-p(1-s)$. Our assumption in Theorem \ref{t:main} that $p < 2/(1-s)$ is natural, as it ensures that $K_u$ is singular at the origin and the operator is generically elliptic of a positive fractional order. If $p \geq 2/(1-s)$, the linearized operator $\LL_u$ would be of order zero leading to no apparent regularity estimates on derivatives of $u$.

Note also that when $p < 1/(1-s)$, the linearized operator $\LL_u$ is generically of order greater than one. In this case, we expect similar regularity as in Theorem \ref{t:main} even if there was a non-zero bounded right-hand side.

\subsection{Viscosity and weak solutions}~~ \label{subsection_Def}

We recall below the notion of viscosity solution of \eqref{eq:fracplap} in $\Omega \subset \R^d$.
\begin{definition}\label{DefViscSol}
An upper (lower, resp.) semicontinuous function $u \in L_{sp}^{p-1}(\R^d)$ is called a viscosity subsolution (supersolution, resp.) of \eqref{eq:fracplap} in $\Omega$ if whenever there is a function $\phi\in C^2(B_r(x_0))$ for some $B_r(x_0) \subset \Omega$, with $\phi$ touching $u$ from above (from below, resp.) at $x_0$ in $B_r(x_0)$, we have $(-\Delta_p)^s \phi_r (x_0)\leq 0 \; (\geq 0\text{, resp.})$ where
\[
\phi_r(x)=\left\{\begin{array}{ll}
\phi(x) & \text{for}\; x\in B_r(x_0),
\\[2mm]
u(x) & \text{otherwise}.
\end{array}
\right.
\]
Furthermore, a viscosity solution of \eqref{eq:fracplap} in $\Omega$ is both a viscosity sub and supersolution in $\Omega$. 
\end{definition}
In order to recall the notion of weak solution, we define the following space
\[
W^{s, p}(\Omega)=\left\{ v\in L^p(\Omega)\; :\;\int_\Omega\int_\Omega \frac{|v(x)-v(y)|^p}{|x-y|^{d+sp}}\dd x\dd y<\infty \right\},
\]
and the corresponding local fractional Sobolev space, which is
\[
W^{s, p}_{loc}(\Omega)=\{ v\in L^p_{loc}(\Omega)\; :\;  v\in W^{s, p}(\widetilde\Omega)\; \text{for any}\; \widetilde\Omega\Subset \Omega \}.
\]

\begin{definition}\label{DefWeakSol}
A function $u\in W^{s, p}_{ loc}(\Omega)\cap L^{p-1}_{sp}(\R^d)$ is said to be a weak subsolution (supersolution, resp.) to \eqref{eq:fracplap} in $\Omega$ if we have
$$ \int_{\R^d}\int_{\R^d} |u(x)-u(y)|^{p-2}(u(x)-u(y))(\phi(x)-\phi(y))|x-y|^{-d-sp}\dd x\dd y\leq 0 \;( \geq 0\text{, resp.}),$$
for all nonnegative $\phi\in C^\infty_0(\Omega)$.
Moreover, $u$ is said to be a weak solution in $\Omega$ if it is both a weak sub and supersolution in $\Omega$. 
\end{definition}

The equivalence between weak solutions and viscosity solutions is established in \cite{korvenpaa2019equivalence}. 

For $\Omega \subset \R^d$ and $\gamma >0$, define $H^{\gamma}(\Omega)$ to be the space of functions $v \in L^2({\Omega})$ for which 
\[
\| v\|_{H^{\gamma}(\Omega)} := \| v\|_{L^2({\Omega})} +  [ v ] _{W^{\gamma,2}(\Omega)} <\infty,
\]
 where the Gagliardo's seminorm is defined in \eqref{eq:def_gagliardo}. 
 We use the following definitions for $H^\gamma_0(\Omega)$ and $H^{-\gamma}(\Omega)$.
\begin{align*}
    H^\gamma_0(\Omega) &:= \{ v \in H^\gamma(\R^d): v=0 \text{ outside } \Omega\}, \\
    H^{-\gamma}(\Omega)& := H^\gamma_0(\Omega)^\ast.
\end{align*}

We shall now state the following elementary estimate for the function $J_p(t)=|t|^{p-2}t$.

\begin{lemma}\label{l:appendix}
Suppose $p>1$ and $a, b\in \R$. Then 
\begin{align*}
&J_p(a) -J_p(b)= (p-1) \int_{0}^{1} |b+t(a-b)|^{p-2} (a-b)  \dd  t, \\
\intertext{ and}
&\frac{1}{C_p} (|b|+|a-b|)^{p-2}\leq \int_0^1 |b+t(a-b)|^{p-2} \dd t\leq C_p (|b|+|a-b|)^{p-2},
\end{align*}
for some constant $C_p$, depending only on $p$. 
\end{lemma}

\subsection{Tail functions}\label{subsection_tail}

We define the following quantities for $q,\alpha >0$ that measure the effect of the function $u$ on the tail of integro-differential operators. First we define the Tail function
    \[
    \text{Tail}_{q,\alpha}(u;x,R):= \left(\int_{B_R^c(x)} \frac{|u(y)|^{q}}{|x-y|^{d+\alpha}} \, \dd y \right)^{\frac{1}{q}}.
    \]
    Whenever $x=0$ we write it as $\text{Tail}_{q,\alpha}(u;R)$.
    Note that $\text{Tail}_{q,\alpha}(u;R)$ is nonincreasing as a function of $R$.
By $L_{\alpha}^{q}(\R^d)$ we denote the $(q,\alpha)$--\textit{tail space} defined by
\[
L_{\alpha}^{q}(\R^d):= \left\{ u \in L_{loc}^{q}(\R^d) : \int_{\R^d} \frac{|u(y)|^{q}}{(1+|y|)^{d+\alpha}  }\dd y < \infty \right\}.
\]

We collect here a few useful inequalities for tail functions that will be repeatedly used in what follows. In particular, the quantities ${\Tail}_{p-2,sp}(u,1)$ and ${\Tail}_{p-1,sp+1}(u,1)$ naturally arise in connection with the operator $\LL_u$ introduced in \eqref{eq:linearized} and the differentiated equation (see Lemma \ref{l:eqforv_e_scaled}).

\begin{lemma} Let $p \in [2,2/(1-s))$  and assume $u \in L_{sp}^{p-1}(\R^d)$ and $R>0$. The following inequalities hold for a universal $C>0$
        \begin{equation} \label{Tail_functions}
        \begin{aligned}
          \text{Tail}_{p-2,sp}(u,R)^{p-2} \leq C R^{\frac{p(1-s)-2}{p-1}} \text{Tail}_{p-1,sp+1}(u,R)^{p-2}, \\
              \text{Tail}_{p-1,sp+1}(u,R)^{p-1} \leq R^{-1}\text{Tail}_{p-1,sp}(u,R)^{p-1}, \\
              \text{Tail}_{p-2,sp}(u,R)^{p-2} \leq C R^{-\frac{sp}{p-1}}\text{Tail}_{p-1,sp}(u,R)^{p-2}.
        \end{aligned}
        \end{equation}
    \end{lemma}
    \begin{proof}
  The first inequality follows directly using H\"older inequality. The second inequality trivially holds. The third inequality in \eqref{Tail_functions} can be obtained by combining the previous estimates.
    \end{proof}

The following result, contained in \cite[Lemma 2.2--Lemma 2.3]{brasco2018higher}, will be useful many times to compare tail functions with different centers. We restate it conveniently for our purposes.
\begin{lemma} \label{l:move_center_tail}
    Let $\alpha,q >0$, and $0<r<R$ and assume $u \in L_{\alpha}^{q}(\R^d)$. For $x_0 \in \R^d$ and $x \in B_{r}(x_0)$, we have
\begin{align*}
    \int_{B_R^c(x_0)} \frac{|u(y)|^{q}}{|x-y|^{d+\alpha}} \dd y \leq  \left( \frac{R}{R-r} \right)^{d+\alpha} {\Tail}_{q,\alpha}(u;x_0,R)^{q}.
\end{align*}
Furthermore, for $x_1 \in \R^d$ such that $B_{r}(x_0) \subset B_{R}(x_1)$, if in addition $ u \in L^{\infty}(B_{R}(x_1)) $ then 
\begin{align*}
    {\Tail}_{q,\alpha}(u;x_0,r)^q &\leq \left( \frac{R}{R-|x_1-x_0|}\right)^{d+\alpha}  {\Tail}_{q,\alpha}(u;x_1,R)^q  +  C \frac{R^d}{r^{d+\alpha}}\|u \|_{L^{\infty}(B_R(x_1))}^{q}.
\end{align*}
Here $C$ depends only on $\alpha$, $q$, $d$, $s$ and $p$.
\end{lemma}
\subsection{Scaling} \label{subsec_Scaling}

Given a solution $u$, we may scale it in the following way. For any $r>0$, we define the rescaled function
\begin{align*}
    u_r(x) &= r^{-1} u(rx).
\end{align*}
We now record some basic properties of $u_r$, which follow by direct computation:
\begin{itemize}
\item $\nabla u_r(x) = \nabla u(rx)$.
\item If $|\nabla u| \leq C_0$ in $B_1$, then $|\nabla u_r| \leq C_0 $ in $B_{r^{-1}}$.
\item For any $\alpha,q >0$ and any $\rho>0$, we have
\begin{equation} \label{scale_tails}
    {\Tail}_{q,\alpha}(u_r;\rho)^q = r^{\alpha-q}{\Tail}_{q,\alpha}(u; r \rho)^q. 
\end{equation}
\end{itemize}

In particular, we need \eqref{scale_tails} in the case $q = p-1$ and $\alpha = sp+1$.  
Notice that in this situation we have $2-p(1-s)> 0$. Hence for small $r$, which corresponds to zooming in, one has $r^{2-p(1-s)} < 1$ and therefore the Tail becomes smaller. It is standard to say that the tail is a subcritical quantity with respect to the scaling of the equation. This is important for our proof since we apply several estimates in a sequence of scales that involve zooming in the solution. The computation above helps us keep the effect of the tails negligible in small scales.

\section{The linearized equation satisfied by directional derivatives} \label{s:linearized}

The first and third stages of our proof, described in Section \ref{s:strategy}, involve analyzing the equation satisfied by directional derivatives. In practice, we want to work with a localized version of them. To that end, we construct the functions $v_e$ which coincide with $e \cdot \nabla u$ in some ball, and equal zero outside of the double ball. In this section, we describe these functions $v_e$ and compute the equation in terms of the linearized operator $\LL_u v_e$.

We start with a technical discussion on the linearized operator $\LL_u$.

\subsection{The linearized operator}

Recall that $\LL_u$ is the linearized operator with kernel $K_u$ defined in \eqref{eq:linearized}. We want to make sure that $\LL_u v$ is well-defined in some sense, at least when $u$ and $v$ are smooth enough.

The linearized operator $\LL_u$ only appears in our proofs in Sections \ref{s:reduc_grad} and \ref{s:stage3} through quadratic expressions of the form
\begin{equation} \label{eq:quadratic-form}
\int_B \LL_u v(x) \varphi(x) \dd x.
\end{equation}
This expression makes sense under rather mild regularity conditions on $u$, $v$ and $\varphi$. We will see below that \eqref{eq:quadratic-form} is well-defined if $u$ is Lipschitz, $v \in H^{1-p(1-s)/2}(B)$ and $\varphi \in H^{1-p(1-s)/2}_0(B)$, in addition to $u,v \in L^{p-1}_{sp}(\R^d)$. There are many other combinations of spaces that would make \eqref{eq:quadratic-form} well-defined, for example, if the three functions belong to $W^{s,p}(B) \cap L^{p-1}_{sp}(\R^d)$.

As we mentioned in the introduction, we initially write the proof of Theorem \ref{t:main}, and all the lemmas leading to it, assuming that our solution $u$ is sufficiently smooth so that all the expressions make sense classically. In Section \ref{s:weak-solutions}, we explain the extra technical work necessary to apply our result to weak solutions. In the case of the first stage of the proof, this is done through an approximate equation that has smooth solutions.

The case $p > 1/(1-s)$ is the simpler one. For this range, Lipschitz regularity of $u$ and $v$ suffice to make the integral in \eqref{eq:linearized} absolutely convergent and $\LL_u v(x)$ well-defined pointwise.

\begin{lemma} \label{l:LL-well-defined-Lipschitz}
Let $p > 1/(1-s)$ and $p\geq 2$. Assume that the functions $u$ and $v$ are both in $L^{p-1}_{sp}(\R^d)$ and Lipschitz at $x$. Then the integral in \eqref{eq:linearized} is well-defined.
\end{lemma}

\begin{proof}
Let us start by analyzing the tails of the integral
\begin{align*}
\int_{B_1^c(x)} &\left| \frac{|u(x)-u(y)|^{p-2}}{|x-y|^{d+sp}} (v(x)-v(y)) \right| \dd y 
\leq C \int_{B_1^c(x)}  \frac{|u(x)|^{p-2} + |u(y)|^{p-2}}{|x-y|^{d+sp}} \left( |v(x)|+|v(y)| \right) \dd y \\
 &\qquad \leq C \left( |u(x)|^{p-2} + {\Tail}_{p-1,sp}(u;x,1)^{p-2} \right) \left( |v(x)| + {\Tail}_{p-1,sp}(v;x,1) \right).
\end{align*}
We used H\"older inequality in the last estimate.

We now restrict our attention to the integration domain $y \in B_1(x)$. Here, we use the Lipschitz continuity of $u$ and $v$ at $x$, to get
\begin{align*}
\int_{B_1(x)} \left| \frac{|u(x)-u(y)|^{p-2}}{|x-y|^{d+sp}} (v(x)-v(y)) \right| \dd y &\leq [u]_{Lip(x)}^{p-2} [v]_{Lip(x)} \int_{B_1(x)}  |x-y|^{p-1-d-sp} \dd y  < +\infty.
\end{align*}
The last integral is finite because $p > 1/(1-s)$.
\end{proof}

The case $p \in [2,1/(1-s)]$ is more complicated. In this case, even further regularity for both $u$ and $v$ is not enough for the integral in \eqref{eq:linearized} to be well-defined even in the principal-value sense. We will define $\LL_u v$ only as a distribution in certain negative-order Sobolev space. If $u$ and $v$ are sufficiently smooth, the pointwise value of $\LL_u v(x)$ is computable through \eqref{eq:linearized} at points $x$ where $\nabla u(x) \neq 0$. We discuss these subtle issues in Appendix \ref{s:appendix_2}, even though it is not strictly necessary for our proofs in this paper.

When $2 < p < 1/(1-s)$, we make sense of $\LL_u v$ as a distribution in $H^{-\gamma}(B_1)$, for some $\gamma>0$. For the proofs in this paper, we may take $\gamma = 1-p(1-s)/2$.

\begin{lemma} \label{l:LL-in-Walpha}
Let $p \geq 2$. Assume that $u$ is in $L^{p-1}_{sp}(\R^d)$ and is Lipschitz in $B_1$. Let $v$ be a function in $L^{p-1}_{sp}(\R^d)$ that is also in $H^\alpha(B_1)$ for some $\alpha \in (\max(0,1-p(1-s)),1)$. Then $\LL_u v(x)$ is in $H^{-\beta}(B_r)$ for all $r < 1$ where $\beta = 2-p(1-s)-\alpha$.

In particular $\alpha=\beta$ if $\alpha = 1 - p(1-s)/2$.
\end{lemma}

\begin{proof}
For any $\varphi \in H^\beta_0(B_r)$, we analyze the bilinear form
\begin{align*}
&\left| \int_{B_1} \LL_u v(x) \, \varphi(x) \dd x \right| = \bigg \vert \frac 12 \iint_{B_1 \times B_1} \frac{|u(x)-u(y)|^{p-2} (v(x)-v(y)) (\varphi(x)-\varphi(y)}{|x-y|^{d+sp}} \dd x \dd y \\
&\qquad + \int_{B_1} \int_{\R^d \setminus B_1} \frac{|u(x)-u(y)|^{p-2} (v(x)-v(y))}{|x-y|^{d+sp}} \, \varphi(x) \dd y \dd x \bigg \vert \\
&\leq \frac 12 [u]_{Lip(B_1)}^{p-2} \left|\iint_{B_1 \times B_1} \frac{(v(x)-v(y)) (\varphi(x)-\varphi(y))}{|x-y|^{d-p(1-s)+2}} \dd x \dd y \right| \\
&\qquad + C \int_{B_1} \int_{\R^d \setminus B_1} \frac{(|u(x)|^{p-2}+|u(y)|^{p-2}) (|v(x)|+|v(y)|)}{|x-y|^{d+sp}} \, |\varphi(x)| \dd y \dd x \\
&\leq  \frac 12 [u]_{Lip(B_1)}^{p-2} [v]_{W^{\alpha,2}(B_1)} [\varphi]_{W^{\beta,2}(B_1)} \\ 
&\qquad + \frac{C}{(1-r)^{d+sp}} \left( \|u\|_{L^{\infty}(B_1)}^{p-2} + {\Tail}_{p-1,sp}(u,1)^{p-2}\right) \left( \|v\|_{L^2(B_1)}+ {\Tail}_{p-1,sp}(v,1)\right) \|\varphi \|_{L^{2}(B_1)} \\
&\leq C(u,v) \|\varphi\|_{H^\beta(B_1)},
\end{align*}
where we used  H\"older's inequality and $|x-y| \geq (1-r)|y|$ for $x \in \supp \varphi$ and $y \in B_1^c$.

Since the expression $\int \LL_u v \, \varphi \dd x$ is well-defined and bounded for $\varphi \in H^\beta_0(B_r)$, it means that $\LL_u v$ is well-defined in the space $H^{-\beta}(B_r)$.
\end{proof}

From this point on, until the beginning of Section \ref{s:weak-solutions}, we assume that the solution $u$ to \eqref{eq:fracplap} is sufficiently regular so that all our computations are justified. We use the linearized operator $\LL_u$ in our proofs in Sections \ref{s:reduc_grad} and \ref{s:stage3}, but it is only applied through its associated bilinear form. Our proofs in Section \ref{s:reduc_grad} are written under the assumption that $u \in C^1 \cap H^{1+\gamma}$ for $\gamma = 1-p(1-s)/2$. The estimates we obtain are independent of the norm of $u$ in these spaces. It is simply a qualitative condition for every one of our expressions in the proof to be well-defined. Later on, in Section \ref{s:weak-solutions} we discuss how to extend the proof to all weak solutions using an approximate equation that has smooth solutions.

It is interesting to point out that a result like Lemmas \ref{l:LL-well-defined-Lipschitz} and \ref{l:LL-in-Walpha} cannot hold in the subquadratic case $p < 2$. Indeed, if $u \equiv 0$, then the kernel $K_u(x,y) \equiv +\infty$ when $p < 2$, and the linearized operator $\LL_u$ would make no sense.

\subsection{Localizing the directional derivatives} 

Let $\eta : \R^d \to [0,1]$ be a smooth radially-symmetric cutoff function such that $\eta(x) = 1$ for $|x| \leq 3/2$ and $\eta(x) = 0$ for $|x| \geq 7/4$.
For any unit vector $e \in S^{d-1}$ and some $R>0$, we define the function
\begin{equation} \label{def_ve}
     v_e(x) = \eta(x/R) \, (e \cdot \nabla u).
\end{equation} 
The function $v_e$ coincides with the directional derivatives $(e \cdot \nabla u)$ in $ B_R$. Moreover, $v_e$ is globally bounded and has compact support contained in $B_{2R}$. 
Note that the function $v_e$ implicitly depends on the radius $R$ that we omit in our notation.

In the next lemma, we compute the linearization of \eqref{eq:fracplap} applied to $v_e$. This lemma will be used both in the first and third stages of our proof of Theorem \ref{t:main}.

\begin{lemma} \label{l:eqforv_e_scaled}
   Let $R>0$ and consider $u \in L_{sp}^{p-1}(\R^d)$ to be a smooth solution of \eqref{eq:fracplap} in $B_{2R}$. The function $v_e$ defined in \eqref{def_ve} satisfies the following equation 
\begin{align*}
    |\LL_{ u} v_e(x)| \leq  C \left(R^{-1-sp} \| u\|^{p-1}_{L^{\infty}(B_{R})} +   \,{\Tail}_{p-1,sp+1}( u,R)^{p-1}\right), \; \text{ for } x\in B_{R}.
    \end{align*}
    Here, the constant $C$ depends only on $d$, $s$, $p$.
\end{lemma}

\begin{proof}
    Let us first use $\eta_R(y) = \eta(y/R)$ to split the domain of integration in the expression for $(-\Delta_p)^s u(x)$. We write $(-\Delta_p)^s u(x) = F_1(x) + F_2(x)$, where
    \begin{align*}
        F_1(x) &= \int_{\R^d} \frac{| u(x)- u(y)|^{p-2}}{|x-y|^{d+sp}} (u(x)-u(y)) \eta_R(y) \dd y \\
        F_2(x) &= \int_{\R^d} \frac{| u(x)- u(y)|^{p-2}}{|x-y|^{d+sp}} (u(x)-u(y)) (1-\eta_R(y)) \dd y
    \end{align*}
    From \eqref{eq:fracplap}, we have that $F_1(x)+F_2(x) = 0$ for all $x \in B_{2R}$. We differentiate this equation to get $e \cdot \nabla F_1(x)+e \cdot \nabla F_2(x) = 0$ and analyze both terms.

    To differentiate $F_1$, we use the change of variables $y = x+h$. Note that since $x \in B_R$ we have $\eta_R(x) = 1$
    \begin{align*}
        e \cdot \nabla F_1(x) &= e \cdot \nabla \int_{\R^d} \frac{| u(x)- u(x+h)|^{p-2}}{|h|^{d+sp}} (u(x)-u(x+h)) \eta_R(x+h) \dd h \\
        &= (p-1) \int_{\R^d} \frac{| u(x)- u(x+h)|^{p-2}}{|h|^{d+sp}} (e\cdot \nabla u(x)-e \cdot \nabla u(x+h)) \eta_R(x+h) \dd h \\
        & \qquad + \int_{\R^d} \frac{| u(x)- u(x+h)|^{p-2}}{|h|^{d+sp}} (u(x)-u(x+h)) (e \cdot \nabla \eta_R(x+h)) \dd h\\
        &= \LL_u v_e(x) + (p-1) \int_{\R^d} \frac{| u(x)- u(y)|^{p-2}}{|x-y|^{d+sp}} (e\cdot \nabla u(x)) (\eta_R(y)-1) \dd y\\
        & \qquad + \int_{\R^d} \frac{| u(x)- u(y)|^{p-2}}{|x-y|^{d+sp}} (u(x)-u(y)) (e \cdot \nabla \eta_R(y)) \dd y.
    \end{align*}

    We also differentiate $F_2(x)$.
    \begin{align*}
    e \cdot F_2(x) &= e\cdot \nabla \int_{\R^d} \frac{| u(x)- u(y)|^{p-2}}{|x-y|^{d+sp}} (u(x)-u(y)) (1-\eta_R(y)) \dd y \\
    &= (p-1) \int_{\R^d} \frac{| u(x)- u(y)|^{p-2}}{|x-y|^{d+sp}} (e \cdot \nabla u(x)) (1-\eta_R(y)) \dd y \\
    & \qquad + \int_{\R^d} | u(x)- u(y)|^{p-2} (u(x)-u(y)) (1-\eta_R(y)) e\cdot \nabla_x \left( \frac 1 {|x-y|^{d+sp}} \right)\dd y.
    \end{align*}

    Note that the factor $1-\eta_R(y)$ equals zero when $|x-y|$ is small. Therefore, there is no singularity in the last integrand.

    We add the expressions of $e \cdot \nabla F_1$ and $e \cdot \nabla F_2$ for $x \in B_R$. We notice some terms cancel out, and we obtain
    \begin{align*}
    0 &= \LL_u v_e(x) - \int_{\R^d} | u(x)- u(y)|^{p-2} (u(x)-u(y)) \, e \cdot \nabla_y \left( \frac{1-\eta_R(y)}{|x-y|^{d+sp}} \right) \dd y .
    \end{align*}
    Recall that $\eta(y/R) \equiv 1$ if $|y| < 3R/2$. Therefore $v_e$ satisfies
    \begin{equation*}
    \begin{aligned}
        |\LL_{u} v_e(x)|\,&\leq   C \int_{B_{3R/2}^c} \frac{|u(y)- u(x)|^{p-1}}{|x-y|^{d+sp+1}} \dd y \\
         &\leq C R^{-1-sp} |u(x)|^{p-1} + C  \,{\Tail}_{p-1,sp+1}(u,3R/2)^{p-1} \\
         &\leq C R^{-1-sp} \|u\|_{L^\infty(B_R)}^{p-1} + C  \,{\Tail}_{p-1,sp+1}(u,R)^{p-1}.
    \end{aligned}
    \end{equation*}    
\end{proof}

Fastidious readers may wonder if the manipulations in the proof of Lemma \ref{l:ve-eq-small-rhs} are well justified when $\LL_u v_e$ is only defined as a distribution in $H^{-1+p(1-s)/2}$. This is indeed the case. For example, one could modify the definition of $\LL_u v_e$ to integrate only in the complement of a small ball $B_\rho^c(x)$, differentiate this approximate expression, and then pass to the limit as $\rho \to 0$. There is no difficulty in this procedure.

\section{Reducing the norm of the gradient}\label{s:reduc_grad}

In this section we prove the main lemma that drives the iteration of the first stage of the proof described in Section \ref{s:strategy}. It is a type of improvement of oscillation lemma, applied to the directional derivatives of $u$ through the functions $v_e$ defined in the previous section. We follow a modified De Giorgi iteration. While some ideas can be compared with those in \cite{chihinchan}, the key idea to estimate the effect of sets of positive measure in this degenerate context is completely new. Note also that our De Giorgi iteration is not standard. Instead of the $L^2$-norm of truncations of the function $v_e$, we will get a recurrence relation on the measure of its level sets. More importantly, there is no coercitivy estimate anywhere in the proofs of this section. Instead, we use a purely nonlocal technique that is new (see Lemmas \ref{l:Bound_Jgap1} and \ref{l:bound_Jr}).

The main lemma of this section is the following.

\begin{lemma} \label{l:iteration}
    For any $r,\mu >0$ small, there exist $\delta, \eps_1>0$ (small) and $K_0>0$ (large) depending on $r,\,\mu, \, d,\, s,\, p$ such that if the following conditions hold for some unit vector $e \in S^{d-1}$,
    \begin{enumerate}
    \item $u$ satisfies \eqref{eq:fracplap} in $B_{2^{K_0+1}}$,
    \item $u(0) = 0$,
    \item  $\|\nabla u\|_{L^{\infty}(B_{2^{n}})} \leq (1-\delta)^{-n} ,$ for $n=0,\dots,K_0$, 
    \item $\Tail_{p-1,sp+1}(u,2^{K_0}) \leq \eps_1$,
    \item $|\{ x \in B_1 : |\nabla u(x) - e| \geq r \}| \geq \mu > 0$,
    \end{enumerate}
    then $e \cdot \nabla u \leq (1-\delta)$ in $B_{1/2}$.
\end{lemma}

We emphasize that if $\delta$, $\eps_1$ and $K_0$ are given values of these  parameters for which Lemma \ref{l:iteration} applies, the statement of the lemma remains true for any smaller values for $\delta$, $\eps_1$ and any larger value for $K_0$.

This whole section is devoted to proving Lemma \ref{l:iteration}. The proof requires several auxiliary lemmas to build up the ideas and finally derive the proof of Lemma \ref{l:iteration} at the end of the section.

To prove Lemma \ref{l:iteration} we carry out an iteration in the style of De Giorgi. We face the difficulty that the kernel $K_u(x,y)$ is degenerate wherever $|u(x)-u(y)| \ll |x-y|$. We must find  pairs of points $(x,y)$ where this does not happen based on the information we have on the directional derivative $e \cdot \nabla u$.

We explain a new key idea, put forward in Lemma \ref{l:Bound_Jgap1}. We want to prove that, if
\[
|\A_{r}| := |\{ x \in B_1 : |\nabla u(x) - e| \geq r \}|\geq \mu,
\]
then $e\cdot \nabla u\leq 1-\delta$ in $B_{1/2}$. Since we started from the assumption that $|\nabla u|\leq 1$ in $B_1$, we can see $\nabla u$ as a map from $B_1$ to itself. Take a point $x\notin \A_r$ and split $\A_r$ into radial pieces starting from $x$:
\[
\A_{r,\nu}(t) := \{ \tau \in (0,t) : \nabla u (x + \tau \nu ) \notin B_{r}(e)\}.
\]

Starting from the apriori assumption that $\nabla u$ is continuous (however the argument is independent of this modulus of continuity), we can consider $\nabla u (x + t \nu )$, for $t\in(0,1)$ and $\nu\in \p  B_1$ , as a family of curves in $B_1$, starting from the point $\nabla u(x)\in B_r(e)$. Our assumption that $|\A_r|\geq \mu$ implies that some of these curves will exit $B_r(e)$ and spend some time outside $B_r(e)$. We quantify this by defining $t_\nu := \inf_{t} \{ |\A_{r,\nu}(t)| >c t\mu \}$, corresponding to the \textit{first time}, in each direction, where the curve spent \emph{some time} outside $B_r(e)$. Clearly $t_\nu>0$. We emphasize, however, that for $t<t_\nu$, the curve actually spent most of the time inside $B_r(e)$, where $\nabla u\approx e$ hence the equation is not degenerate (this is the key point in the argument). We also define \emph{good} directions as those close to $e$ for which the complete curve, $\nabla u(\A_{r,\nu}(1))$, spends a sufficient amount of time outside $B_r(e)$, i.e., those directions for which $t_\nu\leq 1$. Since $\A_r$ has positive measure, then for some directions, $\A_{r,\nu}$ also needs to have positive measure. Therefore, the measure of good directions has to be bounded from below by some universal constant.

For these good directions and for $t<t_\nu$, we are able to prove that 
\[
u(x+t \nu) - u (x)\geq ct,
\]
which, after writing the kernel in polar coordinates, gives
\[
K_u(x,x+t\nu) = (p-1) \frac{|u(x+ t\nu)-u(x)|}{t^{d+sp}}^{p-2} \geq ct^{-d-2+p(1-s)}
\]
where $-2+p(1-s)<0$ for $p<2/(1-s)$. The kernel has exactly the right singularity. Integrating in $\A_{r,\nu}(t_\nu)$ and in the good directions $\nu$, we obtain enough control on the linearized kernel to get an improvement of flatness for the directional derivatives of $u$.

\medskip

As in Lemma \ref{l:eqforv_e_scaled}, we define $v_e$ to be a properly localized version of the directional derivative of $u$ with $R=2^{K_0-1}$,
\begin{equation} \label{def_ve_dyadic} 
v_e(x) = \eta(2^{1-K_0}x) (e\cdot \nabla u(x)).
\end{equation}
With this definition, $v_e$ has support contained in $B_{2^{K_0}}$  and $v_e(x) = e \cdot \nabla u(x)$ for $x \in B_{2^{K_0-1}}$. Moreover, from Condition \textit{(3)}, we have $|v_e| \leq 1$ in $B_{1}$ and $|v_e(x)| \leq (2|x|)^{\alpha_0}$ in $B_{1}^c$ where $\alpha_0=\log(1-\delta)/\log(1/2)$.
In the next lemma, we analyze the equation satisfied by $v_e$ in $B_1$. 

\begin{lemma} \label{l:ve-eq-small-rhs}
Let $u$ be a function satisfying the assumptions of Lemma \ref{l:iteration}. Let $v_e$ be defined as in \eqref{def_ve_dyadic}. Then, it solves the equation
\begin{equation} \label{eq:equation_ve}
|\LL_u v_e| \leq \eps_0 \text{ in } B_{1},
\end{equation} 
where $\eps_0$ is arbitrarily small if we choose $\eps_1$ and $\delta$ small and $K_0$ large in Lemma \ref{l:iteration}.
\end{lemma}

\begin{proof}
We use Assumptions \textit{(3)}--\textit{(4)} and apply Lemma \ref{l:eqforv_e_scaled} to obtain
\begin{align*}
|\LL_u v_e| &\leq C \left( 2^{-K_0(1+sp)} \|u\|^{p-1}_{L^\infty(B_{2^{K_0}})} +{\Tail}_{p-1,sp+1}(u,2^{K_0})^{p-1} \right)  \\
&\leq C \left( 2^{-K_0(1+sp)} (1-\delta)^{-K_0(p-1)}2^{K_0(p-1)} + \eps_1 \right) \\& \leq C \left( \left((1-\delta)^{-(p-1)} 2^{-2+p(1-s)} \right)^{K_0}+ \eps_1  \right) =: \eps_0. 
\end{align*} 
We point out that $\eps_0$ can be made arbitrarily small by taking $K_0$ large enough, since $(1-\delta)^{-(p-1)} 2^{p(1-s)-2}  <1$ for $\delta$ small and $\eps_1$ is small.    
\end{proof}

The equation \eqref{eq:equation_ve} is the main tool we use in this section to prove Lemma \ref{l:iteration}. We will first prove the following preparatory lemma. Lemma \ref{l:iteration} will follow later on as a consequence.

\begin{lemma} \label{l:insomeball}
Let $u$ be a function satisfying the assumptions of Lemma \ref{l:iteration} and $v_e$ defined by \eqref{def_ve_dyadic}. Then, there exists a point $x_1 \in \p B_{1/2}$ such that
\[
v_e\leq 1-\delta \quad \text{ in } B_{1/16}(x_1).
\]
\end{lemma}

We introduce the following notation. We denote the unitary sphere in $\R^d$ by $S^{d-1}$. Given a vector $e \in S^{d-1}$, $\Gamma \subset S^{d-1}$ and $a_0,\rho_0>0$, we define the sets
\begin{equation}
    \begin{aligned} \label{A_rho}
    \A_{\rho_0} &:= \{ x \in B_1 : |\nabla u(x) - e| \geq \rho_0 \}, \\
    \Gamma(e,a_0)&:= \{ \nu \in S^{d-1} \, : |\nu \cdot e| > a_0\},\\
    C_\Gamma(x) &:= \{ x + t\nu: t \in [0,1) \text{ and } \nu \in  \Gamma  \}.
\end{aligned} 
\end{equation}
In particular, Assumption \textit{(5)} in Lemma \ref{l:iteration} can be rewritten as $|\A_r| \geq \mu > 0$.

The point $x_1$ in Lemma \ref{l:insomeball} is chosen based on the following lemma. 

\begin{lemma} \label{l:pickaball}
Assume that $|\A_r| \geq \mu$. There is some point $x_1 \in \p B_{1/2}$ such that \begin{equation} \label{cond}
    |C_{\Gamma(e,1/6)}(x) \cap \A_r| \geq \mu/2  \qquad \text{for all $x \in B_{1/8}(x_1)$}.
\end{equation}
\end{lemma}
\begin{proof}
    Consider $x_1$ and $x_2$ two antipodal points on $\p B_{1/2}$ such that $x_1-x_2$ is parallel to $e$. Observe that, by construction,
    \[
    \left[ \bigcap_{x \in B_{1/8}(x_1)} C_{\Gamma(e,1/6)}(x) \right] \bigcup \left[ \bigcap_{x \in B_{1/8}(x_2)} C_{\Gamma(e,1/6)}(x) \right] = B_1.
    \]
    Thus, either $x_1$ or $x_2$ satisfies \eqref{cond}.
\end{proof}

The idea of Lemma \ref{l:pickaball} is that we are finding a ball $B_{1/8}(x_1)$ so that the set $\A_r$ intersects a cone of nondegenerate directions for the kernel $K_u$ centered at any point in this ball. In Lemma \ref{l:insomeball}, which we will prove through this section, we take advantage of this fact to improve our bound on $v_e$ inside the ball $B_{1/16}(x_1) $. We will eventually use the fact that this set intersects the cone of non-degeneracy from any point in $B_{1/2}$ to propagate our improved bounds and obtain Lemma \ref{l:iteration}.

We now explain the setup that eventually leads to the proof of Lemma \ref{l:insomeball}. It follows the ideas of De Giorgi's iteration scheme.

Let $x_1$ be the point from Lemma \ref{l:pickaball}. We consider a bump function $\eta_1$ that is equal to $1$ in $B_{3/32}(x_1)$ and vanishes outside $B_{1/8}(x_1)$. In fact, we will consider a whole family of such bump functions $\eta_k$ so that
\begin{itemize}
\item $\eta_k$ is equal to $1$ in $B_{1/16 + 4^{-k-2}}(x_1)$.
\item $\eta_k$ is equal to $0$ outside $B_{1/16 + 4^{-k-1}}(x_1)$.
\item $\eta_k$ is smooth, with $|\nabla \eta_k| \leq C 2^{2k}$ and $|\nabla^2 \eta_k| \leq C 2^{4k}$.
\end{itemize}
Note that by construction, the function $\eta_k$ is identically equal to one on the support of $\eta_{k+1}$.

We also consider a sequence $\delta_k$ of small positive numbers. For some fixed parameter $\delta>0$ (to be determined), we define
\[ \delta_k = (1 + 2^{-k}) \, \delta. \]

The important properties of the sequence $\delta_k$ to keep in mind are $\delta \leq \delta_k \leq 2 \delta$, and $\delta_k - \delta_{k+1} = 2^{-k-1} \delta$. We also use the name $\varphi_k$ to denote the function \begin{equation} \label{def_varphik}
    \varphi_k=1-\delta_k \eta_k. 
\end{equation}

Let us define the set $A_k$ where $v_e$ is larger than $1-\delta_k \eta_k$
\begin{equation} \label{def_levelsetsAk}
A_k := \{ x \in B_1 : v_e(x) > 1 - \delta_k \eta_k(x) \}.
\end{equation}

Since $v_e$ is uniformly bounded by $1$, we have that $A_k$ is a subset of the support of $\eta_k$. Our objective is to show that $|A_k|$ is arbitrarily small for large $k$, effectively proving that $\bigcap_{k=1}^\infty A_k$ has measure zero. From this we will conclude that $v_e \leq 1 - \delta$ in $B_{1/16}(x_1)$ for some small $\delta > 0$.

For fixed $e \in S^{d-1}$ and any value of $k$, we define
\begin{equation} \label{def:v_k}
v_k= \begin{cases}
 \left( v_e - (1-\delta_k \eta_k)  \right)_+ & \text{in } B_1, \\
    0 & \text{outside } B_1.
\end{cases} \end{equation}

We multiply the equation $\LL_u [v_e] \leq \varepsilon_0$ by $v_k$ and integrate over $B_1$ to get

\begin{equation} \label{eq:testvk}
    \int \LL_u [v_e] v_k\,\dd x\leq \varepsilon_0 \int v_k \,\dd x.
\end{equation}

\begin{remark}
From Lemma \ref{l:LL-in-Walpha}, we see that the integral on the left-hand side of \eqref{eq:testvk} makes sense as soon as $u \in L^{p-1}_{sp}$ is Lipschitz in $B_1$, and $v_e$ (and thus also $v_k$) is in $H^{1-p(1-s)/2}(B_1)$.

When $p \in (1/(1-s),2/(1-s))$, using Lemma \ref{l:LL-well-defined-Lipschitz}, we may also see that the integrand is well-defined pointwise and bounded provided that $u \in C^{1,1}$.

Incidentally, for $p \in [2,1/(1-s)]$, the integrand is also pointwise well-defined and bounded if $u \in C^{2,\alpha}$ for some $\alpha > 1 - p(1-s)$. We explain in Appendix \ref{s:appendix_2} that the linearized operator $\LL_u v(x)$ may not make sense pointwise, even if $u$ and $v$ are smooth, where $\nabla u(x) = 0$. However, the support of $v_k$ consists only of points where $e \cdot \nabla u > 1-\delta_k$. We have $|\nabla u| \geq 1-\delta_k > 0$ at these points. Lemma \ref{l:LL-well-defined} applies and tells us that $\LL_u v_e(x)$ is well-defined in the support of $v_k(x)$.
\end{remark}

Note that the symmetry $K_u(x,y)=K_u(y,x)$ gives $\LL_u$ a variational structure. We split the left hand side into several terms.  We write
\[
\int_{B_1} \LL_u [v_e] v_k\,\dd x\geq J_{Tail,k}+J_{\varphi,k}+J_{gap,k}+J_{r,k},
\]
where
\begin{equation} \label{eq:defJgapk}
\begin{aligned}
    J_{Tail,k} =&\,\int_{\R^d\setminus B_1}\int_{B_1} (\varphi_k(x)-v_e(y))v_k(x) K_u(x,y)\,\dd x\dd y\\
J_{\varphi_k}=&\,\iint_{B_1\times B_1}(\varphi_k(x)-\varphi_k(y))v_k(x)K_u(x,y)\,\dd x\dd y\\
    J_{gap,k}=&\,\iint_{B_1\times B_1}[\varphi_k(y)-v_e(y)]_+v_k(x)K_u(x,y)\,\dd x\dd y\\
    J_{r,k}=&\, \iint _{B_1\times B_1} (v_k(x)-v_k(y))v_k(x) K_u(x,y)\,\dd x\dd y.
\end{aligned}
\end{equation}
Indeed, 
\begin{align*}
    \int_{B_1} \LL_u [v_e] v_k\,\dd x 
    &=\int_{\R^d\setminus B_1}\int_{B_1} (v_e(x)-v_e(y))v_k(x) K_u(x,y)\,\dd x\dd y  \\
    &+ \iint_{B_1\times B_1} [(v_e(x)-v_k(x)) + v_k(x) -(v_e(y)-v_k(y))-v_k(y)] v_k(x)  K_u(x,y)\,\dd x\dd y 
    \intertext{note that both $v_e \geq \varphi_k$ and $v_e -v_k=\varphi_k $ hold in $\supp v_k$}
    &\geq\int_{\R^d\setminus B_1}\int_{B_1} (\varphi_k(x)-v_e(y))v_k(x) K_u(x,y)\,\dd x\dd y \\
    &+ \iint_{B_1\times B_1} [\varphi_k(x) + v_k(x) +\varphi_k(y) - \varphi_k(y) -(v_e(y)-v_k(y)) -v_k(y)] v_k(x)  K_u(x,y)\,\dd x\dd y  \\
    &= J_{Tail,k}+\iint_{B_1\times B_1}(\varphi_k(x)-\varphi_k(y))v_k(x)K_u(x,y)\,\dd x\dd y \\
    &+ \,\iint_{B_1\times B_1}[\varphi_k(y)-v_e(y)]_+v_k(x)K_u(x,y)\,\dd x\dd y \\ &+ \iint_{B_1\times B_1}(v_k(x)-v_k(y))v_k(x)K_u(x,y)\,\dd x\dd y.  
\end{align*}

We will see that the terms $J_{gap,k}$ and $J_{r,k}$ are strictly positive, while all the other terms are not too negative. Appropriate lower bounds for these four terms combined with \eqref{eq:testvk} will yield a sequence of estimates on the measure of the sets $A_k$.

\subsection{Estimating \texorpdfstring{$J_{Tail,k}$}{J-Tail}}

We start with some simple bounds for the tails of the kernel $K_u$ under the assumptions of Lemma \ref{l:iteration}. They will be used to estimate $J_{Tail,k}$.

\begin{lemma} \label{l:tails_of_kernel_1}
Let $u$ be a function satisfying the assumptions of Lemma \ref{l:iteration} and $K_u$ be the kernel defined in \eqref{eq:kernel}. Assuming that $\delta$ and $\eps_1$ are small, it satisfies the following bounds
\begin{align*}
    \int_{B_1^c} K_u(x,y) \, \dd y &\leq C \text{ for } x \in B_{3/4}.
\end{align*}
Here, $C$ depends on dimension, $s$ and $p$.
\end{lemma}

\begin{proof}
Indeed, for $x \in B_{3/4}$, using that $|u(x)|\leq 1$ and \eqref{Tail_functions}, we obtain
\begin{align*}
    \int_{B_1^c} K_u(x,y) \dd y &\leq \int_{B_1^c} \frac{(|u(y)|+1)^{p-2}}{|x-y|^{d+sp}} \dd y  \\
    & \leq C \left( {\Tail}_{p-2,sp}(u,1)^{p-2} +1 \right) \\
    & \leq C\left( {\Tail}_{p-1,sp+1}(u,1)^{p-2} +1 \right).
\end{align*}
for a universal constant $C>0$. The tail of $u$ is bounded as follows.

Define $\theta_0 := 2^{p(1-s)-2}<1$. We choose $\delta$ sufficiently small so that $\theta := (1-\delta)^{-(p-1)} 2^{p(1-s)-2}< (1+\theta_0)/2$. Note that this value of $\theta$ depends on $s$ and $p$ only.

\begin{equation} \label{eq:bound_tail_iter}
  \begin{aligned} 
  \text{Tail}_{p-1,sp+1}(u,1)^{p-1}
  &\leq \sum_{n=0}^{K_0-1} \int_{B_{2^{n+1}} \setminus B_{2^n}} \frac{|u(z)|^{p-1}}{|z|^{d+sp+1}}  \dd  z +\text{Tail}_{p-1,sp+1}(u,2^{K_0})^{p-1}  \\
  & \leq \sum_{n=1}^{K_0-1} \int_{B_{2^{n+1}} \setminus B_{2^n}} \frac{|(1-\delta)^{-n}2^n)|^{p-1}}{|2^n|^{d+sp+1}}  \dd  z + \eps_1^{p-1} \\ 
  &\leq C \sum_{n=1}^{K_0-1} \theta^{n} + \eps_1^{p-1}
  \\
  & \leq C \frac{1}{1-\theta} + \eps_1^{p-1}   \leq C_1,
\end{aligned}
\end{equation}
where $C_1$ is a positive constant depending on $d$, $s$ and $p$ only.
\end{proof}

\begin{lemma} \label{l:tails_of_kernel_2}
Let $u$ be a function satisfying the assumptions of Lemma \ref{l:iteration} and $K_u$ be the kernel defined in \eqref{eq:kernel}. Assuming $\delta$ and $\eps_1$ are small, it satisfies the following bound
\begin{align*}
    \sum_{n=0}^\infty \int_{B_{2^n} \setminus B_{2^{n+1}}} (1-\delta)^{-n} K_u(x,y) \, \dd y &\leq C \text{ for } x \in B_{3/4}.
\end{align*}
Moreover, we also have
\begin{align*}
    \sum_{n=0}^\infty \int_{B_{2^n} \setminus B_{2^{n+1}}} \left( (1-\delta)^{-n} - 1 \right) K_u(x,y) \, \dd y &\leq C(\delta) \text{ for } x \in B_{3/4},
\end{align*}
where $C(\delta) \to 0$ as $\delta \to 0$. 

In this lemma, $C$ depends on $d$, $s$, $p$. The bound $C(\delta)$ depends on these quantities as well as $\delta$.
\end{lemma}

\begin{proof}
    We do a computation similar to Lemma \ref{l:tails_of_kernel_1} using that $|u(x)|\leq 1$.
    \begin{align*}
        \sum_{n=0}^\infty \int_{B_{2^n} \setminus B_{2^{n+1}}} (1-\delta)^{-n} K_u(x,y) \, \dd y &= \sum_{n=0}^\infty \int_{B_{2^{n+1}} \setminus B_{2^{n}}} (1-\delta)^{-n} \frac{|u(x)-u(y)|^{p-2}}{|x-y|^{d+sp}} \, \dd y \\
        & \leq \sum_{n=0}^\infty \int_{B_{2^{n+1}} \setminus B_{2^{n}}} (1-\delta)^{-n} \frac{(1 + |u(y)|)^{p-2}}{|x-y|^{d+sp}} \, \dd y \\
        &\leq C\int_{B_1^c} \frac{(1+|u(y)|)^{p-2}}{|y|^{d+sp-\alpha_0}} \\
        &\leq C \left( 1 + {\Tail}_{p-2,sp-\alpha_0}(u,1) \right).
    \end{align*}
    Here, $\alpha_0$ is the exponent so that $2^{-\alpha_0} = (1-\delta)$. Note that $\alpha_0$ is small if $\delta$ is chosen small. Using H\"older's inequality similarly as in \eqref{Tail_functions}, provided $\alpha_0 < (p(1-s)-2)/(p-1)$, we get for all $R \geq 1$,
    \[ {\Tail}_{p-2,sp-\alpha_0}(u,R) \leq C R^{\frac{p(1-s)-2}{p-1} + \alpha_0} {\Tail}_{p-1,sp+1}(u,R).\]
    In particular, we apply it for $R=1$ and proceed similarly as in the proof of Lemma \ref{l:tails_of_kernel_1}.

    When we proceed with the computation of the upper bound for the second quantity, we obtain
    \begin{align*}
        \sum_{n=0}^\infty \int_{B_{2^n} \setminus B_{2^{n+1}}} \left( (1-\delta)^{-n} - 1 \right) K_u(x,y) \, \dd y &= \sum_{n=0}^\infty \int_{B_{2^{n+1}} \setminus B_{2^{n}}} \left((1-\delta)^{-n}-1  \right)\frac{|u(x)-u(y)|^{p-2}}{|x-y|^{d+sp}} \, \dd y \\
        & \leq C \int_{B_1^c} ((2|y|)^{\alpha_0}-1) \frac{1 + |u(y)|^{p-2}}{|y|^{d+sp}} \, \dd y \\
        \intertext{applying H\"older's inequality,}
        &\leq C(\alpha_0) + C 
    \; {\Tail}_{p-1,sp+1}^{p-2} \left( \int_{B_1^c} \frac{((2|y|)^{\alpha_0}-1)^{p-1}}{ |y|^{(d+sp) - p + 2}} \dd y \right)^{1/(p-1)}.
    \end{align*}  
    The last integral factor converges to zero as $\alpha_0 \to 0$. Recall that $\alpha_0$ equals $-\log (1-\delta) / \log 2$, and $\alpha_0 \to 0$ as $\delta \to 0$. This concludes the proof of the second inequality.
\end{proof}

\begin{lemma}\label{l:bound_J_Tail}
    Under the assumptions of Lemma \ref{l:iteration}, we have
    \begin{equation} \label{est:JTail}
          J_{Tail,k} \geq - \varepsilon(\delta) \int_{B_1} v_k\,\dd x,
    \end{equation}
    where $\varepsilon(\delta)$ is small, provided $\delta$ is small. In addition to $\delta$, the value of $\eps(\delta)$ depends on $d$, $p$ and $s$.
\end{lemma}

\begin{proof}
By construction $v_e$ is supported inside $B_{2^{K_0}}$. Moreover, $v_e \leq |\nabla u| \leq (1-\delta)^{-n}$ in $B_{2^n}$. We get 
    \begin{align*}
        J_{Tail,k} &\geq \int_{B_{2^{K_0}} \setminus B_1}\int_{B_1} \left[ (1-\delta_k\eta_k)-v_e(y) \right]v_k(x)K_u(x,y)\,\dd x\dd y\\
        &\geq - \left( \int_{B_1} v_k(x)\,\dd x \right) \cdot \sup_{x \in \supp v_k} \left( \delta_k \int_{B_1} \eta_k(y) K_u(x,y) \dd y + \sum_{n=0}^\infty \int_{B_{2^{n+1}} \setminus B_{2^n}} (1-\delta)^{-n-1} K_u(x,y) \dd y \right) \\
        &\geq -(C_{1,k}\delta_k + C_2(\delta))\int_{B_1} v_k(x)\,\dd x.
    \end{align*}
We use Lemma \ref{l:tails_of_kernel_1} to bound $C_{1,k}$ and Lemma \ref{l:tails_of_kernel_2} to bound $C_2(\delta)$. We recall that $C_2(\delta) \to 0$ as $\delta \to 0$ and that $\delta_k \leq 2 \delta$. The result follows.
\end{proof}

\subsection{Estimating \texorpdfstring{$J_{gap,k}$}{Jgap}}

We now focus on estimating $J_{gap,k}$. 

A key idea to control the non-degeneracy of the kernel in some directions shows in this proof. At some point, we will need to find pairs of points $x$ and $y$ so that $K_u(x,y)$ has a bound from below. We do it in terms of the values of $v_e$, by looking for $(x-y)$ that is not perpendicular to $e$, and so that $\nabla u$ stays close to $e$ in a large proportion of the segment from $x$ to $y$. The following elementary calculus lemma will come handy.

\begin{lemma} \label{l:elementary-1D}
Let $f : [0,t] \to \R$ be a continuously differentiable function so that
\begin{itemize}
\item $|f'(\tau)| \leq 1$ for $\tau \in [0,t]$.
\item $|\{\tau \in (0,t) : f'(\tau) < 1/10\}| \leq c_0 t$, where $c_0 = 1/22$.
\end{itemize}
Then $f(t) - f(0) > t/20$.
\end{lemma}

\begin{proof}
This is a straightforward application of the fundamental theorem of Calculus.
\begin{align*}
    f(t) - f(0) &= \int_0^t f'(\tau) \dd \tau, \\
    &\geq (1-c_0) \frac t {10} - c_0t = \left( \frac 1 {10} - c_0 \frac {11} {10} \right) t.
\end{align*}
And the result follows. 

Note that the constants $1/10$ and $1/20$ could be changed by any other pair of ordered fractions in $(0,1)$ by appropriately adjusting the value of $c_0$.
\end{proof}

We now derive a lower bound on $J_{gap,k}$.

\begin{lemma}\label{l:Bound_Jgap1} For any $ r,\mu>0$, let $u \in C^{1}(B_2)$ and for $e \in S^{d-1}$ let $v_e$ be defined as in \eqref{def_ve} and the set $\A_r$ be the one defined in \eqref{A_rho}. If $|\A_r|>\mu|B_1|$, then there exists  $\delta>0$ depending on $\mu,\, r,\, d,\,s,\,p$ such that, given $v_k$ defined as \eqref{def:v_k} and $J_{gap,k}$ defined in \eqref{eq:defJgapk}, we have 
    \[
    J_{gap,k}  \geq c\,r^2 \displaystyle \mu^2 \int_{B_1} v_k(x)   \dd x .
    \]
\end{lemma}

\begin{proof}
Recall that $v_k \leq \delta_k \eta_k$, and these functions are supported inside the ball $B_{1/8}(x_1)$ obtained in Lemma \ref{l:pickaball}.

For any $x \in B_{1/8}(x_1)$, if we take $y \in \A_r$ then $\varphi_k(y)$ and $v_e(y)$  separate strictly for $\delta$ smaller than $r^2/2$. Indeed, $\varphi_k(y)\geq 1-\delta_k$ and $v_e(y) \leq 1-\frac{r^2}{2}$ for $y \in \A_r$.  Therefore, we obtain
\begin{align*}
    J_{gap,k}&= \,\iint_{B_1\times B_1}[\varphi_k(y)-v_e(y)]_+v_k(x)K_u(x,y) \dd  x \dd  y \\ 
    & \geq \int_{B_1} \int_{C_{\Gamma(e,1/6)}(x) \cap \A_r }[\varphi_k(y)-v_e(y)]_+v_k(x)K_u(x,y) \dd  x \dd  y \\
    &\geq \int_{B_1} v_k(x) \left( \frac{r^2}{2} - \delta_k\right) \int_{C_{\Gamma(e,1/6)}(x) \cap \A_r } K_u(x,y)  \dd  y  \dd  x.
 \end{align*}

We choose $\delta$ sufficiently small to ensure that $r^2 / 2 - \delta_k \geq r^2/4$.

The values of the kernel $K_u(x,y)$ depend on the values of $u$ at $x$ and $y$. In order to get a good lower bound for the integral on the right-hand side above, we must avoid those points where $u(x)$ and $u(y)$ are too close, making the kernel $K_u(x,y)$ small. The choice of $y$ in the cone $C_{\Gamma(e,1/6)}(x)$ is made with that purpose, but it is not enough to ensure $|u(x) - u(y)| \gtrsim |x-y|$.

We introduce the following notation. Given a fixed $x$ in the support of $v_k$ (which is contained in $B_{1/8}(x_1)$), for each direction $\nu \in S^{d-1}$ and $t\in(0,1)$, we define a subset of the line segment $[x,x+t\nu]$ by 
\[
\A_{r,\nu}(t) := \{ \tau \in (0,t) : \nabla u (x + \tau \nu ) \notin B_{r}(e)\}
\]
and
\[
t_\nu := \inf_{t} \left\{ |\A_{r,\nu}(t)| \geq \frac{c_0 \mu} {|S^{d-1}|} t \right\}.
\]
Here, $c_0>0$ is the universal constant of Lemma \ref{l:elementary-1D}.

We may think about the set $\A_{r,\nu}$ in terms of the curve described by $\nabla u(x+\tau \nu)$ as $\tau \in (0,t)$. The set $\A_{r,\nu}$ corresponds to the values $\tau$ for which this curve moves outside of $B_r(e)$. Figure \ref{fig:Arnu} explains this geometric interpretation.

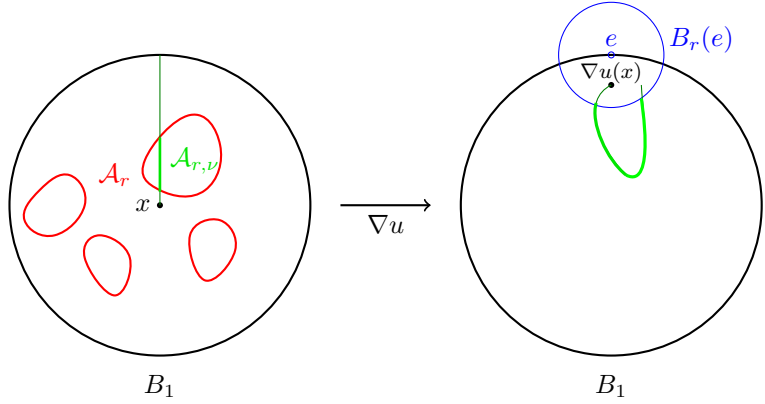
\begin{figure}[ht]
    \centering
    \begin{tikzpicture}[scale=2]

\draw[thick] (0,0) circle(1) node at (0,-1.2) {$B_1$};

\begin{scope}
  \draw[red,thick]
    plot[smooth cycle,tension=0.8] coordinates{(-0.1,0.3) (0.2,0.6) (0.4,0.4) (0.3,0.1) (0,0.1)} ;
  \draw[red,thick]
    plot[smooth cycle,tension=0.8] coordinates{(-0.7,-0.2) (-0.5,0) (-0.6,0.2) (-0.8,0.1) (-0.9,-0.1) } node  at (-0.3,0.2){$\mathcal{A}_r$};
  \draw[red,thick]
    plot[smooth cycle,tension=0.8] coordinates{(0.3,-0.5) (0.5,-0.3) (0.4,-0.1) (0.2,-0.2)};
      \draw[red,thick]
    plot[smooth cycle,tension=0.8] coordinates{(-0.3,-0.6) (-0.5,-0.35) (-0.4,-0.2) (-0.2,-0.35)};
\end{scope}

\fill[black] (0,0) circle(0.02) node[left] {$x$};

\draw[green!50!black] (0,0) -- (0,1);

\draw[line width=1pt,green!90!black] (0,0.09) -- (0,0.45) node[below right] {$\mathcal{A}_{r,\nu}$};

\draw[thick] (3,0) circle(1) node at (3,-1.2) {$B_1$};

\draw[->,black,thick] (1.2,0) -- (1.8,0) node[midway,below] {$\nabla u$};

\draw[blue] (3,1) circle(0.35) node at (3.6,1.1) {$B_r(e)$};

\draw[blue] (3,1)  circle(0.02) node[above] {$e$};

\fill[black] (3,0.8) circle(0.02) node[font=\footnotesize] at (3,0.88) {$\nabla u(x)$};

\draw[green!50!black]
  (3.2,0.8) .. controls (3.2,0.7) and (3.3,0.1) .. (3.1,0.2)
          .. controls (2.9,0.4) and (2.8,0.7) .. (3,0.8);

\draw[green!90!black,very thick]
  (3.1,0.2) .. controls (2.97,0.3) and (2.86,0.55) .. (2.9,0.68);

\draw[green,very thick]
  (3.2,0.71) .. controls (3.2,0.7) and (3.3,0.1) .. (3.1,0.2);

 \end{tikzpicture} 
    \caption{Description of the method.}
    \label{fig:Arnu}
\end{figure}

\medskip

Since $x$ is a point in the support of $v_k$, we know that $v_e(x) > 1-\delta$, and in particular $\nabla u(x) \in B_r(e)$. Thus, $0 \notin \A_{r,\nu}$ and $t_\nu > 0$ (from the continuity of $\nabla u$). Define also the set of \emph{good} directions 
\[
L:= \left\{   \nu \in \Gamma(e,1/6) : |\A_{r,\nu}(1)|\geq \frac{c_0 \mu} {|S^{d-1}|} \right\}.
\]
Since \eqref{cond} holds, we have that there exists a small dimensional constant $c_1$ such that $|L|>c_1\mu $. This is justified below by evaluating the measure of $\A_r$ using integration in polar coordinates.
\begin{align*}
    \mu |B_1|\leq |\A_r| &\leq \frac{|L|}{d}+\int_{L^c}\int_{\A_{r,\nu}(1)}t^{d-1}  \dd  t \dd  \nu\\
    &\leq \frac{|L|}{d} + \int_{L^c} |\A_{r,\nu}(1)| \dd \nu \\
    &\leq \frac{L}{d}\ + \max_{\nu\in L^c}|\A_{r,\nu}(1)| \cdot |S^{d-1}|, \\
    \intertext{using that $|\A_{r,\nu}(1)|\leq c_0 \mu / |S^{d-1}|$ for all $\nu \notin L$,}
    \mu |B_1| &\leq \frac{|L|}{d} + c_0 \mu, \\
    \intertext{which finally gives}
    |L| &\geq c_1 \, \mu,
\end{align*}
as intended.

 Clearly, for any $\nu \in L$ the infimum is attained and thus we have $t_\nu \leq 1$. 

Define also the cone with vertex at $x$ consisting of the good directions
\[
C_L(x):= \{x + \tau\nu  : \nu \in L, \tau \in (0,t_{\nu})\}.
\]
In the cone $C_L(x)$, we can obtain good properties for the kernel $K_u$
    \begin{align*}
    J_{gap,k}&\geq  \int_{B_1} v_k(x) \left( \frac{r^2}{2} - \delta_k\right) \int_{C_L(x)} K_u(x,y)  \dd  y  \dd  x ,
    \intertext{expressing the inner integral in polar coordinates and recalling that $\delta_k < r^2/4$, we get}
    &= \frac {r^2}4 \int_{B_1} v_k(x) \left( \int_{L} \int_{\A_{r,\nu}(t_\nu)} K_u(x,x+t\nu) t^{d-1}  \dd  t  \dd  \nu \right) \dd  x .
\end{align*}

Recall that $L$ is a subset of the directions $\Gamma(e,1/6)$. In particular for any $\nu \in L$ we have $\nu \cdot e > \frac{1}{6}$. If $t \leq t_{\nu}$ we also have that $|\A_{r,\nu}(t)| \leq c_0 \mu |S^{d-1}|^{-1} t < c_0t$. Thus, we can apply Lemma \ref{l:elementary-1D} to the function $f(\tau) = u(x+\tau \nu)$. Indeed, we see that $|f'|\leq 1$ and $f'(\tau) \geq (1/6-r) \geq 1/10$ if $\tau \notin \A_{r,\nu}$. Therefore, we deduce that $u(x+t\nu) - u(x) \geq t/20$.

Hence, for $\nu \in L$ and $t \in \A_{r,\nu}(t_\nu)$ we have
\[
K_u(x,x+t\nu) \geq c\, t^{p-2-(d+sp)}.
\]
Therefore
\begin{align*}
     \int_{L} \int_{\A_{r,\nu}(t_\nu)} K_u(x,x+t\nu) t^{d-1}  \dd  t  \dd  \nu &\geq c   \int_{L} \int_{\A_{r,\nu}(t_\nu)} t^{p(1-s)-3}  \dd  t \dd \nu.
\end{align*}
Note that 
\begin{align*}
    \int_{\A_{r,\nu}(t_\nu)} t^{p(1-s)-3} \dd t &\geq |\A_{r,\nu}(t_\nu)| t_\nu^{p(1-s)-3} \geq c \mu t_\nu^{p(1-s)-2}.
\end{align*}
This bound allows us to control $J_{gap,k}$. Going back to the computation at the beginning of this proof, we obtain
\begin{align*}
J_{gap,k} &\geq c r^2 \int_{B_1} v_k(x) \int_{C_{\Gamma(e,1/6)}(x) \cap \A_r } K_u(x,y)  \dd  y  \dd  x \\
&\geq c r^2 \mu \, \int_{B_1} v_k(x) \left( \int_{L} t_\nu^{p(1-s)-2}   \dd  \nu \right) \dd  x.
\end{align*}
Hence, exploiting that $|L| > c_1\mu $ and $t_\nu \leq 1 $ for $\nu \in L$, we conclude
\[
    J_{gap,k} \geq c r^2  \mu^2 \int_{B_1} v_k(x) \dd x .
\]

\end{proof}

\subsection{Estimating \texorpdfstring{$J_{\varphi_k}$}{Jphi}}

We use the symmetry $K_u(x,y)=K_u(y,x)$ to rewrite the double integral defining $J_{\varphi_k}$ and $J_{r,k}$ in a more symmetric form.
\begin{equation} \label{est:Jvarphik}
\begin{aligned}
    J_{\varphi_k}&= \iint_{B_1 \times B_1} (\varphi_k(x)-\varphi_k(y))v_k(x)K_u(x,y)\,\dd x\dd y  \\
    &=\frac{1}{2} \iint_{B_1 \times B_1} (\varphi_k(x)-\varphi_k(y))(v_k(x)-v_k(y))K_u(x,y)\,\dd x\dd y \\
    &= \frac{\delta_k}{2} \iint_{B_1 \times B_1} (\eta_k(x)-\eta_k(y))(v_k(x)-v_k(y))K_u(x,y)\,\dd x\dd y ,\\
   J_{r,k} &=\, \iint _{B_1\times B_1} (v_k(x)-v_k(y))v_k(x) K_u(x,y)\,\dd x\dd y \\
   &=\frac{1}{2} \iint_{B_1 \times B_1} (v_k(x)-v_k(y))^2 K_u(x,y)  \dd  x \dd  y.
\end{aligned}
\end{equation}

The term $J_{\varphi_k}$ may have either sign. We must estimate its absolute value appropriately.

\begin{lemma} \label{l:bound_Jphi}
    The term $J_{\varphi_k}$ satisfies the following upper bound.
    \[ |J_{\varphi_k}| \leq C \delta_k |A_k|^{1/2} \|\eta_k\|_{Lip(B_1)} \sqrt{J_{r,k}}.\]
    Here, $C$ is a universal constant depending only on $d$, $s$ and $p$.
\end{lemma}

\begin{proof}
We use the Cauchy-Schwarz inequality to estimate $J_{\varphi_k}$, paying attention to the support of the integrand. We have
\begin{equation} \label{eq:ineq11}
\begin{aligned}
|J_{\varphi_k}| &= \frac {\delta_k} 2  \left| \iint_{B_1 \times B_1} (\eta_k(x)-\eta_k(y))(v_k(x)-v_k(y)) K_u(x,y) \dd y \dd x \right| \\
& \leq  \frac{\delta_k}{2} \left( \iint_{B_1 \times B_1} (\eta_k(x)-\eta_k(y))^2 \one_{\{v_k(x)>0 \text{ or } v_k(y)>0\}} K_u(x,y) \dd y \dd x \right)^{1/2}\times \\
&\qquad \times \left( \iint_{B_1 \times B_1} (v_k(x)-v_k(y))^2 K_u(x,y) \dd y \dd x \right)^{1/2} \\
& \leq  \frac{\delta_k} 2 \left( 2 \iint_{B_1 \times B_1} (\eta_k(x)-\eta_k(y))^2 \one_{\{v_k(x)>0  \}} K_u(x,y) \dd y \dd x \right)^{1/2}\times \\
&\qquad \times
\left( \iint_{B_1 \times B_1} (v_k(x)-v_k(y))^2 K_u(x,y) \dd y \dd x \right)^{1/2} \\
& \leq C \delta_k |A_k|^{1/2} \|\eta_k\|_{Lip(B_1)} \left( \iint_{B_1 \times B_1} (v_k(x)-v_k(y))^2 K_u(x,y) \dd y \dd x \right)^{1/2}.
\end{aligned} 
\end{equation}
The last factor coincides with $J_{r,k}$.
\end{proof}

\subsection{Estimating \texorpdfstring{$J_{r,k}$}{Jr}}

The term $J_{r,k}$, as well as $J_{gap,k}$, has a strictly positive sign. In the next lemma, we compute a lower bound for $J_{r,k}$ in terms of the measures of the sets $A_k$ defined in \eqref{def_levelsetsAk}.

\begin{lemma} \label{l:bound_Jr}
    For any $k \geq 1$, we have
    \[ J_{r,k} = \frac 12 \iint_{B_1 \times B_1} (v_k(x)-v_k(y))^2 K_u(x,y) \dd y \dd x \geq c \delta^2 2^{-2k} |A_{k+1}| \cdot |A_k|^{-\frac{2-p(1-s)}{d}} . \]
\end{lemma}

\begin{proof}
The proof follows similar ideas as the proof of Lemma \ref{l:Bound_Jgap1}, but this time we track the influence of the set where $v_k \leq 0$ instead of the set where $\nabla u$ is far from $e$.

Consider any $x \in A_{k+1}$. In particular, we have $v_k(x) > \delta_k - \delta_{k+1} = 2^{-k-1} \delta$. For this value of $x$, we study the integral
\begin{align*}
    \int_{B_1} (v_k(x)-v_k(y))^2 K_u(x,y) \dd y &= \int_{S^{d-1}} \int_0^{1} (v_k(x)-v_k(x+t \nu))^2 K_u(x,x+t\nu) t^{d-1} \dd t \dd \nu.
\end{align*}
To obtain a good lower bound on this integral, we must identify directions $\nu$ which are not nearly perpendicular to $e$, and so that $u(x+t\nu)-u(x)$ is nearly equal to $t \nu \cdot e$. This is not true for all directions $\nu$ and all values of $t$, but we will see that it is true for a large enough set.

Mimicking the notation of Lemma \ref{l:Bound_Jgap1}, for each $\nu \in S^{d-1}$ and $t > 0$, we define
\[ \A_{\nu}(t) = \{ \tau \in (0,t) : v_k(x+\tau \nu) = 0 \}. \]
Note that $1-v_e(x+\tau \nu)< \delta_k$ for every $\tau \notin \A_{\nu}(t)$. In particular, this implies that $|\nabla u(x+\tau \nu)- e| \leq 2\sqrt{\delta_k}$ whenever $\tau \notin \A_{\nu}(t)$.
Recall that $\delta_k < 2\delta$ and $\delta$ is arbitrarily small.

For any direction $\nu$ such that $|\nu \cdot e| \geq 1/6$, we write
\[ t_\nu = \inf \{ t:|\A_\nu(t)| > c_0  t \}.\]
Here $c_0$ is the constant from Lemma \ref{l:elementary-1D}. We make $t_\nu = +\infty$ for those directions where $|\A_\nu(t)| \leq c_0 t$ for all $t > 0$.  When $t_\nu < +\infty$, we will always have $|\A_\nu(t_\nu)| \geq c_0  t_\nu$. Like in the proof of Lemma \ref{l:Bound_Jgap1} we set $f(t):=u(x+t \nu)$. Since
\[
|\{\tau  \in (0,t) \, :\, f'(t)>1/10  \}|\geq |\A_\nu(t)^c|>(1-c_0)t,
\]
by applying Lemma \ref{l:elementary-1D}, we deduce that for all $t \in (0,t_\nu)$, we have $|u(x+t\nu)-u(x)| \geq t/20$, and therefore
\[ 
K_u(x,x+t\nu) \geq c t^{p-2-(d+sp)}.
\]

Integrating over $t \in (0,t_\nu)$, we obtain
\begin{equation} \label{e1}
\begin{aligned}
\int_0^{t_\nu} \!\!
   \bigl(v_k(x)-v_k(x+t\nu)\bigr)^2 
   K_u(x,x+t\nu)\, t^{d-1}\,dt
&\;\;\geq\; 
   c \int_0^{t_\nu} 
      \delta^2 2^{-2k}\, 
      \one_{t \in A_\nu(t_\nu)}\,
      t^{\,p-2-(d+sp)} \, t^{d-1}\,dt \\[0.5em]
&\;\;\geq\; 
   c \,\delta^2 2^{-2k}
   \int_0^{t_\nu} 
      \one_{t \in A_\nu(t_\nu)}\, 
      t^{\,p(1-s)-3}\,dt \\[0.5em]
&\;\;\geq\; 
   c \,\delta^2 2^{-2k}\,
   t_\nu^{\,p(1-s)-2}.
\end{aligned}
\end{equation}

Recall that $p(1-s) -2 < 0$. Thus, the inequality still makes sense when $t_\nu = +\infty$.

We now define the following set of directions
\[
L:=\{\nu \in \Gamma(e,1/6)\,:\,t_\nu \leq C |A_k|^\frac{1}{d}\}.
\]
Then, by choosing $C$ sufficiently large, we can ensure that $|L|\geq |\Gamma(e,1/6)|/2$. Indeed, this can be verified expressing the measure of $A_k$ in polar coordinates.
\begin{align*}
    |A_k| &= \int_{S^{d-1}} \int_0^\infty \one_{\{x+t\nu \in A_k\}} \, t^{d-1} \dd t \dd \nu \\
    & \geq \frac 1d \int_{S^{d-1}} |\{t : x+t\nu \in A_k\}|^d \dd \nu.
\end{align*}
Chebyshev's inequality says that for any $m > 0$,
\[ |\{ \nu \in S^{d-1} : |\{t:x+t\nu \in A_k\}|^d > m \}| \leq \frac{d |A_k|} m.\]
We choose $m = C|A_k|$, so that the right-hand side is less than $|\Gamma(e,1/6)|/2$. Note that the choice of $C$ depends on dimension only. Taking the complement, we get
\[ |\{ \nu \in \Gamma(e,1/6) : |\{t:x+t\nu \in A_k\}|^d \leq C|A_k| \}| \geq \frac{|\Gamma(e,1/6)|} 2.\]
We observe that $|\{t:x+t\nu \in A_k\}|^d \leq C|A_k|$ implies that $t_\nu < (1+c_0) C^{1/d} |A_k|^{1/d}$. Therefore, adjusting the constant $C$ appropriately, $|L|\geq |\Gamma(e,1/6)|/2$ as claimed.

We proceed to integrate the inequality \eqref{e1} over all directions $\nu\in L$, obtaining for every $x \in A_{k+1}$ that
\begin{align*}
    \int_{B_1} (v_k(x)-v_k(y))^2 K_u(x,y) \dd y &\geq c \delta^2 2^{-2k} |A_{k}|^{-\frac{2-p(1-s)}{d}}.
\end{align*}
Integrating over $x \in A_{k+1}$, we obtain
\begin{align*}
    \iint_{B_1 \times B_1} (v_k(x)-v_k(y))^2 K_u(x,y) \dd y \dd x &\geq c \delta^2 2^{-2k} |A_{k}|^{-\frac{2-p(1-s)}{d}} \cdot |A_{k+1}|.
\end{align*}
This gives us the desired inequality.
\end{proof}

\subsection{Performing De Giorgi's iteration} \label{sub:degiorgi_it}

With the lemmas collected above, we are ready to estimate the measures of the sets $A_k$ defined in \eqref{def_levelsetsAk}.

Recall that in \eqref{eq:testvk} we showed that
\[    J_{Tail,k} + J_{\varphi_k} + J_{gap,k} + J_{r,k} \leq \eps_0 \int v_k \dd x. \]

The four terms on the left-hand side are estimated in lemmas \ref{l:bound_J_Tail}, \ref{l:bound_Jphi}, \ref{l:Bound_Jgap1} and \ref{l:bound_Jr}. We use these lemmas together with the inequality above to deduce estimates on $|A_k|$.

Recall that $J_{gap,k}$ and $J_{r,k}$ are both strictly positive.  Moreover, the lower bound on $J_{gap,k}$ given in Lemma \ref{l:Bound_Jgap1} absorbs the right-hand side provided that $\eps_0$ is sufficiently small (which is true if $\delta$ and $\eps_1$ are small according to Lemma \ref{l:ve-eq-small-rhs}), and also $J_{Tail,k}$ according to Lemma \ref{l:bound_J_Tail}. Thus, we end up with the simplified inequality
\begin{equation} \label{eq:3terms_leq_0}
    J_{\varphi_k}+ J_{r,k} + c(r,\mu) \int v_k \dd x \leq 0.
\end{equation}

To successfully carry out De Giorgi's iteration, we must prove that $|A_k| \to 0$ as $k \to \infty$. The first step is to show that at least one of these level sets has a very small measure (in our case, $A_2$), and then to obtain a recurrence relation between $|A_{k+1}|$ and $|A_k|$ that ensures convergence to zero as $k \to \infty$.

We start with an elementary observation about the integral of $v_k$.

\begin{lemma} \label{l:chevishev}
    The following inequalities hold
    \[ 2^{-k-2} \delta_k |A_{k+1}| \leq \int_{B_1} v_k \dd x \leq \delta_k |A_k|.\]
\end{lemma}

\begin{proof}
The right inequality follows from the fact that $v_k \leq \delta_k$ and $\supp v_k \subset A_k$. The left inequality is the result of applying Chebyshev inequality.
\end{proof}

The following lemma is used to ensure the initial step in De Giorgi's iteration.

\begin{lemma} \label{l:A2small}
By choosing $\delta$ small enough, we can ensure that $|A_2|$ is as small as we want.
\end{lemma}

\begin{proof}
    We use the bound from Lemma \ref{l:bound_Jphi} in \eqref{eq:3terms_leq_0} with $k=1$. We get
    \begin{equation} \label{eq:iteration-ineq-simplified}
    \begin{aligned}
    0 \geq \,& c(r,\mu)  \left(\int_{B_1} v_k \dd x\right) + J_{r,k} - C \delta_k |A_k|^{1/2} \|\eta_k\|_{Lip(B_1)} \sqrt{J_{r,k}}.
    \end{aligned}    
    \end{equation}
    We use the second term to absorb part of the third term and we are left with
    \begin{align*}
    0 &\geq c(r,\mu) \left(\int_{B_1} v_1 \dd x\right) - C \delta_1^2 |A_1|.
    \end{align*}
    Applying Lemma \ref{l:chevishev}, we see that
    \[ |A_2| \leq C(r,\mu) \delta.\]
    Choosing $\delta$ small, the result follows.
\end{proof}

Lemma \ref{l:A2small} takes advantage of the lower bound for $J_{gap,k}$ given in Lemma \ref{l:Bound_Jgap1}. We now want to get a recurrence relationship to estimate $|A_{k+1}|$ in terms of $|A_k|$. For that, our main tool will be the lower bound on $J_{r,k}$ given in Lemma \ref{l:bound_Jr}.

\begin{lemma} \label{l:iteration-ineq-measures}
For $k=2,3,4,\dots$, we have
\begin{equation} \label{eq:iteration-ineq-measures}
    |A_{k+1}| \leq C 2^{6k} |A_k|^{1+\kappa},
\end{equation}
for $\kappa = \frac{2-p(1-s)}{d} > 0$ and some universal constant $C$, depending on $d$, $s$ and $p$.
\end{lemma}

\begin{proof}
Focusing on first two terms in \eqref{eq:3terms_leq_0}, applying Lemma \ref{l:bound_Jphi} we obtain
\[ 0 \geq J_{r,k} - C \delta_k |A_k|^{1/2} \|\eta_k\|_{Lip(B_1)} \sqrt{J_{r,k}}. \]
Therefore,
\[
J_{r,k}  \leq C \delta^2 2^{4k}|A_k| .
\]
Using Lemma \ref{l:bound_Jr}, we deduce that 
\begin{align*}
    c 2^{-2k} \delta^2 |A_{k+1}| \cdot |A_k|^{-\frac{2-p(1-s)}{d}} &\leq C  \delta^2 2^{4k} |A_k|,
\end{align*}
from which the desired result follows.
\end{proof}

The inequality for the measures $|A_k|$ in Lemma \ref{l:iteration-ineq-measures} is the typical recurrence relationship in De Giorgi's iteration. Together with the smallness of $|A_2|$ given by Lemma \ref{l:A2small}, it implies that $|A_k| \to 0$ as $k \to \infty$. This is the result of an elementary convergence result for sequences, that we recall below.

\begin{lemma}\label{l:fast_convergence}
    Let $(Y_n)_{n}$ be a sequence of positive real numbers satisfying the recursive inequalities
    \[
        Y_{n+1}\leq Cb^n Y_n^{1+\delta},
    \]
    where $C,b>1$ and $\delta>0$ are given numbers. If
    \[
        Y_0\leq C^{-1/\delta}b^{-1/\delta^2},
    \]
    then $Y_n\to 0$ as $n\to \infty$.
\end{lemma}

The proof of Lemma \ref{l:insomeball} follows simply putting together our lemmas above.

\begin{proof} [Proof of Lemma \ref{l:insomeball}]
Lemma \ref{l:iteration-ineq-measures}, together with Lemma \ref{l:fast_convergence}, implies that the sequence $|A_k|$ is decreasing and converges to zero, provided that $|A_2|$ is small enough. We know that $|A_2|$ is small from Lemma \ref{l:A2small}. Therefore, we deduce if we choose $\delta>0$ sufficiently small, we must have $|\bigcap_{k\geq 2} A_k| = 0$, which means that $v_e \leq 1 - \delta$ in $B_{1/16}(x_1)  $.
\end{proof}

The proof of Lemma \ref{l:iteration} is deduced based on the same type of arguments.

\begin{proof} [Proof of Lemma \ref{l:iteration}]
The reason why the conclusion of Lemma \ref{l:insomeball} holds in $B_{1/16}(x_1)$ only is because it is only for $x \in B_{1/8}(x_1)$ that we can assure that $|\A_r \cap C_{\Gamma(e,1/6)}(x)| \geq c\mu$ for some $c>0$. We achieve it through Lemma \ref{l:pickaball}.

After we apply Lemma \ref{l:insomeball}, we may set $r_1: = \delta$ and observe that $B_{1/16}(x_1) \subset \A_{r_1}$. Now, we can argue that $C_{\Gamma(e,1/6)}(x)$ intersects with $\A_{r_1}$ for a much larger set of points $x$. The plan is to use this observation to derive an upper bound to the rest of the ball $B_{1/2}$.

Recall that the point $x_1$ that results from Lemma \ref{l:pickaball} is either $\frac 12 e$ or $-\frac 12 e$. Let us pick $x_2$ to be the other one of these two points. Consequently, it follows that for any $x \in B_{1/8}(x_2)$, the cone $C_{\Gamma(e,\frac{1}{6})}(x)$ contains $B_{1/16}(x_1)$ and
\begin{equation} \label{cond_x_2}
|\A_{r_1} \cap C_{\Gamma(e,\frac{1}{6})}(x)|\geq |B_{1/16}(x_1)| =:  \mu_1.
\end{equation}
Hence, we can repeat the argument in the proof of Lemma \ref{l:insomeball} with $r_1$, $\mu_1$ as above and replacing \eqref{cond} with \eqref{cond_x_2}. This implies that $v_e \leq 1-\delta_1$ in $B_{1/16}(x_2)$ for some $\delta_1 > 0$ smaller than $\delta$.

Let us set $r_2 := \delta_1$. Therefore $\A_{r_2} \supset B_{1/16}(x_1) \cup B_{1/16}(x_2)$. We observe that if $z$ is any point in $B_{7/16}$, for any $x \in B_{1/8}(z)$, the cone $C_{\Gamma(e,1/6)}(x)$ contains at least one of the two balls $B_{1/16}(x_1)$ or $B_{1/16}(x_2)$. Figure \ref{fig:Lemma4.1} illustrates the case when the cone is centered at a point $x$ farthest away from the origin in the direction perpendicular to $e$.

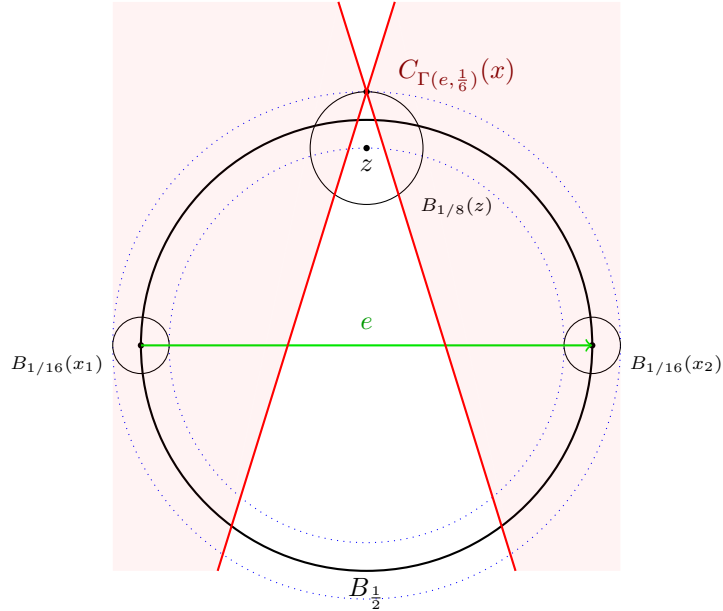
\begin{figure}[ht]
    \centering
    \begin{tikzpicture}[scale=6]

\def\rA{7/16}
\def\rB{8/16}
\def\rC{9/16}
\def\rsmall{1/16}
\def\rZ{1/8}

\draw[blue,dotted] (0,0) circle (\rA);
\node at (0,-0.55) {$B_{\frac{1}{2}}$};
\draw[black,thick] (0,0) circle (\rB);
\draw[blue,dotted] (0,0) circle (\rC);

\coordinate (x1) at (-\rB,0);
\coordinate (x2) at (\rB,0);

\draw[black] (x1) circle (\rsmall);
\fill (x1) circle (0.2pt) node[below left,font=\scriptsize]  at (-\rC,0)  { $B_{1/16}(x_1)$};

\draw[black] (x2) circle (\rsmall);
\fill (x2) circle (0.2pt) node[below right,font=\scriptsize]  at (\rC,0) {$B_{1/16}(x_2)$};

\coordinate (z) at (0,\rA) node[] at (0,0.4) {$z$}; 
\fill (z) circle (0.2pt) node[below,font=\scriptsize] at (0.2,0.35) {$B_{1/8}(z)$};

\draw[black] (z) circle (\rZ);

\coordinate (x) at (0,\rC);
\fill (x) circle (0.2pt);

\draw[red!50!black] node at (0.2,0.6) {$C_{\Gamma(e,\frac{1}{6})}(x)$};


\draw[->,green!90!black,thick] (x1) -- (x2);
\draw[green!60!black] node at (0,0.05) {$e$}; 

\draw[red,thick] (x) -- (0.0625,195/256);
\draw[red,thick] (x) -- (-0.33,-0.5);
\draw[red,thick]  (-0.0625,195/256)--(x);
\draw[red,thick]  (0.33,-0.5) -- (x);

\coordinate (coneleft) at (-0.33,-0.5);
\coordinate (coneright) at (0.33,-0.5);

\fill[red,opacity=0.05] (x) -- (0.0625,195/256) -- (0.33,-0.5) -- cycle;

\fill[red,opacity=0.05] (0.0625,195/256) -- (9/16,195/256)  --  (9/16,-0.5) -- (0.33,-0.5) -- cycle;

\fill[red,opacity=0.05] (x) -- (-0.0625,195/256) -- (-0.33,-0.5) -- cycle;

\fill[red,opacity=0.05] (-0.0625,195/256) -- (-9/16,195/256)  --  (-9/16,-0.5) -- (-0.33,-0.5) -- cycle;

\end{tikzpicture} 
    \caption{Setting of the proof of Lemma \ref{l:iteration}.}
    \label{fig:Lemma4.1}
\end{figure}

Consequently \eqref{cond_x_2} also holds for every $x \in B_{1/8}(z)$. We then proceed once again with the proof of Lemma \ref{l:insomeball}, and this time we get $v_e \leq 1-\delta_2$ in $B_{1/16}(z)$ for some $\delta_2>0$ (smaller than $\delta_1$). Since $z$ is an arbitrary point in $B_{7/16}$, we conclude that $v_e \leq 1-\delta_2$ in $B_{1/2}$. The proof follows up to renaming $\delta_2$ as $\delta$.

\end{proof}

\section{First stage: the iterative reduction of the norm of the gradient}

\label{s:stage1}

In this section we carry out the first stage toward the proof of Theorem \ref{t:main}, as described in Section \ref{s:strategy}.

Our starting point is a properly scaled solution $u$  to $(-\Delta_p)^s u = 0$, so that  assumptions \textit{(1)--(4)} of Lemma \ref{l:iteration} are satisfied.
The main result of the section is the following. 

\begin{lemma} \label{l:mainlemma5}
For any $r,\mu >0$ small, there exist $\delta, \eps_1>0$ (small) and $K_0>0$ (large) depending on $r,\,\mu, \, d,\, s,\, p$ such that if the following conditions hold
    \begin{enumerate}
    \item $u$ satisfies \eqref{eq:fracplap} in $B_{2^{K_0+1}}$,
    \item $u(0) = 0$,
    \item  $\|\nabla u\|_{L^{\infty}(B_{2^{n}})} \leq (1-\delta)^{-n} ,$ for $n=0,1,\dots,K_0$, 
    \item $\Tail_{p-1,sp+1}(u,2^{K_0}) \leq \eps_1$,
    \end{enumerate}
 then one of the following must hold
    \begin{itemize}
      \item 
      There is a nonnegative integer $\kmax$ such that  
      \[
      \|\nabla u\|_{L^\infty(B_{2^{-n}})} \leq (1-\delta)^n  \text{ for } n = 0,1,\dots,\kmax.
      \] 
      Moreover, there exists a unit vector $e \in S^{d-1}$ such that
       \[ |\{ x \in B_{2^{-\kmax}} : |\nabla u(x) - (1-\delta)^{-\kmax} e| \leq r  (1-\delta)^{-\kmax} \}| \geq (1-\mu) |B_{2^{-\kmax}}|.\]
    \item
    For $\alpha_0>0$ so that $2^{-\alpha_0}=(1-\delta)$ and  $C_0 = (1-\delta)^{-1}$, we have
    \[
    |\nabla u(x)| \leq C_0 |x|^{\alpha_0} \quad \text{ for }  x \in B_1.
    \]
\end{itemize}
    
\end{lemma}

The strategy is to iteratively apply Lemma \ref{l:iteration} to rescaled versions of $u$. 
We recall that the main assumption which drives the first stage of the proof is that for some $e \in S^{d-1}$,
\begin{align}\label{eq:assumption_Ar_sec5}
    |\{ x \in B_1 \, : \, \nabla u(x) \notin B_{r}(e) \}| \geq \mu|B_1|.
\end{align}
Notice that if \eqref{eq:assumption_Ar_sec5} is satisfied in all the directions $e \in S^{d-1}$, then applying Lemma \ref{l:iteration}, we obtain an improvement of oscillation for $|\nabla u|$, namely
\[
|\nabla u |\leq 1-\delta \quad \text{ in }B_{1/2}.
\]

Hence, as long as condition \eqref{eq:assumption_Ar_sec5} holds for every $e \in S^{d-1}$ and for the rescaling of $u$ up to a certain scale, we can improve the upper bound of $|\nabla u|$ in dyadic balls, iterating Lemma \ref{l:iteration}. This is expressed in the next result.
\begin{lemma} \label{l:iteration_n_0times}
 For any $r,\mu >0$, there exist $\delta, \eps_1>0$ (small) and $K_0>0$ (large) depending on $r,\,\mu, \, d,\, s,\, p$ such that if the following conditions hold
    \begin{enumerate}
    \item $u$ satisfies \eqref{eq:fracplap} in $B_{2^{K_0+1}}$,
    \item $u(0) = 0$,
    \item  $\|\nabla u\|_{L^{\infty}(B_{2^{n}})} \leq (1-\delta)^{-n} ,$ for $n=0,\dots,K_0$, 
    \item $\Tail_{p-1,sp+1}(u,2^{K_0}) \leq \eps_1$,
    \end{enumerate}
    and, in addition, there exists a nonnegative integer $n_0$ such that
    \begin{equation} \label{cond:A_rupton_0}
        |\{ x \in B_{2^{-n}} : |\nabla u(x) - (1-\delta)^{n} e| \geq r(1-\delta)^n \}| \geq \mu |B_{2^{-n}}| \quad \text{ for all }e \in S^{d-1} \text{ and }n=0,\dots,n_0,  
    \end{equation}
    then 
  \begin{align}\label{eq:improvement_iteration_u}
    |\nabla u| \leq (1-\delta)^n \quad \text{ in } B_{2^{-n}} \text{ for }n=0,\dots, n_0+1.
\end{align}

\end{lemma}

\begin{proof}
We prove this lemma using finite induction over $n_0$. When $n_0=0$, it follows directly from Lemma \ref{l:iteration} that $e \cdot \nabla u \leq 1-\delta$ in $B_{1/2}$ for all $e \in S^{d-1}$. This implies that $|\nabla u | \leq 1-\delta$ in $B_{1/2}$.

Assume by induction that \eqref{eq:improvement_iteration_u} holds for $n=0,\dots,n_0$ and we are going to prove that
\begin{equation} \label{eq_thesis1stage}
|\nabla u| \leq (1-\delta)^{n_0+1} \quad \text{ in } B_{2^{-n_0-1}}.
\end{equation}
Consider the rescaling of $u$ given by
\[
v(x) = (1-\delta)^{-n_0} 2^{n_0} u(2^{-n_0}x).
\]
We now check that $v$ satisfies the hypotheses of Lemma \ref{l:iteration}.

Conditions \textit{(1)}-\textit{(2)} are straightforward. Since \eqref{eq:improvement_iteration_u} holds for $n=0, \dots,n_0$, combining it with Assumption \textit{(3)} for $u$, we get, for $n = 0, \dots, K_0+n_0$,
\[
\|\nabla u\|_{L^\infty(B_{2^{K_0-n}})}\leq (1-\delta)^{n-K_0}.
\]
This, in terms of $v$ becomes
\begin{align*}
\|\nabla v\|_{L^{\infty}(B_{2^n})} = (1-\delta)^{-n_0} \| \nabla u\|_{L^{\infty}(B_{2^{n-n_0}})} \leq (1-\delta)^{-n} \quad \text{ for }n=0,\dots,n_0+K_0.
\end{align*}
Therefore, we also get
\begin{align*}
\|v\|_{L^\infty(B_{2^n})} & \leq  \left((1-\delta)^{-1}2\right)^{n} \quad \text{  for $n=0,\dots,K_0+n_0$}.
\end{align*}
The tail condition \textit{(4)} holds using \eqref{scale_tails}. Namely, for $\theta:=(1-\delta)^{-(p-1)}2^{p(1-s)-2} <1$, we have
\begin{align*}
    &{\Tail}_{p-1,sp+1}(v,2^{K_0+n_0})^{p-1} = \theta^{n_0}{\Tail}_{p-1,sp+1}(u,2^{K_0})^{p-1} \leq \theta^{n_0} \eps_1^{p-1}.
\end{align*}
Repeating the argument in \eqref{eq:bound_tail_iter}, we split ${\Tail}_{p-1,sp+1}(v,2^{K_0})^{p-1}$, in a sequence of annulus and using the estimates for $v$ above, we obtain 
\begin{align*}
    {\Tail}_{p-1,sp+1}(v,2^{K_0})^{p-1} &= \sum_{n=0}^{n_0-1} \int_{ B_{2^{n+K_0+1}} \setminus B_{2^{n+K_0}}} \frac{|v(y)|^{p-1}}{|y|^{d+sp+1}} \dd y + {\Tail}_{p-1,sp+1}(v,2^{K_0+n_0})^{p-1}  \\
    &\leq C \sum_{n=0}^{n_0-1} \left( (1-\delta)^{-(p-1)} 2^{p(1-s)-2} \right)^{K_0+n+1}  +\theta^{n_0} \eps_1^{p-1} \\
    &= C \theta^{K_0+1}\sum_{n=0}^{n_0-1} \theta^{n} +\theta^{n_0} \eps_1^{p-1} \\  
    &\leq C \frac{\theta^{K_0+1}}{1-\theta}  + \theta^{n_0}\eps_1^{p-1}\\
    &\leq \eps_1^{p-1}
\end{align*}
by possibly taking $K_0$ larger depending on $\eps_1$.
One can observe that condition \eqref{cond:A_rupton_0} for $n=n_0$, if expressed in terms of $v$, gives
\[
    |\{ x \in B_{1} : |\nabla v(x) - e| \geq r \}| \geq \mu |B_1| \quad \text{ for all }e \in S^{d-1}, 
\]
that is condition \textit{(5)}. Therefore, by Lemma \ref{l:iteration} we obtain
\[
e\cdot \nabla v \leq 1-\delta \quad \text{in }B_{1/2} \text{ for all }e \in S^{d-1}.
\]
This implies that $|\nabla v| \leq 1-\delta$ in $B_{1/2}$ which in terms of $u$ gives \eqref{eq_thesis1stage} and concludes the proof.
\end{proof}

We are now in position to prove Lemma \ref{l:mainlemma5}.
\begin{customproof}{Lemma \ref{l:mainlemma5}}
For any $r,\mu>0$ let $N$ be the minimum nonnegative integer such that \eqref{cond:A_rupton_0} does not hold.
If $\kmax$ is finite then we can apply Lemma \ref{l:iteration_n_0times} $\kmax-1$ times and obtain
\[
\|\nabla u\|_{L^\infty(B_{2^{-n}})} \leq (1-\delta)^n \quad \text{ for } n = 0,1,\dots,\kmax.
\]
The choice of $\kmax$ implies that there exists $e\in S^{d-1}$ for which
\[ |\{ x \in B_{2^{-\kmax}} : |\nabla u(x) - (1-\delta)^{-\kmax} e| \leq r  (1-\delta)^{-\kmax} \}| \geq (1-\mu) |B_{2^{-\kmax}}|. \]
If $N=\infty$ then we can apply indefinitely Lemma \ref{l:iteration_n_0times} and obtain, for all $n> 0$,
\[
\|\nabla u\|_{L^\infty(B_{2^{-n}})} \leq (1-\delta)^n ,
\]
which gives the desired result.
  
\end{customproof}

The second stage of the paper begins considering a rescaling of $u$ from Lemma \ref{l:mainlemma5}. For future reference, we collect its properties in the final result of this section.

\begin{corollary} \label{c:end_of_stage1}
    Let $u$ be a function as in Lemma \ref{l:mainlemma5} and assume $N<\infty$. Let us consider the rescaling
\[
\bar u(x) := 2^\kmax (1-\delta)^{-\kmax}  u(2^{-\kmax} x) .
\]
This function satisfies
\begin{enumerate}
  \item $\bar u$ solves \eqref{eq:fracplap} in $B_{2^{\kmax+K_0+1}}$,
  \item $\bar u(0) = 0$,
   \item $\|\bar u\|_{L^\infty(B_{2^k})} \leq ((1-\delta)^{-1}2)^k$ for $k =0, 1,2,\dots,K_0+\kmax$,
  \item $\|\nabla \bar u\|_{L^\infty(B_{2^k})} \leq (1-\delta)^{-k}$ for $k = 0,1,2,\dots, K_0+\kmax$,
  \item $\Tail_{p-1,sp+1}(\bar u,1)\leq C_1$ where $C_1$ is a constant independent of $\kmax$,
  \item $|\{ x\in B_1 \,:\,  \nabla \bar u \in B_r(e)\}|\geq (1-\mu)|B_1|$.
\end{enumerate}
\end{corollary}

Condition \textit{(5)} is justified by the computations in \eqref{eq:bound_tail_iter}.

This function $\bar u$, satisfying the conditions above, is the starting point of the stage two in the proof of our main theorem.

\section{Second stage: the Ishii-Lions method}
\label{s:stage2}

In this section we deal with the second stage of the proof. We aim at proving that, if $\nabla u$ is very close to a unit vector in most of the ball $B_1$, then $\nabla u$ is close to that vector for every point in $B_{1/2}$, resulting in some form of non-degeneracy for the linearized kernel $K_u$. The following lemma is the objective of this section.

\begin{lemma} \label{l:stage2}
    Let $L \in (0,1/4]$ and $C_1>0$ be arbitrary given parameters. There exist $\mu$ and $r$ sufficiently small, depending on $L$ and $C_1$, such that the following statement holds.
        
    Assume that $u$ is a solution of \eqref{eq:fracplap} in $B_2$ with $p \in [2, \frac{2}{1-s})$, with $\| u\|_{Lip(B_1)} \leq 1$ and  $\Tail_{p-1,sp+1}(u;1) \leq C_1$. Assume further that, for some $e \in S^{d-1}$,
    \[ 
    |\{x \in B_1 : |\nabla u(x) - e| \leq r\}| \geq (1-\mu) |B_1|.
    \]
     Then
     \[
     |\nabla u(x) - e| \leq L \quad \text{ for all } x\in B_{1/2}.
     \]
\end{lemma}

We start with the following elementary lemma.
\begin{lemma} \label{l:eps_osc}
    Given any $\eps>0$, there are small values for $r>0$ and $\mu>0$ so that the following conclusion holds. Assume that $|\nabla u| \leq 1 $ in $B_1$ and for some $ e \in S^{d-1}$ we have 
    \[
        |\{x\in B_1\, : \, \nabla u(x) \in B_r(e)\}| \geq (1-\mu) |B_1|.
    \] 
Then,
\begin{align}\label{eq:Morrey12}
    \underset{x \in B_{1}}{\osc}\,(u(x) -  e \cdot x)\leq \eps.
\end{align}
\end{lemma}

\begin{proof}
    Assume the result is not true. That is, there is an $\eps > 0$ and two sequences $\mu_k \to 0$ and $r_k \to 0$, so that there are functions $u_k$ satisfying
    \begin{enumerate}
    \item $\|\nabla u_k\|_{L^\infty(B_1)} \leq 1$.
    \item $|\{x\in B_1\, : \, \nabla u_k(x)\in B_{r_k}(e)\}| \geq (1-\mu_k) |B_1|$,
    \end{enumerate}
    yet such that $\osc_{B_1} (u_k(x) - e\cdot x) > \eps$.

    Condition (1) tells us that the sequence is uniformly Lipschitz. By subtracting a constant if necessary, we can apply the Arzelà–Ascoli theorem to extract a uniformly convergent subsequence. On the other hand, (2) tells us that this sequence converges in $W^{1,1}$ to $e \cdot x$ plus a constant. Therefore, $u_k(x) - u_k(0)$ converges uniformly to $e\cdot x$ contradicting the bound on the oscillation.
\end{proof}

Lemma \ref{l:stage2} follows by combining Lemma \ref{l:eps_osc} with the following result.

\begin{lemma}\label{l:IL}
Let $0<L\leq 1/4$, and $C_1 > 0$. There exists some $\eps_0>0$ small enough so that the following statement is true.
Assume $u$ is a smooth solution of $(-\Delta_p)^s u = 0$  in $B_{2}$ for $p \in [2, \frac{2}{1-s})$,  $ \|u\|_{Lip(B_1)} \leq 1 $ and ${\Tail}_{p-1,sp+1}(u;1) \leq C_1$. Furthermore for some $e \in S^{d-1}$ we have 
    \begin{equation} \label{eq:osc_small}
        \underset{x \in B_{1}}{\osc}\,(u(x) -  e \cdot x)\leq \eps_0.
    \end{equation} 
Then, it holds that
\begin{equation} \label{eq:unifelli}
    |\nabla u - e| < L \quad \text{in }B_{1/2}.
\end{equation}
Here $\eps_0$ depends only on $L$, $C_1$, $d$, $s$ and $p$.
\end{lemma}

The purpose of the rest of this section is to prove Lemma \ref{l:IL}. We use the Ishii-Lions method, as described in Section \ref{s:strategy}. The idea is to apply the Ishii-Lions method to the function $v(x) = u(x) - e \cdot x$, which has a small oscillation in $B_1$ by assumption. The method was successfully applied to obtain Lipschitz estimates for fractional $p$-Laplacian equations in \cite{biswas2025lipschitz,biswas2025improved}. Our setting is similar, but we have to note that while $u$ is the function that solves the equation \eqref{eq:fracplap}, $v$ is the function with a small oscillation. The equation \eqref{eq:fracplap} is nondegenerate in the direction of $\nabla u$, which will be close to $e$, and will be unrelated to the derivative of the modulus of continuity like in a standard application of the Ishii-Lions method. This discrepancy with the analysis in \cite{biswas2025lipschitz,biswas2025improved} affects the computation of several estimates in this section, although many of the key ideas remain the same.

\subsection{Setting up the Ishii-Lions method}

Our objective is to obtain a Lipschitz estimate for the function $v = u - e \cdot x$ in $B_{1/2}$, with an arbitrarily small Lipschitz constant $L$.
\begin{equation} \label{eq:thesis2}
    \|v\|_{Lip(B_{1/2})} \leq L.
\end{equation} 

We set up the Ishii-Lions method. Here, $u$ is the function satisfying the hypothesis of Lemma \ref{l:IL}. We define the function $\Phi(x,y)$ in double variables as
\[
\Phi(x, y) = u(x) - u(y) - e\cdot (x-y) - L\omega(|x - y|) - \kappa \eps_0 \,\psi(x) \quad x, y \in B_{1}.
\]

The function $\omega$ is a Lipschitz strictly-concave modulus of continuity and $\kappa>0$ is a parameter to be chosen soon. The function $\psi$ satisfies $\psi(x) = \psi_0(x)^m$, where $\psi_0$ is a nonnegative smooth function that is equal to zero in $\overline{{B}}_{1/2}$ and is strictly positive in $\overline{B}_1 \setminus \overline{{B}}_{1/2}$.  The exponent $m \in \mathbb{N}$ is a large integer to be chosen later. The parameter $\eps_0$ is the one from \eqref{eq:osc_small}.

The purpose of the choice $\psi = \psi_0^m$ is so that we can easily verify the following inequality, which will be useful later. 
There exists some constant $C$ so that
\begin{equation} \label{eq:grad_eta}
    |\nabla \psi(x)| \leq C \psi(x)^{\frac{m-1}{m}} \quad \text{ for all } x \in B_1.
\end{equation}
The inequality \eqref{eq:grad_eta} follows readily by setting $C = m \max |\nabla \psi_0|$ and noticing that $\nabla \psi_0 = \frac 1m \psi^{1/m-1} \,\nabla \psi$.

Let us describe the modulus of continuity $\omega$. We take $\omega : [0,2] \to [0,\infty)$ defined as
\begin{equation} \label{def:omega}
    \omega(r) = r + \frac{r}{20 \log (r/4)}.
\end{equation}

Note that
\begin{align*}
    \omega'(r) &= 1 + \frac{1}{20 \log(r/4)} - \frac{1}{20 (\log^2(r/4))} \in (1/4,1) \quad \text{ for } r \in (0,2], \\
    \omega''(r) &= -\frac{1}{20 r (\log^2(r/4))} + \frac{1}{10 r (\log^3(r/4))}.
\end{align*}

For small values of $r$, for example $r \in (0,1/2)$, we may use the simpler expression $\omega''(r) \approx -1 / (r \log^2(r))$. By this, we mean that there are universal constants $c, C > 0$ such that
\begin{equation} \label{eq:omega-2nd-derivative}
-\frac{C}{r (\log(r))^2} \leq \omega''(r) \leq -\frac{c}{r (\log(r))^2} \quad \text{ for } r \in (0,1/2].
\end{equation}

Our objective is to show that $v$ has $L\omega$ as a modulus of continuity in $B_{1/2}$. To achieve that, we use the function $\Phi$ defined above. We will prove that $\Phi(x,y) \leq 0$ for all $x,y \in B_1$. Since $\psi$ is zero in $B_{1/2}$, it implies the desired modulus of continuity in $B_{1/2}$.

We argue by contradiction, assuming that there are $\bar x$, $\bar y$  such that
\begin{equation} \label{contrad1}
\Phi(\bar x, \bar y) = \sup_{(x,y) \in B_{1} \times B_{1}} \Phi(x,y) >0,
\end{equation}
i.e. 
\[
u(\bar x)-u(\bar y) - e\cdot (\bar x-\bar y) > L\omega(|\bar x-\bar y|)+\kappa\eps_0 \,\psi(|\bar x|).
\]
From the assumption \eqref{eq:osc_small}, we know that the left-hand side is bounded by $\eps_0$. The two terms on the right-hand side are nonnegative.  We will choose $\kappa$ large enough so that the second term forces $\bar x$ to be in $B_{5/8}$. We also have $L \omega(|\bar x - \bar y|) \leq \eps_0$, which will force $|\bar x - \bar y|$ to be small if we take $\eps_0$ small enough depending on $L$. For $\eps_0>0$ sufficiently small, we can ensure that $|\bar x - \bar y| < 1/8$, and thus $\bar x$ and $\bar y$ both belong to $B_{3/4}$.

Let us call $\bar a = \bar x - \bar y$. We have $L \omega(|\bar a|) \leq \eps_0$. As we said, by taking $\eps_0$ sufficiently small depending on $L$, we can make $|\bar a|$ arbitrarily small. The small values of $|\bar a|$ are the focus of the upcoming estimates below.

The idea is to apply the Ishii-Lions argument to $v$, with $L\in (0,1/4]$ arbitrarily small but fixed, and obtain a contradiction when \eqref{contrad1} holds and $\eps_0$ is sufficiently small in \eqref{eq:osc_small}. For this construction, we borrow some ideas from \cite{biswas2025lipschitz,biswas2025improved}, particularly the decomposition of the domain of integration.  
 
We define the following test function.
\begin{equation} \label{eq:def_phi1}
    \phi(x, y) = e\cdot (x-y)+ L \omega(|x - y|) + \kappa \eps_0 \psi(x).
\end{equation} 

We will analyze the values of $(-\Delta_p)^s u$ at $\bar x$ and $\bar y$. We use a decomposition of the domain of integration in a similar manner as in \cite{biswas2025lipschitz,biswas2025improved}.

\subsection{The cone of concavity}

For any given $a \in B_{1/2}$, let $\delta_0(a)$ be given by
\begin{equation} \label{def:delta0}
\delta_0(a) = -c_0 / \log(|a|)
\end{equation}
for some small universal constant $c_0$ to be determined. We define the cone $\C(a)$ as
\begin{equation} \label{def:cone}
\C(a)=\left\{ z \in B_{|a|/2} : |\langle a, z \rangle| \geq \sqrt{1 - \delta_0(a)^2} \, |a|\,|z| \right\}.
\end{equation}
Note that $z \in \C(a)$ whenever the component of $z$ that is orthogonal to $a$ has norm at most $\delta_0|z|$. The meaning of this cone is that in this region we have a good control of the second derivative of $\omega$, and consequently also of $\phi$ in \eqref{eq:def_phi1}. This is made precise in the next lemma.

We use the notation $\delta^2 f$ to denote the second increment of a generic function $f$ as
\[
\delta^2 f(x, z) = f(x) + \nabla f(x) \cdot z - f(x + z).
\]
In particular, when we write $\delta^2 \omega(|\cdot|)(a,z)$, in the next lemma, we mean
\[ \delta^2 \omega(|\cdot|)(a,z) = \omega(|a|) + \omega'(|a|) \, \frac {a}{|a|} \cdot z - \omega(|a+z|). \]
\begin{lemma} \label{l:2nd-increment-omega}
Let $a \in B_{1/2}$ and $z \in \C(a)$. Assume that $c_0$ is some small universal constant. The following estimate holds
\begin{equation} \label{eq:2nd-increment-omega} \frac{c}{|a| \log^2 |a|} |z|^2 \leq \delta^2 \omega(|\cdot|)(a,z) \leq \frac{C}{|a| \log^2 |a|} |z|^2,
\end{equation}
where $c$ and $C$ are universal constants.
\end{lemma}

\begin{proof}
When $a$ and $z$ are collinear, the inequality \eqref{eq:2nd-increment-omega} follows from \eqref{eq:omega-2nd-derivative} and Taylor's expansion. Note that the values of $\omega''(r)$ are all comparable in the range $r \in (|a|/2,3|a|/2)$, which is the range of $|a+z|$ and $|a-z|$ when $|z| \leq |a|/2$.

For the general case, we write $z = z_1 + z_2$ where $z_1$ is collinear to $a$ and $z_2$ is orthogonal to $a$. We note that $\nabla_z \omega(|a+z|)$ is perpendicular to $z_2$ at $z = z_1$. Therefore,
\[ 0 \leq \omega(|a+z_1+z_2|) - \omega(|a+z_1|) \leq \frac{C}{|a|} |z_2|^2,\] 
where $C$ depends on $|\nabla \omega(|a+z_1|)| \approx 1$. Thus, the inequality \eqref{eq:2nd-increment-omega} holds in this case as well provided that $|z_2| \lesssim \delta_0(a) |z|$. This condition characterizes the cone $\C(a)$.
\end{proof}

The statement of the next corollary is similar to \cite[Lemma 2.2]{biswas2025lipschitz}.

\begin{corollary} \label{cor:2nd-increment-phi}
Let $\phi$ be the function from \eqref{eq:def_phi1}. Let $x,y \in B_{3/4}$ and $z \in \C(x-y)$. Assume that $|a| := |x-y|$ is sufficiently small. Then, the following estimates hold for $z \in \C(a)$, 
\begin{equation} \label{eq:2nd-increment-phi} 
    \begin{aligned}
    \frac{cL}{|a| \log^2 |a|} |z|^2 &\leq \delta^2 \phi(x,\cdot)(y,z) \leq \frac{CL}{|a| \log^2 |a|} |z|^2,\\
    \frac{cL}{|a| \log^2 |a|} |z|^2 &\leq \delta^2 \phi(\cdot,y)(x,z) \leq \frac{CL}{|a| \log^2 |a|} |z|^2,
    \end{aligned}
\end{equation}
where $c$ and $C$ are universal constants.
\end{corollary}

\begin{proof}
The function $\phi$ has three terms. Let us evaluate the second increment of each of them. The first term is linear, so its second increment is zero. The second term is $L \omega(|x-y|)$, and we can apply Lemma \ref{l:2nd-increment-omega} to obtain the desired estimate. The third term is $\bar\kappa \eps_0 \psi(x)$, which only affects the second increment in $x$.

The function $\psi$ is smooth, so its second increment is bounded by the maximum of its second derivative. In this case, it gives us
\[ \left| \eps_0 \kappa\, \psi(x) + \eps_0 \kappa \nabla \psi(x) \cdot z - \eps_0 \kappa \,\psi(x+z) \right| \leq C \eps_0 \kappa |z|^2.\]
Recall that $|z| < |a|/2$ because $z \in \C(a)$ and $\eps_0$ is small. Then, for $\eps_0$ sufficiently small (and thus also $|a|$), we can absorb this term into the estimate for the second increment of $L \omega(|x-y|)$.
\end{proof}

Corollary \ref{cor:2nd-increment-phi} gives us a positive bound for the second increment of $\phi$ in the cone $\C(a)$, which is the set of favorable directions. If we consider second order increments in arbitrary directions $z \in B_{|a|/2}$, we cannot expect a positive lower bound. The most we are going to estimate is based on the gross upper bound $|D^2 \phi(x,y)| \leq CL / |a|$. Thus, for $|z| < |a|/2$, we have
\begin{equation} \label{eq:2nd-increment-phi-rough}
    |\delta^2 \phi(x,\cdot)(y,z)| \leq C L \frac{|z|^2}{|a|}.
\end{equation}

\subsection{Decomposing the domain of integration}

Let $\bar x$ and $\bar y$ be the points where the maximum in \eqref{contrad1} is achieved.  We have $\bar x, \bar y \in B_{3/4}$ and $|\bar a| = |\bar x - \bar y|$ is small if we take $\eps_0$ sufficiently small depending on $L$. We have
\begin{equation} \label{eq:splapforu}
    (-\Delta_p)^s u (\bar x) - (-\Delta_p)^s u (\bar y) = 0.
\end{equation}

We split the domain of integration in \eqref{eq:splapforu}. Given $D \subset \R^d$, we introduce the notation  
\begin{equation}\label{eq:notation_Integral}
    \begin{aligned}
        \mathcal{L}[D] w(x) :=& \, \int_{D}J_p(w(x)-w(x+z))|z|^{-d-sp}  \dd  z \\ 
    =& \,\int_{D} |w(x)-w(x+z)|^{p-2} (w(x)-w(x+z)) |z|^{-d-sp}  \dd  z.
    \end{aligned}
\end{equation}

Let $\C = \C(\bar a)$ be the cone of directions defined in \eqref{def:cone}. Let also
\[
\mathcal{D}_1=B_{\delta_1(\bar a) \, |\bar a|} \cap \mathcal{C}^c, \qquad \mathcal{D}_2 = B_{\frac{1}{16}} \setminus \left( \mathcal{D}_1 \cup \mathcal{C} \right).
\]
Here, $\delta_1(\bar a) = c_1 |\log |\bar a||^{-q}$ for some appropriate exponent $q$ to be chosen later and $c_1$ is a small universal constant.

Using the notation above we write \eqref{eq:splapforu} as 
\begin{equation} \label{contrad2}
\begin{aligned}
   &\underbrace{\mathcal{L}[\mathcal{C}] u(\bar x) -\mathcal{L}[\mathcal{C}] u(\bar y)}_{I_1} +  \underbrace{\mathcal{L}[\mathcal{D}_1] u(\bar x) -\mathcal{L}[\mathcal{D}_1] u(\bar y)}_{I_2} \\
   +& \underbrace{\mathcal{L}[\mathcal{D}_2] u(\bar x) -\mathcal{L}[\mathcal{D}_2] u(\bar y)}_{I_3}+ \underbrace{\mathcal{L}[B_{\frac{1}{16}}^c] u(\bar x) -\mathcal{L}[B_{\frac{1}{16}}^c] u(\bar y)}_{I_4} = 0 .
\end{aligned}
\end{equation}
We aim at proving that $I_1$ must be very positive while  $I_2, I_3, I_4$ cannot be too negative and derive a contradiction from \eqref{contrad2}. For this purpose, we present some technical estimates which will lead us to the desired conclusion.

We start with the following estimate on $\C$. The following lemma is similar to \cite[Lemma 3.1]{biswas2025lipschitz}.

\begin{lemma}[Estimate on $\mathcal{C}$]\label{l:cone}
For every $L \in (0,1/4]$  there exists $\eps_0>0$ sufficiently small and a positive constant $c$ such that 
if $\bar x$ and $\bar y$ satisfy~\eqref{contrad1} and that $\delta_0 = \delta_0(a)$ is the one from \eqref{def:delta0}, then 
\begin{align}\label{eq:IL-I1}
     I_1 \geq c L \, \delta_0^{d+p-1} \, |\bar a|^{p(1-s)-1}.
\end{align}
Here, $c$ depends on $d$, $s$ and $p$.
\end{lemma}

\begin{proof}  
We estimate the two terms $\mathcal{L}[\mathcal{C}] u(\bar x)$ and $-\mathcal{L}[\mathcal{C}] u(\bar y)$, obtaining similar bounds for both. Let us focus on the analysis for $\mathcal{L}[\mathcal{C}] u(\bar x)$.

Recall that we use the notation $\delta^2$ to denote the second-order incremental quotient applied to a function. 

Let us define  $\ell(z)= - \nabla_x \phi (\bar x, \bar y) \cdot z$ and note that $\mathcal{L}[\mathcal{C}] \ell(\bar x)=0$ since $\mathcal{C}$ is symmetric.

From \eqref{contrad1}, we know that 
\[
\Phi(\bar x+z, \bar y ) \leq \Phi (\bar x, \bar y) \quad \text{for }z \in \mathcal{C},
\]
which implies that
\[
u(\bar x+z) - u(\bar x) \leq \phi(\bar x + z, \bar y) - \phi(\bar x, \bar y) \quad \text{for }z \in \mathcal{C},
\]
according to the definition of $\phi$ in \eqref{eq:def_phi1}. From the monotonicity of $J_p$, we obtain
\begin{align*}
    \mathcal{L}[\mathcal{C}] u(\bar x) &\geq \mathcal{L}[\mathcal{C}] \phi(\cdot ,\bar y)(\bar x)\\
    & =  \mathcal{L}[\mathcal{C}] \phi(\cdot ,\bar y)(\bar x)- \mathcal{L}[\mathcal{C}] \ell (\bar x) \\
    & = (p-1)\int_{\mathcal{C}} \int_0^1 |\ell(z)+t\delta^2 \phi(\cdot,\bar y)(\bar x,z)|^{p-2} \delta^2 \phi(\cdot,\bar y)(\bar x,z)  \dd  t |z|^{-d-sp}  \dd  z,
\end{align*}
where we used Lemma \ref{l:appendix}.

Let us now introduce the set where $|\ell(z)| \geq \frac{\delta_0}4 |z|$.
\[
\C_1=\left\{ z \in \R^d :   |\langle \nabla_x \phi(\bar x, \bar y), z \rangle| \geq \frac{\delta_0}4 \, |z| \right\}.
\]
The important observation to make here is that $\nabla_x \phi(\bar x, \bar y)$ is not too far from $e$. Indeed, differentiating \eqref{eq:def_phi1} term by term, we get
\[ \nabla_x \phi(x, y) = e + L \omega'(|x-y|) \frac{x-y}{|x-y|} + \kappa \eps_0 \nabla \psi(x). \]
For $\eps_0$ sufficiently small, we have $|x-y|$ small as well. The value of $\omega'(|x-y|)$ is approximately equal to $\omega'(0)=1$. The value of $L$ is chosen in the interval $(0,1/4]$. Therefore $|\nabla_x \phi(x, y) - e|$ cannot be larger than $1/4$ plus an arbitrarily small contribution of the other terms. In particular, choosing $\eps_0$ small, we can guarantee that $|\nabla_x \phi(x, y)| \in (1/2,3/2)$. Thus, the set $\C_1$ is a cone whose complement has width at most $\delta_0/2$. It must intersect with at least half of the cone $\C(\bar a)$.

Using the strict positivity of  $\delta^2 \phi(\cdot,\bar y)(\bar x,z)$ for $z \in \mathcal{C}$, given by Corollary \ref{cor:2nd-increment-phi}, together with Lemma \ref{l:appendix} we can further estimate $ \mathcal{L}[\mathcal{C}] u(\bar x)$ as
\begin{align*}
     \mathcal{L}[\C] u(\bar x) \geq c \int_{\C \cap \C_1} |\ell(z)|^{p-2}\delta^2 \phi(\cdot,\bar y)(\bar x,z) |z|^{-d-sp}  \dd  z.
\end{align*}
From the definition of $\C_1$, we have $|\ell (z)|  \geq \frac{\delta_0}{4} |z|$. Hence, we proved
\begin{align} \label{eq:I_1bound1}
\mathcal{L}[\mathcal{C}] u(\bar x) &\geq  c \,   \delta_0^{p-2}\int_{\C \cap \C_1} |z|^{p-2 - d- sp} \delta^2 \phi(\cdot, \bar y)(\bar x, z)  \dd  z,
\end{align}
for some $c >0$ depending only on $d,s,p$. We use the the bounds for $\delta^2 \phi(\cdot, \bar y)(\bar x, z)$ of Corollary \ref{cor:2nd-increment-phi} and get
    \begin{align*}
    I_1 &\geq c \, L(|\bar a| \log^2|\bar a|)^{-1} \delta_0^{p-2} \int_{\C \cap \C_1} |z|^{p-d-sp}  \dd  z \\
    &\geq c \, L(|\bar a| \log^2|\bar a|)^{-1} \delta_0^{p-2} \delta_0^{d-1} |\bar a|^{p(1-s)} \\
    &= c L \, \delta_0^{d+p-1} \, |\bar a|^{p(1-s)-1},
    \end{align*}
    as intended. An analogous estimate holds for $-\mathcal{L}[\mathcal{C}] u(\bar y)$.

\end{proof}

We now proceed to estimate $I_2 = \mathcal{L}[\mathcal{D}_1] u(\bar x) -\mathcal{L}[\mathcal{D}_1] u(\bar y)$. Recall that $\delta_1$ is the parameter in the definition of $\mathcal D_1$. The following lemma is similar to \cite[Lemma 3.2]{biswas2025improved}.

\begin{lemma}[Estimate on $\mathcal {D}_1$]\label{l:D1}
 Assume $(\bar x, \bar y)$ satisfies~\eqref{contrad1}. Then, there exists $C > 0$ such that
\begin{align*}
I_2 \geq & -C L \delta_1^{p(1 - s)}  |\bar a|^{p(1 - s) - 1},
\end{align*}
for $L \in (0,1/4]$ and $\eps_0$ sufficiently small.
The constant $C $ depends on $d, \,s, \,p$, but not on $\bar a$.
\end{lemma}
\begin{proof}
As before, from \eqref{contrad1} and the monotonicity of $J_p$, we have
\[
\mathcal{L}[\mathcal{D}_1] u (\bar x) \geq (p-1) \int_{\mathcal{D}_1} \int_0^1 |\ell(z)+t\delta^2 \phi(\,\cdot\,,\bar y)(\bar x,z)|^{p-2} \delta^2 \phi(\cdot,\bar y)(\bar x,z)  \dd  t |z|^{-d-sp}  \dd  z.
\]

To estimate $\delta^2 \phi(\cdot,\bar y)(\bar x,z)$, we use the crude upper bound \eqref{eq:2nd-increment-phi-rough}. Note also that for small enough $\eps_0$ and $t\in [0,1]$,
\[ |\ell(z)+t\delta^2 \phi(\,\cdot\,,\bar y)(\bar x,z)| \leq |\ell(z)| + |z| \leq 2|z|. \]
Therefore,
\begin{align*}
\mathcal{L}[\mathcal{D}_1] u (\bar x) &\geq -C \int_{\mathcal{D}_1} L |z|^{p-d-sp} \frac 1 {|\bar a|} \dd z \\
&= -C L \delta_1^{p(1-s)} |\bar a|^{p(1-s)-1}.
\end{align*}
An analogous estimate holds for $-\mathcal{L}[\mathcal{D}_1] u (\bar y)$.
\end{proof}

We now proceed to estimate $I_3 = \mathcal{L}[\mathcal{D}_2] u(\bar x) -\mathcal{L}[\mathcal{D}_2] u(\bar y)$. 

\begin{lemma}[Estimate on $\mathcal{D}_2$]\label{l:D2}
 Assume $(\bar x, \bar y)$ satisfies~\eqref{contrad1}. Then for every $L \in (0,1/4]$, there exists a constant $C > 0$ such that
\[
I_3\geq -C |\bar a|^{p(1-s)/2},
\]
where $C$ depends on $d$, $s$,  $p$ and the choice of the exponent $q$ in the definition of $\delta_1$. The constant $C$ does not depend on $\bar a$.
\end{lemma}

\begin{proof}
Let $ \hat \delta := |\bar a|^{1/2} $. Since $|\bar a|$ is small, we have $\delta_1|\bar a| < \hat \delta < \frac{1}{16}$. We divide $I_3 $ in two parts as follows
\[
I_3 = \underbrace{\mathcal{L}[\mathcal{D}_2 \cap B_{\hat \delta}] u(\bar x) -\mathcal{L}[\mathcal{D}_2  \cap B_{\hat \delta}] u(\bar y)}_{I_{1,3}} +\underbrace{\mathcal{L}[\mathcal{D}_2  \cap B_{\hat \delta}^c] u(\bar x) -\mathcal{L}[\mathcal{D}_2  \cap B_{\hat \delta}^c] u(\bar y)}_{I_{2,3}}.
\]

Define $\delta^1 f(x,z)= f(x)- f(x+z)$. Using Lemma \ref{l:appendix} we can write $I_{1,3}$ as follows
\begin{align*}
I_{1,3} &= (p-1) \int_{\mathcal{D}_2 \cap B_{\hat \delta}} \int_{0}^1 |\delta^1 u(\bar y,z) + t(\delta^1 u(\bar x,z) -\delta^1 u(\bar y,z) )|^{p-2} (\delta^1 u(\bar x,z) -\delta^1 u(\bar y,z) )  \dd  t |z|^{-d-sp}  \dd  z.
\end{align*}

Recall that $\bar x$ and $\bar y$ are the points where the maximum in \eqref{contrad1} is achieved. Thus, $\Phi(\bar x , \bar y) \geq \Phi(\bar x +z, \bar y+z)$ and
\[
\delta^1 u(\bar x,z) -\delta^1 u(\bar y,z) \geq - \bar\kappa \eps_0 \delta^1 \psi(\bar x,z),
\]
which implies
\[
I_{1,3} \geq -  C \eps \int_{\mathcal{D}_2 \cap B_{\hat \delta}} \int_{0}^1 |\delta^1 u(\bar y,z) + t(\delta^1 u(\bar x,z) -\delta^1 u(\bar y,z) )|^{p-2} \delta^1 \psi(\bar x,z)  \dd  t  |z|^{-d-sp}  \dd  z.
\]
Using the fact that $[u]_{Lip(B_1)}\leq 1$, we get
\[
 |\delta^1 u(\bar y,z) + t(\delta^1 u(\bar x,z) -\delta^1 u(\bar y,z) )| \leq 3 |z|.
\]
Now, writing $|\delta^1 \psi(\bar x,z)| \leq C \left( |z|^2 + |\nabla \psi(\bar x)||z| \right)$, we can further estimate  $I_{1,3}$ by
\begin{align*}
    I_{1,3} &\geq - C \eps_0 \left[\int_{\mathcal{D}_2\cap B_{\hat \delta}} |z|^{p-d-sp}  \dd  z
    + |\nabla \psi(\bar x)| \int_{\mathcal{D}_2 \cap B_{\hat \delta}} |z|^{p-1-d-sp}  \dd  z   \right]\\
 &\geq  - C\eps_0 \left[\int_{\delta_1|\bar a|}^{|\bar a|^{1/2}} r^{p(1-s)-1}  \dd  r
    + |\nabla \psi(\bar x)| \int_{\delta_1|\bar a|}^{|\bar a|^{1/2}}  r^{p-2-sp}  \dd  r  \right].
\end{align*}

Notice that from \eqref{contrad1} we also have that 
\[
\psi(\bar x) < \frac{1}{\kappa\eps_0} (u(\bar x)- u(\bar y) - e\cdot(\bar x - \bar y)) \leq  \frac{2|\bar a|}{\kappa \eps_0}.
\]

Recall from \eqref{eq:grad_eta} that $|\nabla \psi(\bar x)| \leq C \psi(\bar x)^{\frac{m-1}m}$. Therefore, $|\nabla \psi(\bar x)| \leq C (|\bar a|/\eps_0)^{\frac{m-1}m}$.

Continuing the estimate for $I_{1,3}$, we get
\begin{align*}
    I_{1,3} &\geq - C\eps_0 \left[ \int_{\delta_1|\bar a|}^{|\bar a|^{1/2}} r^{p(1-s)-1}  \dd  r
    + \eps_0^{-\frac{m-1}{m}} |\bar a|^{\frac{m-1}{m}} \int_{\delta_1|\bar a|}^{|\bar a|^{1/2}}  r^{p-2-sp}  \dd  r \right] \\
    &\geq - C \left[ \eps_0 |\bar a|^{p(1-s)/2} + \eps_0^{\frac{1}{m}} |\bar a|^{\frac{m-1}{m}} \int_{\delta_1|\bar a|}^{|\bar a|^{1/2}}  r^{p-2-sp}  \dd  r  \right].
    \end{align*}

For $I_{2,3}$, we proceed using Lemma \ref{l:appendix} as before, but instead use the Lipchitz continuity of $u$ to get $|\delta^1 u(\bar x,z) -\delta^1 u(\bar y,z)| \leq 2|\bar a|$. Therefore,
\begin{align*}
    I_{2,3} &\geq - C \int_{\mathcal{D}_2 \cap B_{\hat \delta}^c} \int_{0}^1 |\delta^1 u(\bar y,z) + t(\delta^1 u(\bar x,z) -\delta^1 u(\bar y,z) )|^{p-2} |\delta^1 u(\bar x,z) -\delta^1 u(\bar y,z) | \dd t |z|^{-d-sp}  \dd  z  \\
    &\geq - C |\bar a|\int_{\mathcal{D}_2 \cap B_{\hat \delta}^c} |z|^{p-2-d-sp}  \dd  z \\
    &\geq -  C|\bar a| \int_{\hat \delta}^{\frac{1}{16}} r^{p-3-sp}  \dd  r
    \\
    &\geq -  C\frac{|\bar a|^{1+\frac 12(p-2-sp)}}{sp-p+2} = -C |\bar a|^{p(1-s)/2}. 
\end{align*}
Combining the estimates, we get
\begin{equation}\label{eq:I3_three_terms}
    I_3 \geq - C \left[\eps_0^{\frac{1}{m}} |\bar a|^{\frac{m-1}{m}} \int_{\delta_1|\bar a|}^{|\bar a|^{1/2}} r^{p-2-sp}  \dd  r + |\bar a|^{p(1-s)/2} \right].
\end{equation}

To estimate the integral inside the bracketed expression, we split the proof into three cases.

\vspace{0.1in}

\fbox{\textit{Case 1:} $p\in [2, \frac{1}{1-s})$}
\medskip

Let us choose $m$ sufficiently large so that $1/m < p(1-s)/2$. Then, using that $p-2-sp<-1$, we compute
\[
|\bar a|^\frac{m-1}{m}\int_{\delta_1|\bar a|}^{|\bar a|^{1/2}}r^{p-2-sp} \dd r\leq C \delta_1^{p(1-s)-1}|\bar a|^{\frac{m-1}{m}+p(1-s)-1} \leq C |\bar a|^{p(1-s)/2}.
\]

We observe that both terms inside the parenthesis in \eqref{eq:I3_three_terms} are bounded by $C |\bar a|^{p(1-s)/2}$ for some constant $C$. Thus, we finally get,
\begin{align*}
    I_3\geq &\, -C |\bar a|^{p(1-s)/2}.
\end{align*}

\vspace{0.1in}

\fbox{\textit{Case 2:} $p = \frac{1}{1-s}$}
\medskip

In this case, $p-2-sp = -1$ and we have
\[ |\bar a|^\frac{m-1}{m}\int_{\delta_1|\bar a|}^{|\bar a|^{1/2}}r^{p-2-sp} \dd r = |\bar a|^{(m-1)/m} \left( -\frac 12 \log |\bar a| - \log \delta_1 \right) \leq C |\bar a|^{\frac{m-2}m}.\]

Let us pick $m$ sufficiently large so that $(m-2)/m > p(1-s)/2$ and we conclude the desired inequality.

\vspace{0.1in}

\fbox{\textit{Case 3:} $p\in (\frac{1}{1-s},\frac{2}{1-s})$}
\medskip

In this case, $p-2-sp > -1$ and we get from \eqref{eq:I3_three_terms} that
\[ I_3 \geq -C \left( \eps_0^{\frac{1}{m}} |\bar a|^{\frac{m-1}{m}+ (p-1-sp)/2} + |\bar a|^{p(1-s)/2} \right). \]

Then, pick $m > 2$ so that $(m-1)/m > 1/2$ and $\frac{m-1}{m}+ \frac{1}{2}(p-1-sp) >  \frac{1}{2}p(1-s)$. We conclude that

\[ I_3 \geq -C |\bar a|^{p(1-s)/2}.\]
\end{proof}

We now bound the term concerning the tail. We state the result in a slight more general fashion since it will be useful later in the paper.

\begin{lemma} \label{l:est_tail_difference}
Let $0<r<R$ and assume $u \in L_{sp}^{p-1}(\R^d)$ to be bounded in $B_{R+r}$ and Lipschitz in $B_r$. Let $F(x)$ be given by
\[ F(x) := \int_{B^c_{R}(x)}\frac{|u(x)-u(z)|^{p-2}}{|x-z|^{d+sp}}(u(x)-u(z))\dd z \]
Then $F$ is Lipschitz in $B_r$ with $[F]_{Lip(B_r)} \leq C$ for some constant $C$ that depends on $R,\, r, \, {\Tail}_{p-1,sp+1}(u,R),$ $ \, \|u\|_{L^{\infty}(B_{R+r})}, \, [u]_{Lip(B_r)}, \,d, \, s$ and $p$.
\end{lemma}

\begin{proof}
    Let us first assume that $|x-y| < R/2$. We later extend the estimate to arbitrary pairs of points $x,y \in B_r$ by connecting them with intermediate points that are close to each other.

    We estimate $|F(x)-F(y)|$ by splitting it into three parts. Using of the notation $J_p(t)=|t|^{p-2}t$, we have
    \begin{align*}
        |F(x)-F(y)| =& \left| \int_{B^c_{R}(x)}J_p(u(x)-u(z))|x-z|^{-d-sp}\dd z - \int_{B^c_{R}(y)}J_p(u(y)-u(z))|y-z|^{-d-sp}\dd z \right| \\
        \leq &\,\left|\int_{B^c_{R}(x)}J_p(u(x)-u(z))\left( |x-z|^{-d-sp}-|y-z|^{-d-sp}\right)\dd z\right|\\
        +&\,\left|\int_{B^c_{R}( x)}J_p(u(x)-u(z))|y-z|^{-d-sp}\dd z-\int_{B^c_{R}(y)}J_p(u(x)-u(z))|y-z|^{-d-sp}\dd z\right|\\
        +&\, \left|\int_{B^c_{R}(y)}\big(J_p(u(x)-u(z))-J_p(u(y)-u(z))\big)|y-z|^{-d-sp}\dd z \right|\\
        =&:\mathcal{J}_1+\mathcal{J}_2+\mathcal{J}_3.
    \end{align*}
     When $|x-z| > R$ and $|x-y|<r$, for some constant $C=C(r,R)$, we have
    \[
    \left||x-z|^{-d-sp}-|y-z|^{-d-sp}\right|\leq C|x-z|^{-d-sp-1}|x-y|.
    \]
    We bound the first term as
\[
\mathcal{J}_1 \leq C|x-y| \left( R^{-sp-1} \| u \|_{L^{\infty}(B_r)}^{p-1} + {\Tail}_{p-1,sp+1}(u;x,R)^{p-1}\right).
\]

To bound $\mathcal{J}_2$, note that we need only to integrate in the symmetric difference of the balls $B_R(x)$ and $B_R(y)$, which we denote by $B_R(x) \,\Delta \, B_{R}(y)$. Since $|B_R(x)\, \Delta \, B_{R}(y)| \leq |x-y||\p B_{R}|$ and for $z \in B_R(x) \Delta B_{R}(y)$, we have $|y-z| \geq R-|x-y| \geq R/2$, we estimate $\mathcal{J}_2$ as follows
\[
\mathcal{J}_2 \leq C\frac{R^{d-1}}{(R/2)^{d+sp}}\|u\|_{L^{\infty}(B_{R+r})}^{p-1}|x-y|.
\]

 Regarding $\mathcal{J}_3$, combining Lipschitz continuity in $B_r$ with the elementary estimate 
$J_p(a+b)-J_p(a) \leq (p-1)|b| \left(|a| + |b| \right)^{p-2}$ we obtain 
\begin{align*}
\left| J_p\left(u(x) - u(z)\right) - J_p\left(u(y) - u(z)\right) \right| 
&\le C \left( \|u \|_{L^{\infty}(B_r)}^{p-2} + |u(z)|^{p-2}\right) \left| u(x) - u(y) \right| \\
&\le C  [u]_{Lip(B_r)}\left( \|u \|_{L^{\infty}(B_r)}^{p-2} + |u(z)|^{p-2}\right)^{p-2} |x-y|.
\end{align*}
This implies the following bound on  $\mathcal{J}_3$
\begin{align*}
\mathcal{J}_3 &\leq C  [u]_{Lip(B_r)}|x-y| \left(R^{-sp}\|u \|_{L^{\infty}(B_r)}^{p-2}  + {\Tail}_{p-2,sp}(u;y,R)^{p-2} \right) \\&\leq 
C  [u]_{Lip(B_r)}|x-y| \left(R^{-sp}\|u \|_{L^{\infty}(B_r)}^{p-2} + R^{\frac{p(1-s)-2}{p-1}}{\Tail}_{p-1,sp+1}(u;y,R)^{p-2} \right),
\end{align*}
where in the last inequality we used \eqref{Tail_functions}.
Combining the esimates for $\mathcal{J}_1,\mathcal{J}_2$ and $\mathcal{J}_3$ we get, for $|x-y| < R/4$,
\begin{align*}
    |F(x)-F(y)| &\leq C\left( R^{-sp-1} \| u \|_{L^{\infty}(B_r)}^{p-1} + {\Tail}_{p-1,sp+1}(u;x,R)^{p-1} + R^{-1-sp}\|u\|_{L^{\infty}(B_{R+r})}^{p-1} \right. \\ & \quad \left. + [u]_{Lip(B_r)}\left( R^{-sp}\|u \|_{L^{\infty}(B_r)}^{p-2} + R^{\frac{p(1-s)-2}{p-1}}{\Tail}_{p-1,sp+1}(u;y,R)^{p-2} \right) \right)|x-y|
    \intertext{and using Lemma \ref{l:move_center_tail},}
    &\leq C\left(r,R,\| u\|_{L^{\infty}(B_{R+r})},[u]_{Lip(B_r)},{\Tail}_{p-1,sp+1}(u;R) \right) |x-y|.
\end{align*}

If $x,y \in B_r$ are two points so that $|x-y|>R/2$, since we know $|x-y|<2R$, we connect them with three intermediate points $z_1$, $z_2$ up to $z_3$ in $B_r$ so that $|x-z_1|<R/2$, $|z_i - z_{i+1}|<R/2$ and $|z_3 - y|<R/2$. For example, we may take for $i=1,2,3$
\begin{align*}
z_i &= \frac{4-i}4 x + \frac i4 y.
\end{align*}

The estimate above applies to $|u(x)-u(z_1)|$, $|u(z_{i+1})-u(z_i)|$ and $|u(y)-u(z_3)|$. We conclude the proof applying the triangle inequality.
\end{proof}

The estimate for the term $I_4$ in \eqref{contrad2} is a consequence from the previous result since $ \|u\|_{Lip(B_1)} \leq 1 $ and ${\Tail}_{p-1,sp+1}(u;1) \leq C_1$.

\begin{corollary}[Estimate on $B_{1/16}^c$]\label{c:I4}
 We have
\[
|I_4| \leq C  |\bar a|,
\]
where the constant $C$ depends on $C_1, \,d, \, s,\, p$.
\end{corollary}

\subsection{Conclusion of the proof of Lemma~\ref{l:IL}}

Finally, we collect all the inequalities and obtain the contradiction. We finish with the proof of the main lemma in this section.

\begin{customproof}{Lemma~\ref{l:IL}}
Recall that by following Ishii-Lions method, we need to reach a contradiction in \eqref{contrad2} assuming \eqref{contrad1} and taking $\eps_0$ sufficiently small.

The identity \eqref{contrad2} tells us that $I_1+I_2+I_3+I_4=0$. Let us recall the estimates we have obtained for each term in Lemmas \ref{l:cone}, \ref{l:D1}, \ref{l:D2} and Corollary \ref{c:I4}.
\begin{align*}
I_1 &\geq c L \, \delta_0^{d+p-1} \, |\bar a|^{p(1-s)-1}, \\
I_2 &\geq -C L \delta_1^{p(1 -s)}  |\bar a|^{p(1 - s) - 1}, \\
I_3 &\geq -C |\bar a|^{p(1-s)/2}, \\
I_4 &\geq -C |\bar a|.
\end{align*}
The term $I_1$ is strictly positive. We see that for a good choice of the exponent $q$ in the definition of $\delta_1$, and taking $\eps_0$ sufficiently small, the terms $I_2$, $I_3$ and $I_4$ are small compared to $I_1$ implying that $I_1+I_2+I_3+I_4>0$, which is a contradiction with \eqref{contrad2}.

We choose $q$ sufficiently large so that $q \, p (1-s) > d+p-1$. This makes $-I_2 \ll I_1$ for $|\bar a|$ sufficiently small.

We note that since we assume $p < 2/(1-s)$, we have $p(1-s) - 1 < p(1-s)/2$. Thus, $-I_3 \ll I_1$ for $|\bar a|$ sufficiently small.

Finally, since $p(1-s)-1 < 1$, we have $-I_4 \ll I_1$ for $|\bar a|$ sufficiently small.

Therefore, we get that $I_1+I_2+I_3+I_4>0$, which is a contradiction with \eqref{contrad2}.

Recall that $|\bar a|$ can be made arbitrarily small by taking $\eps_0$ sufficiently small. This concludes the proof.
\end{customproof}

At this point, we may combine the results of Section \ref{s:stage1} with Section \ref{s:stage2} to obtain the following.

\begin{corollary} \label{c:stages1and2}
 Let $u$ satisfy the conditions of Lemma \ref{l:mainlemma5}. Then, there exists a small constant $\delta>0$  so that one of the following must hold.
 \begin{itemize}
    \item  There is a positive integer $N$ such that $|\nabla u(x)| < (1-\delta)^k$ for all $x \in B_{2^{-k}}$, $k=1,2,\dots,\kmax$. Moreover, there exists a unit vector $e \in S^{d-1}$ such that 
    \[
    |\nabla u(x) - (1-\delta)^{\kmax} e| < (1-\delta)^{\kmax}/4 \quad \text{for $x \in B_{2^{-{\kmax-1}}}$}.
    \]
     \item  Let $\alpha_0$ be a small number satisfying $2^{-\alpha_0} = (1-\delta)$ and let $C_0 = (1-\delta)^{-1}$. Then
    \[
    |\nabla u(x)| \leq C_0 |x|^{\alpha_0} \quad \text{ for } x \in B_1.
    \]
 \end{itemize}
\end{corollary}

\begin{proof}
We start by taking $r$ and $\mu$ so small that Lemma \ref{l:stage2} holds for $L=1/4$, and we apply Lemma \ref{l:mainlemma5} which in turn fixes $\delta$ small. Either the second alternative holds directly, or there exists a nonnegative integer $\kmax$ such that
\[
|\nabla u(x)| < (1-\delta)^k \quad \text{ for all } x \in B_{2^{-k}}, \quad k=1,2,\dots,\kmax
\]
and a unit vector $e$ for which
\[ 
|\{ x \in B_{2^{-\kmax}} : |\nabla u(x) - (1-\delta)^{-\kmax} e| \leq r  (1-\delta)^{-\kmax} \}| \geq (1-\mu) |B_{2^{-\kmax}}|.
\] 
In the second case, we apply Lemma \ref{l:stage2} to the function $\bar u$ from Corollary \ref{c:end_of_stage1} for this choice of parameters. Corollary \ref{c:stages1and2} summarizes the results properly scaled back to $u$.

\end{proof}

The first alternative of Corollary \ref{c:stages1and2} sets up the third stage of the proof. Recall the rescaling $\bar u$ defined in Corollary \ref{c:end_of_stage1}. Then, in addition to the assumptions there, $\bar u$ further satisfies
\[
 |\nabla \bar u(x) - e|<1/4, \text{ for every } x\in B_{1/2}.
\]
This function $\bar u$, satisfying the condition above as well as \textit{(1)--(5)} in Corollary \ref{c:end_of_stage1}, is the starting point of the third stage of the paper.

\section{Third stage: the nondegenerate case} \label{s:stage3}
To conclude the third stage of the proof, we analyze the linearized equation establishing a H\"older modulus of continuity for $|\nabla u|$. The nondegeneracy of the gradient obtained in Section \ref{s:stage2} allows us to work in a regime where the linearized kernel is mildly uniform elliptic, which implies coercivity and opens the way for regularity results using more standard techniques. The main result in this section is the following.

\begin{lemma} \label{l:stage3}
      Let $u$ be a solution of \eqref{eq:fracplap} in $B_2$ with $p \in [2, \frac{2}{1-s})$, $\|\nabla  u\|_{L^{\infty}(B_1)}\leq 1$ and $ \Tail_{p-1,sp+1}(u,1) \leq C_1$. Suppose further that there is a vector $ e \in S^{d-1}$ such that
      \[
      |\nabla u(x) -e| < 1/4 \quad \text{ for } x \in B_{1}.
      \]
    Then there exist positive constants $\alpha$ depending on $d$, $s$ and $p$, and $C$ depending on $d$, $s$, $p$ and $C_1$ such that $\nabla u \in C^{\alpha}(B_{1/16})$ with the estimate 
    \[
    \|\nabla u\|_{C^{\alpha}(B_{1/16})} \leq C.
    \]
\end{lemma}

    We recall the following results for nonlocal operators in divergence form in the literature, restated in a convenient way. The first one was obtained in   \cite{chaker2020coercivity} and provides coercivity for our linearized kernel $K_u$ under the nondegeneracy condition.

\begin{theorem} \label{Thm:ChakerLuis}
    Assume there exist $\mu >0$ and $\lambda >0$ such that for every ball $B \subset \R^d$ and $x \in B$ the kernel $K$ satisfies the following hypothesis
     \begin{equation} \label{AssumptionA1}
        |\{ y \in B \, : \, K(x,y) \geq \lambda |x-y|^{-d-2\gamma}\}| \geq \mu |B|,
    \end{equation}
    for some $\gamma >0$.
    Then there exists $c$ depending only on $\mu$ and $d$ such that 
\begin{equation}\label{eq:Eq:Chaker_Luis}
        \int_{B_2} \int_{B_2}  K(x,y) (v(x)-v(y))^2\dd y \dd x \geq c \lambda [ v ]_{W^{\gamma,2}(B_1)}^2.
    \end{equation}
\end{theorem}

The following result is essentially from \cite{Kassman-Schwab}. 
    
\begin{theorem} \label{Thm:Kassman&co}
    Let $v: \R^d \to \R$ be a solution of 
      \[
     \int_{\mathbb{R}^d} K(x,y)\big( v(x) - v(y) \big)   \dd y  = f(x) \quad \text{ for } x\in B_1,
    \]
    where $f$ is a bounded function in $B_1$.
    Assume that there are $\gamma>0, \, \Lambda >1$ such that for every $x_0 \in B_1$, the following assumptions on the kernel are satisfied.
\begin{align*}
  \textit{(A1)} &\quad K(x,y)   =  K(y,x),  \\
  \textit{(A2)} &\quad \rho^{-2} \int_{B_\rho(x_0)} |x_0-y|^2 K(x_0,y) \dd y + \int_{B_\rho^{c}(x_0)} K(x_0,y) \dd y \leq \Lambda \rho^{-2\gamma}, \\
    \textit{(A3)} &\qquad \text{Any time $B_\rho(x_0) \subset B_1$ and $v \in H^{\gamma}(B_{\rho}(x_0))$, }  \\
    &\qquad \qquad \int_{B_{\rho}(x_0)} \int_{B_{\rho}(x_0)} K(x,y) \left(v(x)-v(y)\right)^2 \dd y \dd x   \geq \Lambda^{-1} [ v ]_{W^{\gamma,2}(B_{\rho/2}(x_0))}^2.
\end{align*}
    Then there exist positive constants $\alpha$ and $C$ depending on $\Lambda, \, \gamma$ and $d$ such that
    \[
    [v]_{C^{\alpha}(B_{1/2}) } \leq C \left( \|v \|_{L^{\infty}(\R^d)} +  \|f\|_{L^{\infty}(\R^d)} \right).
    \]
\end{theorem}

In \cite{imbert2019weak}, the authors extended this result to the more general \emph{kinetic} context, allowing also for more general assumptions on the kernel $K$. The original result in \cite{Kassman-Schwab} is without a right-hand side ($f=0$). The result in \cite{imbert2019weak} includes a nonzero right-hand side. The assumption \textit{(A2)} above implies other conditions on the kernel $K$ that are sometimes listed as assumptions. It implies that the bilinear form associated to $K$ is bounded in $H^\gamma$ (See \cite[Theorem 4.1]{imbert2019weak}), which corresponds to the following inequality.
\[ \int_{B_{\rho}(x_0)} \int_{B_{\rho}(x_0)} K(x,y) \left(v(x)-v(y)\right)^2 \dd y \dd x   \leq \Lambda [ v ]_{W^{\gamma,2}(B_\rho(x_0))}^2. \]
Note that the inequality in \cite[Theorem 4.1]{imbert2019weak} is stated in terms of the full norm $\|v\|_{H^\gamma}$ due to the potential contribution of the nonsymmetric part of the kernel in the setting of that paper. The assumption (A3) in \cite{Kassman-Schwab} is stated with the $H^\gamma$ norm in the full ball $B_\rho(x_0)$ in the right hand side instead of the half ball, but that condition is also relaxed in \cite{imbert2019weak}.

We will see that when $\nabla u$ stays sufficiently close to some unit vector $e$, the linearized kernel $K_u$ is nondegenerate in a cone of directions. This allows us to verify \eqref{AssumptionA1} and get the coercivity bound through Theorem \ref{Thm:ChakerLuis}. As a consequence, we can apply Theorem \ref{Thm:Kassman&co} to get a H\"older modulus of continuity. We write the details in the following proof.

\begin{customproof}{Lemma \ref{l:stage3}}
     Consider a direction $\sigma \in S^{d-1}$ and define $v_\sigma(x) = \eta(12x) (\sigma \cdot \nabla u(x))$ where $\eta$ is a cutoff function as in \eqref{def_ve}, so that $\eta(12x)>0$ for $|x|<7/48<1/6$  and $\eta(12x)=1$ for $|x|\leq 1/8$. From Lemma \ref{l:eqforv_e_scaled},  we get 
     \begin{align}\label{eq:ve_RHS}
         |\LL_u v_\sigma(x)| \leq C,
     \end{align}
     for any $x \in B_{1/8}$.
The condition $|\nabla u - e| < 1/4$ in $B_{1}$ directly leads to 
\begin{equation}
    |\nabla u (x) \cdot \tau| >\frac{1}{4} \quad \text{for every }x \in B_{1} \text{ and }|\tau\cdot e| >1/2,
\end{equation}
 which implies that for every $x \in B_{1/8}$, the kernel $K_u(x,y)$ is uniformly elliptic for $y$ in the cone
 \[
 \mathcal{C}(x):=\{x+z\in B_{1/2}\,:  \, |z \cdot e| > |z|/2  \}.
 \]
 Indeed, writing $y=x+t\nu$ for $t \in (0,1/2)$ and $|\nu \cdot e| > 1/2$, it holds that 
\begin{equation*}
|u(x+t\nu)- u(x)|= \left|\int_{0}^{t} \nabla u (x+s\nu) \cdot \nu \dd s\right| > \frac{1}{4}t,
\end{equation*}
which leads to
    \begin{equation} \label{eq:K_u_A3}
\begin{aligned}
    K_{u}(x,x+t\nu)& = (p-1) \frac{|u(x)-u(x+t\nu)|^{p-2}}{t^{d+sp}}  \\ &\geq c \, t^{-d-(sp-p+2)} =  c \, t^{-d+p(1-s)-2}. 
\end{aligned}
\end{equation}
Since this condition only holds locally near the origin, we must modify the kernel outside a small ball while maintaining this growth condition. In doing so, we modify the right-hand side of the equation, but using the tail condition of $u$ as well as the fact that $v_\sigma$ has compact support, we can prove that it remains bounded and thus we can apply Theorem \ref{Thm:Kassman&co}.

More precisely, it follows from \eqref{eq:ve_RHS} that
\[
\int_{B_{1/2}} K_u(x,x+z) (v_\sigma(x) - v_\sigma(x+z)) \dd h  \leq C-v_\sigma(x)\int_{B_{1/2}^c} K_u(x,x+z)\dd z.
\]
We used the fact that $x+h\notin B_{1/6}$ for $x\in B_{1/8}$ and $z\notin B_{1/2}$ and thus $v_\sigma(x+z)=0$.

Note that the bounds on $\|\nabla u\|_{L^{\infty}(B_{1})} $ and $\Tail_{p-1,sp+1}(u,1)$ together with \eqref{Tail_functions} imply that 
\[
\left|\int_{B_{1/2}^c} K_u(x,x+z)\dd z\right|\leq C.
\]
Furthermore, $v_\sigma(x)=\sigma \cdot \nabla  u(x)$ for $x \in B_{1/12}$.

Consider now the following kernel
\[
\widetilde K(z)=c |z|^{-d+p(1-s)-2},
\]
where the constant $c$ is the same as in \eqref{eq:K_u_A3}.

Define the new kernel as
\[
K(x,x+z)=K_u(x,x+z)\one_{B_{1/2}}(z)+\widetilde K(z)\one_{B_{1/2}^c}(z).
\]
Then
\[
\int_{\R^d} K(x,x+z) (v_\sigma(x) - v_\sigma(x+z)) \dd z  \leq C+v_\sigma(x)\int_{B_{1/2}^c} \widetilde K(z)\dd z\leq C.
\]
Similarly we get the lower bound.

We now check that $K$ satisfies the assumptions of Theorem \ref{Thm:Kassman&co} properly scaled, with $x_0\in B_{1/8}$ and $0<\rho<1/8$. \textit{(A1)} is readily verified. Regarding \textit{(A2)}, since $B_{2\rho}(x_0)\subset B_{1/2}$, we get from Lipschitz continuity of $u$
\[
\rho^{-2}\int_{B_\rho}|z|^2K(x_0,x_0+z)\dd z\leq C\rho^{p(1-s)-2}.
\]
Regarding the other term, we split $B_\rho^c$ into the annulus $B_{1/2}\setminus B_\rho$, where we use again Lipschitz continuity of $u$, and $B_{1/2}^c$ where the bound follows trivially. Therefore, we get
\begin{align*}
    \int_{B^c_\rho}K(x_0,x_0+z)\dd z    \leq C\rho^{p(1-s)-2}.
\end{align*}

In order to verify Assumption \textit{(A3)} we only need to observe that the assumption \eqref{AssumptionA1} of Theorem \ref{Thm:ChakerLuis} is satisfied. This is a consequence of the lower bound \eqref{eq:K_u_A3} that holds in a cone of directions. In fact, if $|B\cap  B_{1/2}|\leq \frac{9}{10} |B|$ then it is straightforward from the choice of $\widetilde K $ that \eqref{AssumptionA1} holds for $\sigma=1/10$. On the other hand, if $|B\cap  B_{1/2}|> \frac{9}{10}|B|$ then by \eqref{eq:K_u_A3} we have for every $x\in B$
\[ 
|\{ y \in B\cap  B_{1/2} \, : \, K(x,y) \geq c |x-y|^{-d+p(1-s)-2}\}|\geq |\mathcal{C}(x) \cap B\cap  B_{1/2}|\geq c_1|B|
\]
where $c_1>0$ depends only on $d$. Therefore, we can apply Theorem \ref{Thm:ChakerLuis} with $\mu=\min\{1/10,c_1\}$ and get Assumption \textit{(A3)}.

 We are in conditions to apply Theorem \ref{Thm:Kassman&co} to $v_\sigma$ and get
\[
[v_\sigma]_{C^{\alpha}(B_{1/16})} \leq C(\|v_\sigma\|_{L^{\infty}(\R^d)} + C) \leq C,
\]
since $\|v_\sigma\|_{L^{\infty}(\R^d)} \leq 1$. Since $\sigma$ is an arbitrary unit vector, the result is obtained.
\end{customproof}

\section{Proof of the main theorem} \label{s:main_proof}

In this section we prove interior $C^{1,\alpha}$ regularity of solutions to \eqref{eq:fracplap}. We start by rescaling the solution to place it in the assumptions of Lemma \ref{l:iteration}. 
By the iteration argument in Section~\ref{s:stage1}, together with Corollary~\ref{c:stages1and2}, we obtain two possible outcomes: if the gradient vanishes at the point under consideration, we directly deduce $C^{1,\alpha}$ regularity there; otherwise, we obtain a bound for the linearized kernel, which brings us into the framework of Lemma \ref{l:stage3}, which gives local $C^{1,\alpha}$ in small balls with radius depending on the point. To conclude the main proof, it only remains to patch these estimates together in a way that is independent of the point.

\begin{proof}[Proof of Theorem \ref{t:main}]
Let $u$ be a solution to \eqref{eq:fracplap} in $B_2$. Applying Theorem \ref{t:biswas&co} to $u$, we know that $u \in C_{loc}^{0,1}(B_{2})$ with an estimate. Without loss of generality, we can assume that $u(0) = 0$ by subtracting $u(0)$ from $u$. Moreover, we can also assume that $\Tail_{p-1,sp+1}(u,1) + \| u\|_{L^{\infty}(B_1)}\leq 1$ and $|\nabla u| \leq 1$ in $B_1$.
Otherwise, we consider the scaled function $\tilde u(x) = (\Tail_{p-1,sp+1}(u,1) + \| u\|_{Lip(B_{1})})^{-1} u(x)$ and apply the analysis below to $\tilde u$, to obtain
\[
\|\tilde u\|_{C^{1,\alpha}(B_1)}\leq C,
\]
which in terms of $u$ becomes
\begin{align*}
    \| u\|_{C^{1,\alpha}(B_1)}\leq&\, C ({\Tail}_{p-1,sp+1}(u,1) + \| u\|_{Lip(B_{1})})\\
    \leq&\, C({\Tail}_{p-1,sp}(u,2) + \| u\|_{L^\infty(B_2)}),
\end{align*}
using Theorem \ref{t:biswas&co} and the properties of tail functions.

Now perform a rescaling to $u$ so that the rescaled function satisfies the assumptions \textit{(1)}--\textit{(4)} of Lemma \ref{l:iteration}. This can be done considering $u_{1}(x):=2^{K_0}u(2^{-K_0}x)$ for $K_0$ large enough, see \eqref{scale_tails}. For clarity, we denote the rescaled function $u_{1}$ simply by $u$.

From Corollary \ref{c:stages1and2}, either
\begin{equation} \label{eq:nabla=0_est}
    |\nabla u (x)| \leq C |x|^{\alpha_0} \quad \text{for }x \in B_{1/2}
\end{equation}
or there exists an integer $\kmax$ such that 
\begin{equation} \label{eq:mesoscale}
\|\nabla u\|_{L^\infty(B_{2^{-n}})} \leq (1-\delta)^n \qquad \text{ for } n=0,1,2,\dots,\kmax
\end{equation}
and moreover, using Lemma \ref{l:stage3}, the rescaled function
\[ 
u_2(x) = \frac{2^{\kmax}}{(1-\delta)^{\kmax}} u(2^{-\kmax}x),
\]
satisfies $[u_2]_{C^{1,\alpha}(B_{1/16})}\leq C$.  This in terms of $u$ becomes
\[
[\nabla u]_{C^{\alpha}(B_{2^{-N-4}})}\leq C(1-\delta)^N 2^{\alpha N}.
\]
Taking $\alpha<\alpha_0$, we get
\begin{equation} \label{eq:thesmallerball}
[\nabla u]_{C^{\alpha}(B_{2^{-N-4}})}\leq C.
\end{equation}
Note that we can translate this argument to every point $ x\in B_{1/2}$, obtaining the same dichotomy, where $N=N(x)$.

Take any other point $y \in B_{1/2}$. We want to prove that for some constant $C$ independent of $\kmax$, we have
\[
|\nabla u(x)-\nabla u(y)|\leq C|x-y|^\alpha.
\]

There are two possibilities. Either $y \in B_{2^{-N-4}}(x)$ or $y \notin B_{2^{-N-4}}(x)$. In the first case, we use \eqref{eq:thesmallerball} and reach the conclusion right away. In the second case, we pick $n \in \{0,1,\dots,\kmax+3\}$ so that $2^{-n-1} \leq |x-y| < 2^{-n}$. We observe that $y \in B_{2^{-n}}(x)$. If $n \leq \kmax$, we use \eqref{eq:mesoscale} to conclude $|\nabla u(y) - \nabla u(x)| \leq 2 (1-\delta)^n \leq 2(1-\delta)^{-1} |x-y|^{\alpha_0}$. In the special cases $n \in \{\kmax+1,\kmax+2,\kmax+3\}$, we still have $y \in B_{2^{-\kmax}}(x)$ and get $|\nabla u(y) - \nabla u(x)| \leq 2 (1-\delta)^\kmax \leq 2(1-\delta)^{-4} |x-y|^{\alpha_0}$. This concludes the proof in all cases.
\end{proof}

\section{Working with weak solutions}\label{s:weak-solutions}

In this section we describe the technical details needed to apply the ideas of the previous sections to prove Theorem \ref{t:main} for weak solutions.

There are two natural notions of non-classical solutions for the equation \eqref{eq:fracplap}: weak solutions (in the sense of distributions) and viscosity solutions. In \cite{korvenpaa2019equivalence}, it is shown that these two notions are equivalent. That is, every function $u \in W_{loc}^{s,p}(\Omega)$ that solves \eqref{eq:fracplap} in the weak sense is also a continuous viscosity solution, and every continuous viscosity solution is also in $W^{s,p}_{loc}(\Omega)$ and solves \eqref{eq:fracplap} in the weak sense.

Some methods can be applied naturally to weak solutions. Others adapt more easily to the context of viscosity solutions. In this paper, we have some of both. Moreover, some of our arguments in the first state of our proof rely on the continuity of $\nabla u$ in a way that is more difficult to handle in either setting, so we resort to an approximation by a more regular equation.

Let us describe our argument for each of the three stages of the proof of Theorem \ref{t:main}.
\begin{enumerate}
    \item We prove a version of Lemma \ref{l:iteration} that applies to weak solutions $u$ without any further smoothness assumption. For this purpose, we regularize the equation and prove the same improvement of flatness for the respective solution $u_\eps$, which is smooth enough. Since the estimates we get are uniform in $\eps$, we can take $\eps\to 0$ and obtain Lemma \ref{l:iteration} for $u$. We can then proceed with the iterative argument in Section \ref{s:stage1}. The regularization of the PDE is used to prove the iteration step at each fixed scale. This is needed to overcome the bad scaling properties of the regularized equation \eqref{eq:main_apriori}. We end with with the same dichotomy as in Corollary \ref{c:end_of_stage1}, concluding the first stage.

    Technically, we could prove the estimates in the second and third stages of the proof by regularization as well. However, these parts of the proof adapt more easily to the notions of viscosity and weak solutions, and a regularization is not strictly necessary.
    \item The Ishii-Lions method originated in the study of viscosity solutions for fully nonlinear elliptic equations. It applies seemlessly to viscosity solutions, with practically no extra effort. We conclude with a version of Corollary \ref{c:stages1and2}.
    \item The third stage of the proof can be adapted to weak solutions (in the sense of distributions). We consider the difference quotients of $u$ instead of the directional derivatives and prove that they satisfy an equation which still satisfies the assumptions of Theorem \ref{Thm:ChakerLuis}. An application of Theorem \ref{Thm:Kassman&co} gives H\"older regularity independently of $\eps$. 
\end{enumerate}

Once we established the main results of Sections \ref{s:reduc_grad}--\ref{s:stage3}, the argument in Section \ref{s:main_proof} can be repeated \textit{verbatim}. In fact, the previous steps imply that the gradient exists at every point.

We split each part of the discussion into their respective subsection.

\subsection{Justifying Sections \ref{s:linearized}--\ref{s:stage1}}

Let $u$ solve \eqref{eq:fracplap} in $B_{2R}$ for some arbitrary $R$ large. This is at the same time a weak solution in $W^{s,p}_{loc}(B_{2R})$ and a viscosity solution, continuous in $B_{2R}$. For $\eps>0$ small, we consider the function $u_\eps$ that solves the regularized equation
\begin{align}\label{eq:main_apriori}
    \begin{cases}
        (-\Delta_p)^su_\varepsilon-\eps\Delta u_\varepsilon=0 &\text{ in } B_R,\\
    u_\varepsilon=u &\text{ on } B_R^c.
    \end{cases}
\end{align}
This function naturally belongs to $$u_\eps \in H^1(B_R) \cap W^{s,p}(B_{2R}) \cap L_{sp}^{p-1}(\R^d)$$ 
as the weak solution to the previous regularized problem. We will see that for any $\eps>0$, the function $u_\eps$ is actually much more regular. In the case $p \geq 1/(1-s)$, we will prove $u_\eps \in C^{2,\alpha}$ for all $\alpha \in (0,1)$. When $p \in [2,1/(1-s))$, we will prove that $u_\eps \in C^{2,\gamma}$ for all $\gamma \in (0,p(1-s))$.

The definition of weak solutions follows the definitions of Subsection \ref{subsection_Def} with obvious changes. This notion of solution is equivalent to the viscosity one (also defined in Subsection \ref{subsection_Def}), as proven in \cite{korvenpaa2019equivalence}. Strictly speaking, we only use here that weak solutions are viscosity solutions, which is the easier implication to prove.

\begin{lemma}\label{l:equivalence}
    The function $u_\varepsilon$ is a weak solution of \eqref{eq:main_apriori} if and only if it is a viscosity solution.
\end{lemma}

We observe that the Lipschitz regularity obtained in \cite{biswas2025lipschitz,biswas2025improved} applies to the equation \eqref{eq:main_apriori} independently of $\eps \in (0,1)$.
 
\begin{lemma}\label{l:Lips_apriori}
    Let $u_\varepsilon$ be a solution of \eqref{eq:main_apriori} with $\eps \in (0,1)$. Then $u_\varepsilon$ is locally Lipschitz continuous in $B_{R}$ and
    \[
    \|u_\varepsilon\|_{Lip(B_{R/2})}\leq CR^{-1}(\|u_\varepsilon\|_{L^\infty(B_R)}+{\Tail}_{p-1,sp}(u_\varepsilon;R)),
    \]
    where $C>0$ depends only on $d,s$ and $p$, but not on $\eps$.
\end{lemma}

\begin{proof}
    The proof follows exactly the same lines as \cite{biswas2025lipschitz,biswas2025improved}. Adding the term $\eps\Delta$ makes the equation more elliptic, which only helps with the estimates.
\end{proof}

Using this uniform regularity as well as good stability properties for viscosity solutions, we obtain the following stability result. 
\begin{lemma}\label{l:stability}
    Consider the solutions $u_\varepsilon$ of \eqref{eq:main_apriori} for various values of $\eps \in (0,1)$. Here, $u$ is a solution to \eqref{eq:fracplap} in $B_{2R}$. Then,
    \[
    u_{\eps}\to u \quad \text{ locally uniformly in } B_{2R}
    \]
    and
    \[
    [u_{\eps}-u]_{W^{s,p}(B_r)}\to 0\quad \text{ as } \eps \to 0 \text{ for any } r <2R.
    \]
\end{lemma}

\begin{proof}
    From Lemma \ref{l:Lips_apriori}, we know that the family $\{u_\eps\}$ is uniformly Lipschitz. By Arzelà-Ascoli Theorem, there exists a subsequence $u_{\eps_n}$ that converges locally uniformly. Since $W^{s,p}(B_r)$ is compactly contained in the space of Lipschitz functions, we can also conclude that $u_{\eps_n}$ converges in $W^{s,p}(B_r)$ by taking a further subsequence if necessary.

    From the stability of viscosity solutions under uniform limits, we conclude that the $\lim_{n\to \infty} u_{\eps_n}$ is a viscosity solution of \eqref{eq:fracplap}, that has to coincide with $u$ by uniqueness. Thus, $\lim_{n\to \infty} u_{\eps_n} = u$.

    Since every subsequential limit of $u_\eps$ has the same limit, and the $u_\eps$ are contained in a compact set of functions, then the full limit holds $\lim_{\eps \to 0} u_\eps = u$ both locally uniformly and locally in $W^{s,p}$.
\end{proof}
In the next result, we summarize the further regularity properties of $u_\eps$. The regularity obtained justifies the assumption in Sections \ref{s:reduc_grad} and \ref{s:stage3}. 
To keep the focus on the discussion, the proof of this result, being technical and arguably standard, is presented in the Appendix  \ref{s:appendix}.

\begin{lemma} \label{l:eps-regularity-needed}
    Let $u_\eps$ solve \eqref{eq:main_apriori}. Then, for each small $\eps$, we have
    \begin{itemize}
    \item If  $p \in [2,1/(1-s))$ then $u_\eps\in C^{2,\alpha}_{loc}(B_R)$ for all $\alpha<p(1-s)$.
    \item If $p \in [1/(1-s),2/(1-s))$ then $u_\eps\in C^{2,\alpha}_{loc}(B_R)$ for all $\alpha \in (0,1)$.
    \end{itemize}
\end{lemma} 

The regularity estimates of Lemma \ref{l:eps-regularity-needed} are probably not optimal, but they are more than enough to carry out all the arguments in Section \ref{s:reduc_grad}. 

We now proceed to prove all the estimates from Section \ref{s:reduc_grad} but for a solution $u_\eps$ of \eqref{eq:main_apriori} instead of a solution $u$ of \eqref{eq:fracplap}. The main objective is to prove that Lemma \ref{l:iteration} holds for $u_\eps$.

We start with a properly modified version of Lemma \ref{l:ve-eq-small-rhs} about the linearized equation, that now contains an extra term.
\begin{lemma} \label{l:equationforv_e_eps}
    Let $u_\eps$ be a function satisfying the assumptions of Lemma \ref{l:iteration} but with the equation \eqref{eq:main_apriori} instead of \eqref{eq:fracplap}. Let
    \begin{align}\label{eq:def_ve_apriori}
        v_e(x):= \eta(2^{1-K_0} x) \left(e \cdot \nabla u_\eps(x) \right),
    \end{align}
    with $\eta$ defined in Section \ref{s:linearized}. Then, it solves the equation 
    \begin{equation} \label{eq:v_e_eps}
    |\LL_{u_\eps} v_e -\varepsilon\Delta v_e|  \leq \eps_0 \quad \text{ in }B_1,
    \end{equation}
    where $\eps_0$ is arbitrarily small if we choose $\eps_1$ and $\delta$ small and $K_0$ large in Lemma \ref{l:iteration}.
\end{lemma}

\begin{proof}
The term $-\eps \Delta v_e$ is local and equals $e \cdot \nabla (-\eps \Delta u_\eps)$ in $B_1$. Thus, it does not affect the proof of Lemma \ref{l:eqforv_e_scaled}, leading to the conclusion.
\end{proof}

Using the equation \eqref{eq:v_e_eps}, we reproduce the result of Lemma \ref{l:insomeball} for $u_\eps$, with an estimate independent of $\eps$. Indeed, the proofs in Section \ref{s:reduc_grad} apply to the equation \eqref{eq:main_apriori} with only some minimal modification that we explain below. In this case, the function $v_\eps$ is guaranteed to be smooth, so all the computations are well justified.

Multiplying equation \eqref{eq:v_e_eps} by $v_k$ and integrating over $B_1$ we get
\begin{equation} \label{eq:testvkeps}
    \int_{B_1} (\LL_{u_\eps} [v_e] - \varepsilon \Delta v_e) v_k\,\dd x\leq \varepsilon_0 \int_{B_1} v_k \,\dd x.
\end{equation}
We split the term $ \int_{B_1} \LL_u [v_e] v_k \dd x$ following the same logic as in Section \ref{s:reduc_grad} and bound each term in a similar way. We have
\begin{align*}
    J_{Tail,k} =&\,\int_{\R^d\setminus B_1}\int_{B_1} (\varphi_k(x)-v_e(y))v_k(x) K_u(x,y)\,\dd x\dd y\\
J_{\varphi_k}=&\,\iint_{B_1\times B_1}(\varphi_k(x)-\varphi_k(y))v_k(x)K_u(x,y)\,\dd x\dd y + \eps \int_{B_1} \nabla v_k(x) \cdot \nabla \varphi_k(x) \dd x \\
    J_{gap,k}=&\,\iint_{B_1\times B_1}[\varphi_k(y)-v_e(y)]_+v_k(x)K_u(x,y)\,\dd x\dd y\\
    J_{r,k}=&\, \iint _{B_1\times B_1} (v_k(x)-v_k(y))v_k(x) K_u(x,y)\,\dd x\dd y + \eps \int_{B_1} |\nabla v_r(x)|^2 \dd x.
\end{align*}
We observe that the term $-\eps \Delta v_e$ in \eqref{eq:testvkeps} affects the terms $J_{\varphi_k}$ and $J_{r,k}$. However, the results of Lemmas \ref{l:bound_Jphi} remains true as stated in terms of the new $J_{\varphi_k}$ and $J_{r,k}$. Moreover, the estimate from Lemma \ref{l:bound_Jr} remains true by ignoring the second term. The rest of the proof proceeds without modification.

Combining the result of Lemma \ref{l:iteration} with Lemma \ref{l:eps_osc}, we obtain an estimate for $u_\eps$ independent of $\eps$. Passing to the limit as $\eps \to 0$, we get the following result that we restate here in terms that do not directly involve the gradient $\nabla u$.

\begin{lemma} \label{l:iteration-weak-solutions}
For any $\eps_0 > 0$, there exists $\delta, \eps_1$ (small) and $K_0$ (large) depending on $\eps_0, d, s, p$ such that if the following conditions hold
\begin{itemize}
    \item $u$ satisfies \eqref{eq:fracplap} in $B_{2^{K_0+1}}$ in the weak (and viscosity) sense,
    \item $u(0)=0$,
    \item $[u]_{Lip(B_{2^n})} \leq (1-\delta)^{-n}$, for $n=0,\dots,K_0$,
    \item ${\Tail}_{p-1,sp+1}(u,2^{K_0}) \leq \eps_1$.
\end{itemize}
Then it implies one of the two options below
\[
\begin{cases}
\text{there exists } e \in S^{d-1} \text{ such that } |u(x) - e\cdot x|   \leq \eps \text{ in } B_{1/2}   \\
[u]_{Lip(B_{1/2})} \leq (1-\delta).
\end{cases}
\]
\end{lemma}

Iterating Lemma \ref{l:iteration-weak-solutions} we reach the same conclusion of Section \ref{s:stage1} concluding the first stage of the proof for weak solutions. Namely, essentially the same conclusions as in Lemma \ref{l:mainlemma5} and Corollary \ref{c:end_of_stage1} apply to weak solutions as well. Here is a reformulated version of Lemma \ref{l:mainlemma5} stated without any reference to derivatives of $u$.

\begin{lemma} \label{l:stage1-weak-sols}
For any $\eps_0 >0$ small, there exist $\delta, \eps_1>0$ (small) and $K_0>0$ (large) depending on $\eps_0, \, d,\, s,\, p$ such that if the following conditions hold
    \begin{enumerate}
    \item $u$ satisfies \eqref{eq:fracplap} in $B_{2^{K_0+1}}$ in the weak sense,
    \item $u(0) = 0$,
    \item  $[u]_{Lip(B_{2^{n}})} \leq (1-\delta)^{-n} ,$ for $n=0,1,\dots,K_0$, 
    \item $\Tail_{p-1,sp+1}(u,2^{K_0}) \leq \eps_1$.
    \end{enumerate}
 Then one of the following must hold.
    \begin{itemize}
      \item 
      There is a nonnegative integer $\kmax$ such that  
      \[
      [u]_{Lip(B_{2^{-n}})} \leq (1-\delta)^n  \text{ for } n = 0,1,\dots,\kmax.
      \] 
      Moreover, there exists a unit vector $e \in S^{d-1}$ such that
       \[ |u(x) - (1-\delta)^{-\kmax} e \cdot x| \leq \eps_0 (1-\delta)^{-\kmax} 2^{-\kmax} \text{ in } B_{2^{-\kmax}}.\]
    \item 
    For $\alpha_0>0$ so that $2^{-\alpha_0}=(1-\delta)$ and  $C_0 = (1-\delta)^{-1}$, we have
    \begin{align*}
    [u]_{Lip(B_r(x))} &\leq C_0 r^{\alpha_0}\quad \text{ for } r \in (0,1], \\
    |u(x)| &\leq C_0 |x|^{1+\alpha_0} \quad \text{ for }  x \in B_1.
    \end{align*}
    In particular $\nabla u(0)$ is well-defined and equals zero.
\end{itemize}
    
\end{lemma}

\subsection{Justifying Section \ref{s:stage2}}

We move on to justifying the second stage of our proof, in Section \ref{s:stage2}, for weak (and viscosity) solutions, without any a priori assumption on their smoothness. 

This stage of the proof consists of an Ishii-Lions argument, applied directly to $u$. This method is well suited for viscosity solutions, hence it can be carried out in this context essentially as written.

In the definition of viscosity solutions, the equation is tested by evaluating the operator $(-\Delta_p)^s$ at special functions and only at points that can be touched from one side by a local smooth function. In our proof in Section \ref{s:stage2}, we obtain a contradiction by evaluating the equality \eqref{eq:splapforu} at the special pair of points $\bar x$ and $\bar y$ where the maximum of \eqref{contrad1} is achieved. These points are exactly points where the function $u$ is touched from above and below respectively by the smooth function $\phi$ defined in \eqref{eq:def_phi1}.

Following the standard definition of viscosity solutions for integro-differential equations, we consider for some small $\rho >0$, the following test functions
\[
w_1(z) = 
\begin{cases}
u(\bar x) - \phi(\bar x, \bar y) + \phi(z,\bar{y}) & \text{if } z \in B_\rho(\bar{x}), \\[6pt]
u(z) & \text{otherwise},
\end{cases}
\]
and
\[
w_2(z) = 
\begin{cases}
u(\bar y) + \phi(\bar x, \bar y) -\phi(\bar{x},z) & \text{if } z \in B_\rho(\bar{y}), \\[6pt]
u(z) & \text{otherwise}.
\end{cases}
\]

The definition of viscosity solutions tells us that $(-\Delta_p)^s w_1(\bar x)\leq 0$ and $(-\Delta_p)^s w_2(\bar y)\geq 0$. Thus, we substitute \eqref{contrad2} with $(-\Delta_p)^s w_1(\bar x) - (-\Delta_p)^s w_2(\bar y) \leq 0$.

The rest of the proof can be repeated \textit{mutatis mutandis} provided that we choose $\rho < \delta_1 |\bar a|$. Indeed, the estimate for $I_1$ and $I_2$ in Lemmas \ref{l:cone} and \ref{l:D1} use only that $u(\bar x+z) - u(\bar x) \leq \phi(\bar x+z) - \phi(\bar x)$ and a corresponding inequality centered at $\bar y$. These inequalities are, of course, preserved for the new test functions $w_1$ and $w_2$.

We obtain again Corollary \ref{c:stages1and2}. Let us reformulate it here without any reference to derivatives of $u$. 

\begin{corollary} \label{c:stages1and2-weaksolutions}
 For small enough $\delta>0$, let $u$ be function that satisfies the assumptions of Lemma \ref{l:stage1-weak-sols}. Then one of the following must hold.
 \begin{itemize}
    \item  There is a positive integer $N$ such that $[u]_{Lip(B_{{2^{-k}}})} < (1-\delta)^k$ for all $k=1,2,\dots,\kmax$. Moreover, there exists a unit vector $e \in S^{d-1}$ such that 
    \[
    \left[u(x) - (1-\delta)^{\kmax} 2^{-\kmax} e\cdot x \right]_{Lip(B_{2^{-N-1}})} < \frac 14 (1-\delta)^{\kmax}.
    \]
     \item  Let $\alpha_0$ be a small number satisfying $2^{-\alpha_0} = (1-\delta)$ and let $C_0 = (1-\delta)^{-1}$. Then
    \begin{align*}
    [u]_{Lip(B_r(x))} &\leq C_0 r^{\alpha_0}\quad \text{ for } r \in (0,1], \\
    |u(x)| &\leq C_0 |x|^{1+\alpha_0} \quad \text{ for }  x \in B_1.
    \end{align*}
    In particular $\nabla u(0)$ is well-defined and equals zero.
\end{itemize}
\end{corollary}

The first alternative sets up of the third stage of the proof. Recall the rescaling $\bar u$ defined in Corollary \ref{c:end_of_stage1}. In addition to the assumptions there, $\bar u$ further satisfies
\begin{align}\label{eq:a.e.apriori}
    \left[ \bar u(x) - e\cdot x \right]_{Lip(B_{1/2})} \leq \frac 14.
\end{align}
This condition above is the same as saying that $|\nabla \bar u - e| \leq 1/4$ almost everywhere, if $\nabla \bar u$ is the weak derivative of the Lipschitz function $\bar u$.

This function $\bar u$, satisfying the condition above as well as \textit{(1)--(5)} in Corollary \ref{c:end_of_stage1}, is the starting point of the next section. We call it again $u$.

\subsection{Justifying Section \ref{s:stage3}}

The third stage in our proof is based on applying elliptic regularity results for integro-differential equations (as in \cite{Kassman-Schwab}) to the directional derivatives of $u$. This method is well suited to be applied to weak solutions (in the sense of distributions) provided that some precautions are taken into consideration.

If we start from a weak solution $u$ of \eqref{eq:fracplap}, even after establishing its Lipschitz regularity through Theorem \ref{t:biswas&co}, its directional derivatives are no better than $L^\infty$ inside the domain of the equation. This is not enough regularity to make sense of the equation \eqref{eq:ve_RHS} in Section \ref{s:stage3}. To overcome this technical inconvenience, we compute an equation for incremental quotients instead.

We will compute the kernel of the equation solved by the difference quotients of $u$ using Lemma \ref{l:appendix}. Furthermore, the same result ensures the nondegeneracy of the kernel in a cone of directions, which in turn allows us to apply the argument developed in Section \ref{s:stage3}.

Let us define the difference quotient of a function $w$ in the direction $e\in S^{d-1}$ by
\[
D_e^hw(x):=\frac{w(x+he)-w(x)}{h} \quad \text{and} \quad  \delta_e^h w(x):= w(x+he)-w(x).
\]

\begin{lemma}\label{l:PDE_DifQuot}
    Let $u$ be a function in $L^{p-1}_{sp}(\R^d)$ that is Lipschitz in $B_1$ and solves \eqref{eq:fracplap}.

    Assume that $[u]_{Lip(B_1)} \leq 1$, $\Tail_{p-1,sp+1}(u,1) \leq C_1$ and there is a unit vector $e \in S^{d-1}$ such that $[u(x) - e \cdot x]_{Lip(B_1)} \leq 1/4$.
    
    For a small $h > 0$, let us define
    \[
    v_e^h(x):= D_e^h u(x)\eta(x),
    \]
    where $\eta$ is a smooth cutoff function which is equal to $1$ in $B_{1/2}$ and equal to zero in $B_{7/8}^c$ Then $v_e^h$  satisfies the equation
    \begin{align*}
        \left|   \int_{\R^d} K(x,x+z) (v_e^h(x) - v_e^h(x+z))\dd z \right| \leq C, \quad \text{ for } x\in B_{1/2}
    \end{align*}
    where $K$ satisfies the assumptions of Theorem \ref{Thm:Kassman&co}.
\end{lemma}

\begin{proof}

Using the notation introduced in \eqref{eq:notation_Integral}, for $x \in B_{1/2}$ and for small positive $h$, we obtain 
\begin{equation} \label{eq:compute_diff_quot}
\begin{aligned}
    0&=(-\Delta_p)^s u(x+he) -  (-\Delta_p)^s u(x) \\
    &= \mathcal{L}[B_{1/2}] u(x+he) -\mathcal{L}[B_{1/2}] u(x) + \mathcal{L}[B_{1/2}^c] u(x+he) -\mathcal{L}[B_{1/2}^c] u(x).
\end{aligned}
\end{equation}
Using Lemma \ref{l:est_tail_difference}, we can bound the terms involving $B_{1/2}^c$ by
\begin{align}\label{eq:tail_difquot_apriori}
    \left| \mathcal{L}[B_{1/2}^c] u(x+he) -\mathcal{L}[B_{1/2}^c] u(x)\right| \leq C h,
\end{align}
where $C$ is a constant depending on ${\Tail}_{p-1,sp+1}(u,1)$, $d$, $s$ and $p$.
On the other hand, to deal with the  terms with $B_{1/2}$ in \eqref{eq:compute_diff_quot} we use Lemma \ref{l:appendix} and obtain the equation satisfied by the difference quotient of $u$, namely for $x \in B_{1/2}$
\begin{align} \label{eq:diffquot_h}
    \left|   \int_{B_{1/2}} K_u^h(x,x+z) (D_e^hu(x)-D_e^hu(x+z))\dd z \right| \leq C,
\end{align}
with
\[
    K_u^h(x,x+z):= \frac{p-1}{|z|^{d+sp}} \int_0^1 |u(x)-u(x+z)+t( \delta_e^h u(x)- \delta_e^h u(x+z))|^{p-2}\dd t.
\]

From \eqref{eq:a.e.apriori}, we get
\[
|\nabla u(x) \cdot \tau| >\frac{1}{4} \quad \text{for a.e. }x \in B_{1}, \, |\tau \cdot e| >1/2.
\]
This implies that the kernel $K_u^h(x,y)$ is uniformly elliptic for $y$ in the cone
\[
\C(x)=\{x+z \in B_{1/2}:|z \cdot e| > |z|/2 \}.
\]
Indeed, writing $z=t\nu$ with $t \in (0,1/2)$ and $|\nu \cdot e|>1/2$, it holds that
\[
|u(x+t\nu)-u(x)|= \left| \int_{0}^{t} \nabla u(x+s\nu) \cdot \nu \dd s\right| >\frac{1}{4}t,
\]
which together with the second part of Lemma \ref{l:appendix}, leads to
\begin{equation} \label{K_u^h_nondeg}
\begin{aligned}
 |K_u^h(x,x+z)| &\geq \frac{p-1}{|z|^{d+sp}} \left( |u(x)-u(x+z)| + | \delta_e^h u(x)- \delta_e^h u(x+z) |\right)^{p-2} \\
 & \geq \frac{p-1}{|z|^{d+sp}} |u(x)-u(x+z)|^{p-2} \geq c |z|^{-d+p(1-s)-2}.
\end{aligned}
\end{equation}
Therefore the difference quotient is a weak solution to an equation satisfying the same conditions as in  Section \ref{s:stage3}.

Consider now, $v_e^h(x):= D_e^h u(x) \eta (x)$, where $\eta$ is a smooth cutoff function which is equal to $1$ in $B_{1/2}$ and equal to zero in $B_{7/8}^c$. Since $[u]_{Lip(B_1)} \leq 1$, $v_e^h$ is a globally bounded function and also satisfies \eqref{eq:diffquot} in $B_{1/2}$. 
Define the kernel
\[
K(x,x+z)=K_u^h(x,x+z)\one_{B_{1/2}}(z)+c|z|^{-d+p(1-s)-2}\one_{B_{1/2}^c}(z),
\]
where $c$ is the same constant appearing in \eqref{K_u^h_nondeg}. Notice that, using the Lipschitz continuity of $u$,  we get
\begin{align*}
    \left|\int_{B_{1/2}^c} c|z|^{-d+p(1-s)-2} \left(v_e^h(x) - v_e^h(x+z) \right)\dd z \right| \leq C [u]_{Lip(B_{1})} \int_{B_{1/2}^c}|z|^{-d+p(1-s)-2}  \dd z\leq C,
\end{align*}
where here we used that $\eta \leq 1$ and for $x+z \in B_{7/8}$ and $h<1/8$
\begin{align*}
    |v_e^h(x) - v_e^h(x+z)| &= \left| \frac{u(x+he)-u(x)}{h} \eta(x) -  \frac{u(x+z+he)-u(x+z)}{h} \eta(x+z) \right| \\
    &\leq  \left|\frac{u(x+he)-u(x)}{h} \right| + \left| \frac{u(x+z+he)-u(x+z)}{h} \right|  \leq 2 [u]_{Lip(B_{1})} \leq 2, 
\end{align*}
thus we find that $v_e^h$ satisfies the following equation in $B_{1/2}$
\begin{align} \label{eq:diffquot}
    \left|   \int_{\R^d} K(x,x+z) (v_e^h(x) - v_e^h(x+z))\dd z \right| \leq C.
\end{align}
It is not difficult to check that also in this case $K$ satisfies the assumptions of Theorem \ref{Thm:Kassman&co}. 

\end{proof}

Using Lemma \ref{l:PDE_DifQuot}, we obtain an equation for the differential quotients $v^h_e(x)$ with the same properties of the equation of $v_e$ in Section \ref{s:stage3}. We proceed to apply Theorem \ref{Thm:Kassman&co} to $v^h_e$ and obtain H\"older estimates independent of $h$. This implies that the solution $u$ is $C^{1,\alpha}$ with the same estimates.

\subsection{Wrapping up the proof}

The estimate of the stage 3 tells us that the function $u$ is locally $C^{1,\alpha}$ around the points where the first alternative in Corollary \ref{c:stages1and2-weaksolutions} holds. We observe that $u$ must be differentiable at the origin in either case. By translation of this argument, $u$ must be differentiable at any point in the domain of the equation. Once we make this observation, and taking into account the results of Lemma \ref{l:stage1-weak-sols} and Corollary \ref{c:stages1and2-weaksolutions}, together with the $C^{1,\alpha}$ regularity of the third stage of the proof, the proof of Theorem \ref{t:main} for weak solutions proceeds verbatim as in Section \ref{s:main_proof}.

\appendix
\section{Regularity estimates for the approximate problem}
\label{s:appendix}
In this appendix, we prove regularity estimates  for a solution of the approximated problem \eqref{eq:main_apriori}.
In particular, the aim is to provide a proof for Lemma \ref{l:eps-regularity-needed}.

The Lipschitz regularity of $u_\eps$ given in Lemma \ref{l:Lips_apriori} is our starting point.

\subsection{Regularity for the fractional p-Laplacian operator}

\begin{lemma} \label{l:frac_plap_of_Lipschitz}
Let $2 \leq p < 1/(1-s)$ and $v$ be a function in $L^{p-1}_{sp}(\R^d)$ that is Lipschitz in $B_R$.  Then $(-\Delta_p)^s v$ is in the local Besov space $B^{-\alpha}_{\infty,\infty}(B_r)$ for any $r<R$ where $\alpha = 1 - p(1-s)$.
\end{lemma}

\begin{proof}
Let $\varphi \in W^{\alpha,1}_0(B_r) $ with $r<R$. Then
    \begin{align*}
        \int_{\R^d} (-\Delta_p )^s v(x) \varphi(x)\dd x=\frac{1}{2} \iint_{\R^d\times \R^d}\frac{|v(x)-v(y)|^{p-2}(v(x)-v(y))(\varphi(x)-\varphi(y))}{|x-y|^{d+sp}}\dd x\dd y.
    \end{align*}
    We split this integral into the local and nonlocal parts. In the local part, we use the Lipschitz regularity of $v$ to get
    \begin{align*}
        &\left|\iint_{B_{\frac{R+r}{2}}\times B_{\frac{R+r}{2}}}\frac{|v(x)-v(y)|^{p-2}(v(x)-v(y))(\varphi(x)-\varphi(y))}{|x-y|^{d+sp}}\dd x\dd y\right|\\
        &\qquad \leq C\iint_{B_{\frac{R+r}{2}}\times B_{\frac{R+r}{2}}}\frac{|\varphi(x)-\varphi(y)|}{|x-y|^{d+sp-p+1}}\dd x\dd y=C [\varphi]_{W^{\alpha,1}},
    \end{align*}
    where $C$ depends on $d,\,s,\,p$ and $[v]_{Lip(B_{\frac{R+r}{2}})}$. Regarding the nonlocal part, we take into account the fact that the support of $\varphi$ is contained in $B_{r}$. Hence we need to estimate
    \begin{align*}
        &\left|\int_{B^c_{\frac{R+r}{2}}}\int_{B_r} \frac{|v(x)-v(y)|^{p-2}(v(x)-v(y))\varphi(x)}{|x-y|^{d+sp}}\dd x\dd y\right|\\
        &\qquad \leq \int_{B_r}|\varphi(x)| \int_{B^c_{\frac{R+r}{2}}} \frac{|v(y)+1|^{p-1}}{|x-y|^{d+sp}}\dd y\dd x\leq C \|
        \varphi\|_{L^1(B_r)},
    \end{align*}
    where $C$ depends on $d,\,s,\,p$, $r,\,R$, $\|v\|_{L^\infty(B_{R})}$ and the $\Tail$ of $v$. Therefore, we proved that 
        \[
        \left|\int_{\R^d} (-\Delta_p )^s v(x) \varphi(x)\dd x\right|\leq C\|\varphi\|_{W^{\alpha,1}}.
        \]
        Since $\varphi \in W_0^{\alpha,1}(B_r)$ was arbitrary, we conclude that $(-\Delta_p )^s v$ belongs to the dual space of $W_0^{\alpha,1}(B_r)$. Recall that the fractional Sobolev space $W^{\alpha,1}(B_r)$ coincides with the Besov space $B^{\alpha}_{1,1}(B_r)$, whose dual consist of distributions that are locally in $B^{-\alpha}_{\infty,\infty}$ (see \cite{triebel1983book,runst1996book}).
\end{proof}
A minor modification of the argument of the previous lemma also gives the case $p=1/(1-s)$. 
\begin{lemma} \label{l:frac_plap_of_Lipschitz_special}
Let $p =1/(1-s)$ and $v$ be a function in $L^{p-1}_{sp}(\R^d)$ that is Lipschitz in $B_R$.  Then $(-\Delta_p)^s v$ is in any local Besov space $B^{-\alpha}_{\infty,\infty}(B_r)$ for any $r<R$ and any $\alpha>0$.
\end{lemma}
\begin{proof}
The proof of this result follows the lines of the one of Lemma \ref{l:frac_plap_of_Lipschitz}. Except when dealing with the local term we make use of the fact that $v\in C^{0,\beta}$ for $\beta=(sp-\alpha)/(p-1)<1$.
\end{proof}
The case $p>1/(1-s)$ is the simplest case in this analysis. 
\begin{lemma} \label{l:bounded_frac_plap_of_Lipschitz}
Let $1/(1-s) < p < 2/(1-s)$ and $v$ be a function in $L^{p-1}_{sp}(\R^d)$ that is Lipschitz in $B_R$. Then $(-\Delta_p)^s v$ is bounded in $B_r$ for any $r<R$.
\end{lemma}

\begin{proof}
Let $x \in B_{r}$ and divide the integral into two parts: within the ball $B_{R-r}(x)$, where Lipschitz continuity applies, and outside such ball, where we make use of the tail bound, obtaining
 \begin{align*}
     \left|(-\Delta_p)^s v(x) \right| &= \left| \int_{B_{R-r}(x)} \frac{|v(x)-v(y)|^{p-2}}{|x-y|^{d+sp}}(v(x)-v(y)) \dd y  + \int_{B_{R-r}^c(x)} \frac{|v(x)-v(y)|^{p-2}}{|x-y|^{d+sp}}(v(x)-v(y)) \dd y \right| \\
     &\leq C[v]_{Lip(B_R)}^{p-1}(R-r)^{p(1-s)-1} +C \left( {\Tail}_{p-1,sp}(v;x,R-r)^{p-1}+(R-r)^{-sp} \|v \|_{L^{\infty}(B_{R})}^{p-1}\right).
 \end{align*}
Using Lemma \ref{l:move_center_tail} we can conclude that $|(-\Delta_p)^s v(x)| \leq C$, where $C$ is a universal constant depending on $R, \, r, \, \|v\|_{Lip(B_R)},\,{\Tail}_{p-1,sp}(v,R),\,d,\, s$ and $p$. The result follows since $x \in B_{r}$ is arbitrary.
\end{proof}

\subsection{Higher regularity for \texorpdfstring{$p \in [2,1/(1-s))$}{p small}}

Making use of the regularity of $(-\Delta_p)^s u_\eps$ in Besov spaces, we obtain H\"older differentiability for $u_\eps$.

\begin{lemma}\label{l:apriori_Coa}
    Let $u_\eps$ solve \eqref{eq:main_apriori}. Then, for each small $\eps$, we have
    \begin{itemize}
     \item $u_\eps\in C^{1,p(1-s)}_{loc}(B_R)$ if $p \in [2,1/(1-s))$.
    \item $u_\eps\in C^{1,\alpha}_{loc}(B_R)$ for all $\alpha \in (0,1)$ if $p = 1/(1-s)$.
    \end{itemize}
\end{lemma}

\begin{proof}
    We split the proof into two cases:
    
\textbf{Case 1: $p\in [2, \frac{1}{1-s})$}.  Since $u_\eps$ is Lipschitz, from Lemma \ref{l:frac_plap_of_Lipschitz}, $(-\Delta_p)^s u_\eps$ is locally in $B^{-\alpha}_{\infty,\infty}$ for $\alpha = 1-(1-s)p$. From  equation \eqref{eq:main_apriori}, we get that $\Delta u_\eps$ belongs to this space and therefore by \cite[Section 3.4.3]{runst1996book} $u_\eps \in B^{1+p(1-s)}_{\infty,\infty}(B_r)$ for $r <R$, that corresponds to $ u_\eps \in C^{1,p(1-s)}_{loc}(B_R)$.

\textbf{Case 2: $p=\frac{1}{1-s}$}. 
 From the Lipschitz regularity of $u_\eps$, by Lemma \ref{l:frac_plap_of_Lipschitz_special}, we deduce that $(-\Delta_p)^s u_\eps$ is in any local Besov space $B^{-\alpha}_{\infty,\infty}$ for $\alpha>0$. Since $u_\eps$ solves \eqref{eq:main_apriori}, also $\Delta u_\eps$ belongs to these spaces which implies that $u_\eps \in  B^{2-\alpha}_{\infty,\infty}(B_r)$ for $r <R$ that is $ u_\eps \in C^{1,1-\alpha}_{loc}(B_R)$ for all $\alpha \in (0,1)$.
\end{proof}

We can obtain higher regularity of $u_\eps$ in the range $2\leq p \leq1/(1-s)$ analyzing the equations satisfied by incremental quotients.
For $\beta \in (0,1)$ and $e \in S^{d-1}$, let us define the incremental quotient of order $\beta$ of a function $v$ to be
\[ v_h(x) = \frac{\delta_h v(x)}{|h|^\beta},\]
where $\delta_h v(x)=v(x+he)-v(x)$.

We start with a generic lemma about $C^{1,\alpha}$ functions and their incremental quotient. 
\begin{lemma} \label{l:increment_quot}
Let $v \in C^{1,\alpha}(B_R)$ for some $\alpha \in (0,1)$. For any $e \in S^{d-1}$ and $h \in (0,R-r)$. Let $v_h$ be the incremental quotient of $v$ of order $\beta$, for some $\beta \in [\alpha,1]$.
Then this function is bounded in $C^{0,1+\alpha-\beta}(B_r)$ uniformly with respect to $e$ and $h$.
\end{lemma}

\begin{proof}
    Let us estimate the difference $v_h(x)-v_h(y)$ for any two points $x,y \in B_r$. We consider two cases depending on whether $|x-y| > h$ or $|x-y| \leq h$.

    In the first case, we estimate $v_h(x)-v_h(y)$ in the following way.
    \begin{align*}
        |v_h(x)-v_h(y)| &= \frac {\left| v(x+he) - v(x) - v(y+he) + v(y) \right|}{|h|^\beta} \\
        &\leq \frac {\left| he \cdot (\nabla v(x) - \nabla v(y)) \right| + C |h|^{1+\alpha}}{|h|^\beta} \\
        &\leq C \left( |h|^{1-\beta} |x-y|^\alpha + |h|^{1+\alpha-\beta} \right) \leq C |x-y|^{1+\alpha-\beta}.
    \end{align*}

    In the second case, we estimate $v_h(x)-v_h(y)$ differently, using that $v_h(x) = v_h(y) + \nabla v_h(y) \cdot(x-y) + O(|x-y|^{1+\alpha})$ and $v_h(x+he) = v_h(y+he) + \nabla v_h(y+he) \cdot(x-y) + O(|x-y|^{1+\alpha})$.
    \begin{align*}
        |v_h(x)-v_h(y)| &= \frac {\left| v(x+he) - v(x) - v(y+he) + v(y) \right|}{|h|^\beta} \\
        &\leq \frac {\left| (x-y) \cdot (\nabla v(y+he) - \nabla v(y)) \right| + C |x-y|^{1+\alpha}}{|h|^\beta} \\
        &\leq C \left( |x-y| |h|^{\alpha-\beta} + |x-y|^{1+\alpha} |h|^{-\beta} \right) \leq C |x-y|^{1+\alpha-\beta}.
    \end{align*}
\end{proof}

\begin{lemma}
Let $2 \leq p\leq 1/(1-s)$  and assume $u_\eps$ is a solution of \eqref{eq:main_apriori} then $u_{\eps} \in C_{loc}^{2,\gamma}(B_R)$ for all $\gamma < p(1-s)$.
\end{lemma}

\begin{proof}
We first deal with the case $p< 1/(1-s)$. From Lemma \ref{l:apriori_Coa} we know $u_\eps \in C_{loc}^{1,p(1-s)}(B_R)$. Take $\alpha=p(1-s)$ and suppose that $1/\alpha$ is not an integer, otherwise take $\alpha$ slightly smaller. 
 
 Take $\beta \in \big(\alpha, \min\{2\alpha,1\}\big)$,
 and for $r<R$ and $e \in S^{d-1}$, let $v_h$ be the incremental quotient of order $\beta$ with $|h| < (R-r)/4$.
 The equation solved by $v_h$ is
\begin{align}  \label{eq:v_h_incre_quot}
-\eps \Delta v_h + \frac{1}{|h|^\beta} \left( (-\Delta_p)^s u_\eps(x+he) - (-\Delta_p)^s u_\eps(x) \right) = 0 \qquad \text{for }x \in B_{r}. 
\end{align} 
To estimate the nonlocal term, we split the integral in $B_{(R+r)/2}$ and outside $B_{(R+r)/2}$. 
The contribution inside $B_{(R+r)/2}$ can be written using the notation introduced in \eqref{eq:notation_Integral} and Lemma \ref{l:appendix}  as
\begin{align} \label{eq:main_term_incre}
    \frac{1}{|h|^{\beta}}\left( \mathcal{L}[B_{\frac{R+r}{2}}] u(x+he) -\mathcal{L}[B_{{\frac{R+r}{2}}}] u(x)\right) =\int_{B_{\frac{R+r}{2}}} K(x,y)\left( v_h(x)- v_h(y)\right) \dd y,
\end{align}
 where
\begin{align*}
     K(x,y) &= (p-1) \int_0^1 \frac{|u_\eps(x)-u_\eps(y)+t(\delta_h u(x)-\delta_h u(y))|^{p-2}}{|x-y|^{d+sp}} \dd t\\
     &\leq C \frac{1}{|x-y|^{d+2-\alpha}}.
\end{align*}
The inequality is a consequence of the Lipschitz continuity of $u_\eps$.
By Lemma \ref{l:increment_quot} we have that $v_h \in C_{loc}^{0,1+\alpha-\beta}(B_R)$. For $\gamma=1-2\alpha + \beta$ we test \eqref{eq:main_term_incre} with $\varphi \in W_0^{\gamma,1}(B_{r})$ getting
\begin{align*}
    &\left| \iint_{B_{\frac{R+r}{2}}\times B_{\frac{R+r}{2}}}  K(x,y)\left( v_h(x)- v_h(y)\right)  \varphi(x) \dd y \, \dd x\right|  \\&=  \left| \frac{1}{2} \iint_{B_{\frac{R+r}{2}}\times B_{\frac{R+r}{2}}}K(x,y)\left( v_h(x)- v_h(y)\right)\left( \varphi(x)-  \varphi(y) \right) \dd y \dd x\right|  \\
    &\leq C \iint_{B_{\frac{R+r}{2}}\times B_{\frac{R+r}{2}}} \frac{(\varphi(x)-\varphi(y))}{|x-y|^{d+1-2\alpha+\beta}} \dd y \dd x = C[\varphi]_{W^{\gamma,1}(B_{r})} .
\end{align*}
The contribution of the part outside $B_{(R+r)/2}$ can be estimated using Lemma \ref{l:est_tail_difference}. In this case, we obtain
\[
\frac{1}{|h|^{\beta}}\left| \mathcal{L}[B_{\frac{R+r}{2}}^c] u(x+he) -\mathcal{L}[B_{\frac{R+r}{2}}^c] u(x)\right| \leq C |h|^{1-\beta}.
\]
Hence, testing with $\varphi \in W_0^{\gamma,1}(B_{r})$ gives 
\begin{align*}
    \int_{B_{r}}\frac{1}{|h|^{\beta}}\left| \mathcal{L}[B_{\frac{R+r}{2}}^c] u(x+he) -\mathcal{L}[B_{\frac{R+r}{2}}^c] u(x)\right| |\varphi(x)| \dd x \leq C |h|^{1-\beta} \| \varphi \|_{L^{1}(B_{r})}.
\end{align*}
Collecting the estimates above gives
\[
\frac{1}{|h|^\beta} \left( (-\Delta_p)^s u_\eps(\,\cdot+he) - (-\Delta_p)^s u_\eps(\,\cdot\,) \right)  \in B_{\infty,\infty}^{-\gamma}(B_r),
\]
which implies by \eqref{eq:v_h_incre_quot} that
$\Delta v_h \in B_{\infty,\infty}^{-\gamma}(B_{r})$. This gives $v_h \in B_{\infty,\infty}^{2-\gamma}(B_\rho)=C^{1,2\alpha-\beta}(B_\rho)$ for $\rho <r$, using \cite[Section 3.4.3]{runst1996book}.
Since $\rho<r<R$ and $e \in S^{d-1}$ were arbitrary, this in terms of $u_\eps$, becomes $u_\eps \in C^{1,2\alpha}_{loc}(B_R)$, see for instance \cite[Lemma 5.6]{CaffarelliCabre}.

We now repeat the argument above $n$ times, where $n\alpha<1$ and $(n+1)\alpha>1$, replacing $\alpha$ by $n\alpha $ and taking $\beta \in (n\alpha, 1)$, we can apply the second alternative in \cite[Lemma 5.6]{CaffarelliCabre} and reach $u \in C_{loc}^{1,1}(B_R)$. 

To conclude, we repeat the argument once more, taking $\gamma\in(0,p(1-s))$ arbitrarily and choosing $\alpha=1+\gamma-p(1-s)$ and $\beta=1$. We get $C_{loc}^{2,\gamma}(B_R)$.

For the case $p=1/(1-s)$, observe that we know from Lemma \ref{l:apriori_Coa}  that $u \in C_{loc}^{1,\alpha}(B_R)$ for every $\alpha \in (0,1)$. Thus, we can repeat directly the last step above and get $u_\eps \in C^{2,\gamma}_{loc}(B_R)$ for all $\gamma \in (0,1)$.

\end{proof}

\subsection{Higher regularity in the range \texorpdfstring{$p \in (1/(1-s),2/(1-s))$}{larger p}}

In this case, we use a bootstrapping argument combined with Schauder estimates to improve the regularity of $u_\eps$.

\begin{lemma}\label{l:apriori_Coa2}
    Let $u_\eps$ solve \eqref{eq:main_apriori}. Then, for each small $\eps$, we have  $u_\eps\in C^{1,\alpha}_{loc}(B_R)$ for all $\alpha \in (0,1)$.
\end{lemma}

\begin{proof}

    The Lipschitz regularity of $u_\varepsilon$ obtained in Lemma \ref{l:Lips_apriori} implies via Lemma \ref{l:bounded_frac_plap_of_Lipschitz} that $(-\Delta_p)^su_\eps\in L^\infty(B_r)$ for every $r<R$. Then we use interior regularity of solutions to
    \[
    \varepsilon\Delta u_\varepsilon\in L^\infty(B_r),
    \]
    to conclude that $u_\varepsilon\in C^{1,\alpha}_{loc}(B_R)$.
\end{proof}
We improve the previous result by showing that if $v \in C^{1,\alpha}$ then $(-\Delta_p)^s v$ is a Lipschitz function.

\begin{lemma}
 \label{l:frac_plap_of_C1a}
 Let $v$ be a function in $L^{p-1}_{sp}(\R^d) \cap C^{1,\alpha}(B_R)$ with $\alpha > 2 - p(1-s)$. Then $(-\Delta_p)^s v$ is locally a Lipschitz function in $B_R$.
\end{lemma}

\begin{proof}
    Write
    \begin{align*}
        &(-\Delta_p)^s v(x)=\int_{B_r}\frac{|v(x)-v(x+h)|^{p-2}}{|h|^{d+sp}}(v(x)-v(x+h))\dd h+\int_{B_r^c}\frac{|v(x)-v(x+h)|^{p-2}}{|h|^{d+sp}}(v(x)-v(x+h))\dd h\\
        &\qquad =:I_1(x)+I_2(x),
    \end{align*}
    where $B_r(x)\subset B_R$.
    We prove that $I_1(\cdot)$ is Lipschitz continuous by differentiating it and proving that it is bounded. In fact,
\begin{align*}
    e\cdot \nabla I_1(x)=&\int_{B_r}(p-1)|v(x)-v(x+h)|^{p-2}(e \cdot \nabla v(x)-e \cdot \nabla v(x+h))\frac{\dd h}{|h|^{d+sp}}\\
    \leq &\, \|v\|_{C^{1,\alpha}(B_R)} (p-1)\int_0^r t^{\alpha+p-2-sp-1} \dd t\leq C
\end{align*}
where  $C>0$ depends only on $s,p$, $\alpha$ and $\|v\|_{C^{1,\alpha}(B_R)}$.

We now prove that $I_2(\cdot)$ is Lipschitz continuous directly. Taking any $x,z \in B_{R}$, we have
\begin{align*}
   \left| I_2(x) - I_2(z)\right| = \left| \int_{B_r^c}\frac{|v(x)-v(x+h)|^{p-2}}{|h|^{d+sp}}(v(x)-v(x+h)) -  \frac{|v(z)-v(z+h)|^{p-2}}{|h|^{d+sp}}(v(z)-v(z+h))\dd h \right|.
\end{align*}
From Lemma \ref{l:est_tail_difference} we deduce that 
\[
\left| I_2(x) - I_2(z)\right|  \leq C |x-z|,
\]
where $C$ depends on $r, \, R, \Tail_{p-1,sp}(v;R), \, \|v\|_{Lip(B_R)}, \, d, \,s $ and $p$.
\end{proof}

We conclude with the following regularity result.
\begin{lemma} \label{l:eps-further-regularity}
Let  $u_\eps$ solve \eqref{eq:main_apriori}. Then, for any $\eps>0$, we have $u_\eps \in C^{2,\alpha}_{loc}(B_R)$ for all $\alpha \in (0,1)$.
\end{lemma}

\begin{proof}
Combining Lemmas \ref{l:apriori_Coa2} and \ref{l:frac_plap_of_C1a}, we get that $(-\Delta_p)^s u_\eps$ is locally Lipschitz. From the equation \eqref{eq:main_apriori}, we deduce that $\Delta u_\eps$ is locally Lipschitz, and we conclude.
\end{proof}

\section{Computing the linearized operator pointwise}\label{s:appendix_2}

In Lemma \ref{l:LL-well-defined-Lipschitz}, we proved that the linearized operator $\LL_u v(x)$ is well-defined in the classical poinwise sense when $u$ and $v$ are locally Lipschitz functions and $p \in (1/(1-s),2/(1-s))$. It is natural to wonder if such a result can hold for $p \in (2,1/(1-s)]$. We will show that the integral \eqref{eq:linearized} is well-defined at $x$ provided that $u \in C^{1,1}(x)$, $v \in C^{1,\alpha}(x)$ with $\alpha > 1-p(1-s)$, and $\nabla u(x) \neq 0$.

At the end of this appendix, we show some examples to explain why the hypothesis $\nabla u(x) \neq 0$ is necessary.

\begin{lemma} \label{l:LL-well-defined}
Let $p \in [2,1/(1-s))$. Let $u, v \in L^{p-1}_{sp}(\R^d)$ be two functions such that $u$ is $C^{1,1}$ at the point $x$ and $v$ is $C^{1,\alpha}$ at the point $x$ for some $\alpha > 1-p(1-s)$. If $p \geq 3/(2-s)$, we assume further that $\nabla u(x) \neq 0$. Then, the integral in \eqref{eq:linearized} is well-defined at $x$, at least in the principal value sense.
\end{lemma}

\begin{proof}
The tail of the integral is bounded in the same way as in Lemma \ref{l:LL-well-defined-Lipschitz}.

We now restrict our attention to the integration domain $y \in B_1(x)$.

The fact that $u$ is $C^{1,1}$ at $x$ means that $u$ is differentiable at $x$ and there is a constant $C_u$ such that
\[ \left| u(y) - u(x) - (y-x) \cdot \nabla u(x) \right| \leq C_u |x-y|^2.\]

The fact that $v$ is $C^{1,\alpha}$ at $x$ means that $v$ is differentiable at $x$ and there is a constant $C_v$ such that
\[ \left| v(y) - v(x) - (y-x) \cdot \nabla v(x) \right| \leq C_v |x-y|^{1+\alpha}.\]

We split the domain of integration. Let $D \subset B_1(x)$ be given by
\[ D := \left\{y \in B_1(x) : C_u|x-y|^2 \leq \frac 1 {10} |(x-y) \cdot \nabla u(x)|\right\}. \]

If $y \in B_1(x) \setminus D$, then we have $|u(y)-u(x)| \leq 11 C_u|x-y|^2$. Therefore
\begin{align*}
\int_{B_1(x) \setminus D} \left| \frac{|u(x)-u(y)|^{p-2}}{|x-y|^{d+sp}} (v(x)-v(y)) \right| \dd y &\leq (11 \, C_u)^{p-2} [v]_{Lip(x)} \int_{B_1(x) \setminus D} |x-y|^{2p-3-d-sp} \dd y.
\end{align*}
If the exponent $2p-3-d-sp$ is larger than $-d$, then the integral on the right-hand side is clearly bounded. This is the case in the first alternative when $p > 3/(2-s)$. Otherwise, we must analyze the shape of the set $B_1 \setminus D$ more closely using that $\nabla u(x) \neq 0$. We use polar coordinates $y-x = r\sigma$, for $r>0$ and $\sigma \in S^{d-1}$. We see that the set $B_1 \setminus D$ is described by
\[ B_1 \setminus D = \left\{ (r,\sigma) : r \in [0,1] \text{ and } C_u r \geq \frac 1{10} |\sigma \cdot \nabla u(x)| \right\}.\]
Using that $\nabla u(x) \neq 0$, for each value of $r$, the values of $\sigma \in S^{d-1}$ for which $(r,\sigma)$ is not in $D$ are in a band on the sphere of width proportional to $r$. Therefore, writing the integral in polar coordinates,
\begin{align*}
\int_{B_1(x) \setminus D} \left| \frac{|u(x)-u(y)|^{p-2}}{|x-y|^{d+sp}} (v(x)-v(y)) \right| \dd y &\leq C \int_0^1 |\{\sigma : (r,\sigma) \notin D\}| r^{2p-4-sp} \dd r \\
&\leq C \int_0^1 r^{2p-3-sp} \dd r < +\infty.
\end{align*}
The last quantity is finite since $2p - 3 - sp = (p-2) + p(1-s) - 1 > -1$.

We move on to the case $y \in D$. We observe that
\begin{equation} \label{eq:a2}
\begin{aligned} 
\left| |u(x) - u(y)|^{p-2} - |\nabla u(x)\cdot (x-y)|^{p-2} \right| &\leq C |\nabla u(x)\cdot (x-y)|^{p-3} \left| u(y) - u(x) - (y-x) \cdot \nabla u(x) \right| \\
&\leq C |y-x|^{2(p-2)}.
\end{aligned}
\end{equation}
This follows from the elementary inequality
\[ \left| |a+b|^{p-2} - |a|^{p-2} \right| \leq C \, |a|^{p-3} \, |b|,\]
which is true as long as $a, b \in \R^d$ and $|b| < c|a|$ for some $c < 1$. The constant $C$ depends on $c$.

Therefore, if $2p- 3 - sp > 0$,
\begin{align*}
\int_{D} &\left| \frac{|u(x)-u(y)|^{p-2}}{|x-y|^{d+sp}} (v(x)-v(y)) -  \frac{|\nabla u(x) \cdot (x-y)|^{p-2}}{|x-y|^{d+sp}} (v(x)-v(y)) \right| \dd y \\
&\qquad \leq \int_D C |x-y|^{2p-3-d-sp} < +\infty.
\end{align*}

Again, when the condition $2p- 3 - sp > 0$ does not hold we need a more careful analysis using $\nabla u(x) \neq 0$. We keep subdividing the domain $D$. We write $D = D_0 \cup D_1 \cup D_2 \cup \dots$, where
\[ D_k := \left\{y \in B_1(x) : 10 \cdot 2^k \, C_u|x-y|^2 \leq |(x-y) \cdot \nabla u(x)| < 10 \cdot 2^{k+1} \, C_u|x-y|^2\right\}. \]

As before, we observe that in polar coordinates, the set of $\sigma \in S^{d-1}$ such that $(r,\sigma) \in D_k$ is contained in a band of width $\approx 2^k r$ (here, we are using that $\nabla u(x) \neq 0$ again). Moreover, it is empty if $r > C 2^{-k}$. When $y \in D_k$, we have the following more precise estimate from \eqref{eq:a2}
\[ \left| |u(x) - u(y)|^{p-2} - |\nabla u(x)\cdot (x-y)|^{p-2} \right| \leq C 2^{k(p-3)} r^{2(p-2)}. \]

Therefore
\begin{align*}
\int_{D_k} &\left| \frac{|u(x)-u(y)|^{p-2}}{|x-y|^{d+sp}} (v(x)-v(y)) -  \frac{|\nabla u(x) \cdot (x-y)|^{p-2}}{|x-y|^{d+sp}} (v(x)-v(y)) \right| \dd y \\
&\leq C \int_0^{C2^{-k}} (2^k r) \cdot \left( 2^{k(p-3)} r^{2(p-2)} \right) \cdot r^{1-d-sp} \cdot r^{d-1} \dd r \\
&\leq C 2^{-k(p-2)} \int_0^{C2^{-k}} r^{2p-3-sp} \dd r \\
&= C 2^{-k p(1-s)}.
\end{align*}
Adding up over all $D_k$, we deduce that the integral over $D$ is finite.

We are left to analyze the term
\[ \int_D \frac{|\nabla u(x) \cdot (x-y)|^{p-2}}{|x-y|^{d+sp}} (v(x)-v(y)) \dd y. \]

We write this integral as $I_1 + I_2$, where
\begin{align*}
I_1 &= \int_D \frac{|\nabla u(x) \cdot (x-y)|^{p-2}}{|x-y|^{d+sp}} (v(x)-v(y) - (x-y) \cdot \nabla v(x)) \dd y \\
I_2 &= \int_D \frac{|\nabla u(x) \cdot (x-y)|^{p-2}}{|x-y|^{d+sp}} ((x-y) \cdot \nabla v(x)) \dd y.
\end{align*}
The expression for $I_2$ is not integrable, but it is odd. Thus, it equals zero when evaluated in the principal value sense.

The integral $I_1$ is integrable in the classical sense. Indeed, since $v$ is $C^{1,\alpha}$ at $x$, it follows that
\[
|v(x) - v(y) - \nabla v(x)\cdot (x-y)| \leq C_v |x-y|^{1+\alpha}.
\]
 This makes the integrand $\lesssim |x-y|^{p(1-s)-1+\alpha-d}$, which has an exponent greater than $-d$.

Adding up all the terms, we conclude that the integral in the definition \eqref{eq:linearized} is well-defined in the principal value sense.
\end{proof}

We now show an example of a pair of functions $u$ and $v$ in 1D, with arbitrarily high regularity, for which $\LL_u v(0)$ is not well-defined.

Let $u(x) = x_+^k$, for some arbitrarily large integer power $k$ for $|x|<1$, and continued smoothly for $|x|>1$ so that $u$ is compactly supported. This function $u$ is in $C^{k-1,1}$. Let $v(x) = x$.

Assume that $s > 1/2$ and $p$ is only barely larger than $2$, so that $k(p-2)-sp < -1$. We claim that $\LL_u v(0)$ is not well-defined in the sense that its integral representation given by \eqref{eq:linearized} does not make sense at $x = 0$.

Indeed since $u(x)=0$ for $x\leq 0$, we get
\[ \LL_u v(0) = \int_0^\infty \frac{u(y)^{p-2} v(y)}{|y|^{1+sp}} \dd y \geq \int_0^1 y^{k(p-2)-sp} \dd y = +\infty\]

The case of the classical (local) $p$-Laplacian is not better. In this case 
\[ \Delta_p u(x) = \nabla \cdot ( |\nabla u|^{p-2} \nabla u). \]
The linearized operator $L_u v(x)$ is given by
\[ L_u v(x) = (p-1) \nabla \cdot ( |\nabla u|^{p-2} \nabla v). \]
If we set $v(x) = e \cdot x$ for some unit vector $e$, and $u(x) = |x|^2$, then we have
\[ L_u v(x) = (p-1) e \cdot \nabla (|x|^{2(p-2)}).\]
Thus, $L_u v(0)$ does not exist classically if $2(p-2) \leq 1$.

This fact is related to the difficulty found in \cite{korvenpaa2019equivalence} when defining viscosity solutions in the singular case. There, they observe that the operator is not well-defined for smooth functions at points where their gradients vanish, and they are forced to take test functions with certain growth. Even though in this paper we are working in the degenerate case $p\geq 2$, when we linearize the equation we effectively lose one degree of homogeneity and this problem reappears.

\section*{Acknowledgements}

This research is partially supported by PRIN 2022 7HX33Z - CUP J53D23003610006 and by University of Bologna funds for the project “Attività di collaborazione con università del Nord America” within the framework of “Interplaying problems in analysis and geometry”. These funds supported the visit to the Department of Mathematics of the University of Chicago, where part of this project was carried out.
D.G. and D.J. are grateful to the Department of Mathematics of the University of Chicago for the warm hospitality.
D.G. is also partially supported by the INDAM-GNAMPA project 2025 “At The Edge Of Ellipticity” - CUP E5324001950001. 
L.S. is supported by NSF grants DMS-2054888 and DMS-2350263.

\bibliographystyle{plain} 
\bibliography{bib} 
\Addresses
\end{document}